\documentclass[a4paper, 11pt, reqno]{amsart}

\usepackage{amssymb}
\usepackage{amsthm}
\usepackage{fullpage}
\usepackage{subfig}

\usepackage{amsmath}
\usepackage{empheq}
\usepackage{color}
\usepackage{float}
\usepackage{relsize}

\usepackage{vmargin}

\usepackage[colorlinks,citecolor=blue,filecolor=black,linkcolor=blue,urlcolor=black]{hyperref}
\usepackage[all,cmtip]{xy}

\numberwithin{equation}{section}

\newcommand{\C}{\ensuremath{\mathbb C}}

\newcommand{\R}{\ensuremath{\mathbb R}}

\theoremstyle{plain}
\newtheorem{thm}{Theorem}[section]
\newtheorem*{thm*}{Theorem}
\newtheorem{lem}[thm]{Lemma}
\newtheorem{prop}[thm]{Proposition}
\newtheorem{cor}[thm]{Corollary}
\newtheorem*{cor*}{Corollary}
\newtheorem*{prop*}{Proposition}
\newtheorem{claim}[thm]{Claim}
\newtheorem*{lemma*}{Lemma}
\newtheorem*{claim*}{Claim}

\theoremstyle{definition}
\newtheorem{defn}[thm]{Definition}

\newtheorem*{exmp*}{Example}
\newtheorem*{defn*}{Definition}
\newtheorem*{rem*}{Remark}
\newtheorem*{note*}{Note}

\author{Hemanth Saratchandran}
\address{Institut f{\"u}r Mathematik \\
Differentialgeometrie \\
Universit{\"a}t Augsburg \\
Universit{\"a}tsstra{\ss}e 14 \\
86159 Augsburg \\
Germany}
\email{hemanth.saratchandran@math.uni-augsburg.de}

\title{Essential self-adjointness of perturbed quadharmonic operators on
Riemannian manifolds
with an application to the separation problem}


\begin{document}

\maketitle

\begin{abstract}
We consider perturbed quadharmonic operators, $\Delta^4 + V$, acting
on sections of a Hermitian vector bundle over a complete Riemannian
manifold, with the potential $V$ satisfying a bound from below
by a non-positive function depending on the distance from a point.
Under a bounded geometry assumption on the Hermitian vector bundle
and the underlying Riemannian manifold, we give a sufficient condition
for the essential self-adjointness of such operators. We then
apply this to prove the separation property in $L^2$ when the 
perturbed operator acts on functions.
\end{abstract}

\tableofcontents

\section{Introduction}

The study of the essential self-adjointness of differential operators
on Euclidean space has a long history leading to many works, see
\cite{kato}, \cite{reed}. 
The generalisation of this problem to the case of Riemannian
manifolds was initiated by M. Gaffney in \cite{gaffney}.  This work solely focused on the essential self-adjointness
of the scalar Laplacian and the Hodge Laplacian. Almost two decades later, generalisations to the case of positive integer powers of
the scalar Laplacian, and the Hodge Laplacian, were proved by H. Cordes in \cite{cordes}. Subsequently, P. Chernoff in \cite{chernoff} 
studied the essential 
self-adjointness of positive integer powers of first order differential
operators, using methods from hyperbolic pde theory. The previous works of Gaffney and Cordes occur as special cases of Chernoff's work.

After the works of these authors, a surge of activity increased in the study of the essential self adjointness of general differential operators on Riemannian manifolds. One special class of such operators,
that were singled out due to their importance in applications in mathematical physics, were second order Schr{\"o}dinger operators. 
There are now
various sufficient conditions for the essential self-adjointness of
second order Schr{\"o}dinger operators, see \cite{cycon}, \cite{kato},
\cite{reed}.

In the past
few decades, there has also been an interest regarding the question of
essential self-adjointness of higher order Schr{\"o}dinger operators.
A particular piece of work, that is worthy of mention in this context,
is the paper \cite{nguyen} by X. D. Nguyen. In this paper, Nguyen considers
2m-th order operators on $\R^n$ of the form
\begin{equation*}
Tu = \sum_{0 < \vert \alpha \vert, \vert\beta\vert \leq m}
D^{\alpha}(a_{\alpha, \beta}D^{\beta}u) + Vu
\end{equation*}
and proves the essential self-adjointness of such operators on 
$C_c^{\infty}(\R^n)$, see theorem 3.1 of \cite{nguyen}, under the following 
assumptions. $T$ is 
uniformly elliptic, $a_{\alpha, \beta}$ are bounded complex-valued 
functions with sufficient smoothness on $\R^n$, and the potential
$V \in L^{\infty}_{loc}(\R^n)$ is real-valued and satisfies 
the bound $V(x) \geq - q(\vert x\vert)$, where 
$q : [0, \infty) \rightarrow [0, \infty)$ is a non-decreasing 
function such that $q(x) = O(x^{2m/(2m-1)})$ as $x \rightarrow \infty$.
Additionally, by assuming the potential satisfies 
$1 \leq V \in C^m(\R^n)$, with some conditions on the derivatives
of $V$. Nguyen also proves the self-adjointness of the operator $T$ in this situation, see theorem 4.1 of \cite{nguyen}.

In the context of Riemannian manifolds, the study of the essential
self-adjointness of such higher order operators is still in its
infancy. 
An important piece of work in this direction was carried out by
O. Milatovic in \cite{milatovic}. In this paper, Milatovic considers 
perturbations of a biharmonic operator, $\Delta^2 + V$, acting on
sections of a Hermitian vector bundle $E$ over a complete Riemannian manifold $M$. His assumptions on the potential are that
$V \in L^{\infty}_{loc}(EndE)$, where $EndE$ denotes the endomorphism bundle associated to $E$, and that $V$
satisfies a bound from below
by a non-positive function depending on the distance from a point.
The key approach of Milatovic is to obtain suitable localised
derivative estimates, which he then employs to prove 
the essential self-adjointness of such operators on $C_c^{\infty}(E)$. A crucial assumption in his work, is that of the Riemannian manifold having Ricci curvature
bounded below by a certain non-positive function. 
The main reason for this assumption is that, by work 
of Bianch and Setti in \cite{bianchi}, it gives rise to a sequence of
cut-off functions satisfying suitable first and second order
derivative estimates. The existence of such functions are 
then used in a critical way to obtain the required local derivative 
estimates. \\
Milatovich then proves self-adjointness of perturbations of the form
$(\Delta_A)^2 + w$, where $\Delta_A$ denotes the magnetic Laplacian 
on a complete Riemannian manifold $M$ with Ricci curvature bounded below by a positive function,
$w \in C^2(M)$ is such that $w \geq 1$ and satisfies certain
derivative assumptions. As an application of this work, Milatovic shows 
that the operator $(\Delta_A)^2 + w$ is separated on $L^2(M)$, when $M$
is a complete Riemannian manifold satisfying the assumption that its
Ricci curvature is bounded below by a positive function.
The separation problem on $\R^n$ was first studied, in the context of the Laplacian on functions, in \cite{everitt} by Everitt and
Giertz. We say the expression $\Delta + V$ is separated if 
$u \in L^2(\R^n)$ and $(\Delta + V)u \in L^2(\R^n)$ imply 
$\Delta u \in L^2(\R^n)$ and $Vu \in L^2(\R^n)$. 
We should mention
that the separation problem for perturbations of the biharmonic 
operator has been studied, previous to Milatovic's work, by the authors 
of 
\cite{atia}. However, in that paper the authors assume that the 
potential $w$ satisfies certain derivative assumptions,
defined via testing it against suitable test functions $u \in 
C_c^{\infty}(M)$.

In this paper, we consider perturbations of the quadharmonic 
operator, $\Delta^4 + V$, acting
on sections of a Hermitian vector bundle $E$ over a complete Riemannian
manifold $M$. Here, $\Delta = \nabla^{\dagger}\nabla$ denotes a Bochner Laplacian associated to a Hermitian connection $\nabla$, $V$ denotes a potential satisfying the assumptions that
$V \in L^{\infty}_{loc}(EndE)$, and $V$
satisfies a bound from below
by a non-positive function depending on the distance from a point.
Assuming $M$ admits bounded geometry and $E$ admits 1-bounded geometry, see section \ref{BG} for the definition of bounded and 1-bounded geometry, we prove that such operators are essentially  
self-adjoint on $C_c^{\infty}(E)$. The primary need to assume that our
manifold admits bounded geometry is to do with the fact that
this assumption leads to the existence of a suitable sequence
of cut-off functions that satisfy higher order derivative
estimates. Unfortunately, the sequence constructed by Bianchi and
Setti in \cite{bianchi} is not adequate for our purposes. Our proof
in this context follows in the
spirit of Milatovic's in \cite{milatovic}. We obtain several localised derivative estimates, which are then used to establish essential self-adjointness. 

We also consider the operator $(\Delta_A)^4 + w$, where in our 
case $w \in C^4(M)$, $w \geq 1$ and satisfies certain derivative assumptions. We prove the self-adjointness of such operators, and then apply this to show that such operators are separated in the
sense of Everitt and Geiretz mentioned above. 

We should mention that recently, there has been an interest in the study of the quadharmonic operator in regards to the quadharmonic map equation,
$\Delta^4 = 0$ on $\R^n$. In \cite{luo}, the authors study the quadharmonic 
Lane-Emden equation $\Delta^4u = \vert u\vert^{p-1}u$ on $\R^n$, and
are able to classify the finite Morse index solutions. An application of their work is that they are able to then use this to obtain 
a monotonicity formula for the quadharmonic maps equation.

The reader may wonder, why these techniques of 
obtaining localised derivative estimates, for proving essential 
self-adjointness, cannot be made to work for higher order perturbations,
$\Delta^{2n} + V$. The key issue is that, in obtaining such 
derivative estimates for quadharmonic perturbations, one needs to resort 
to certain commutation formulae for connections, see section \ref{commutation_formulae}. The 
terms that come out of such a formula depend on 
derivatives of lower order powers of the Bochner Laplacian. As the power 
of the Bochner Laplacian grows, these terms that come out of the 
commutation formulae grow in number, and cannot be estimated as they can 
in the quadharmonic case. As of yet, at least to this author, there 
seems to be no way to by pass the use of such commutation formulae, and 
this seems to be the underlying stumbling block to making such an 
approach go through for higher order perturbations. This issue is 
explained in detail in section \ref{conclusion}.

Let us now describe the contents of the paper. Section \ref{prelims}
consists of preliminary material, where we setup the notation and
explain the assumptions of the paper. In section \ref{main_results},
we outline the main results of the paper. 
Section \ref{sec_covariant_derivatives} obtains localised covariant
derivative estimates, and
section \ref{derivative_est_powers_cut_offs} obtains 
derivative estimates for powers of cut-off functions. These two sections
are then used to obtain various localised estimates in 
section \ref{sec_laplacian_derivatives}. Section 
\ref{proof_main_theorem_1} gives the proof of our first main theorem, 
on the essential self-adjointness of perturbations of the form
$\Delta^4 + V$, where $\Delta$ is a Bochner Laplacian on a Hermitian 
vector bundle. Section \ref{localised_derivative_mag_lap} then
obtains localised derivative estimates for the magnetic Laplacian, and
section \ref{proof_main_theorem_2}
proves the self adjointness of $(\Delta_A)^4 + w$, where
$\Delta_A$ denotes the magnetic Laplacian. In section 
\ref{App_separation}, we prove a separation result for the
operator $(\Delta_A)^4 + w$. Finally, section 
\ref{conclusion} explains why these methods cannot be pushed to
prove the essential self-adjointness of higher order perturbations on
a Riemannian manifold.

\section*{Acknowledgements}
The author wishes to sincerely thank Ognjen Milatovic for several
discussions related to this work.

\section{Preliminaries}\label{prelims}

\subsection{Background and notation}\label{B_N}

Throughout this paper, $(M, g)$ will denote a smooth connected
Riemannian $n$-manifold without boundary, where $g$ denotes
the Riemannian metric on $M$. 
The canonical Levi-Civita connection
on $M$ will be denoted by $\nabla_M$, the associated Riemannian
volume form by $d\mu$, and the associated curvature tensor by $Rm$.
The Laplace Beltrami operator on functions on $M$ will be denoted by
$\Delta_M$.

We will fix a smooth Hermitian vector bundle $(E, h)$ over $M$, with
Hermitian metric $h$. We will also fix a metric connection $\nabla$
on $E$. This connection gives rise to a curvature tensor, which we will
denote by $F$. The formal adjoint of $\nabla$ will be denoted by
$\nabla^{\dagger}$, with the associated Bochner Laplacian being given
by $\Delta := \nabla^{\dagger}\nabla$.

The metric $g$ induces a metric (given by the inverse) on $T^*M$. 
Together with the metric $h$ we can then extend these metrics to the
bundles $\bigotimes_sT^*M \bigotimes_rTM \bigotimes_q E$. We will often 
denote
the norm of a section of any one of these bundles by $\vert \cdot \vert$. 
This should not cause any confusion, as the context should make it clear
which bundles our sections are mapping into.

Using the connection $\nabla_M$, we can extend the connection $\nabla$ 
on $E$ to the tensor products  
$\bigotimes_sT^*M \bigotimes_rTM \bigotimes_q E$. We 
will denote these extended connections by $\nabla$ as well, the context
making it clear as to which bundle it is acting on.

The magnetic Laplacian on functions will be denoted by $\Delta_A$. We 
remind the reader that this is constructed as follows. Let $d_A$ stand for
the magnetic differential
\begin{equation*}
d_Au := du + iuA
\end{equation*}
where $d$ is the exterior derivative, and $A \in T^*M$ is a real-valued
one-form on $M$. We then define the magnetic Laplacian by 
$\Delta_A := d^{\dagger}_Ad_A$, where $d^{\dagger}_A$ denotes the
formal adjoint of $d_A$. In the special case that $A = 0$, we recover
the Laplace-Beltrami operator $\Delta_M$.

We will use the notation $C^{\infty}(M)$, $C^{\infty}_c(M)$ to denote
the smooth functions and smooth functions with compact support on $M$
respectively. Similarly, we use the notation $C^{\infty}(E)$ and
$C^{\infty}_c(E)$ to denote smooth sections and smooth sections
with compact support of $E$ respectively.

The notation  $L^2(E)$ will denote the Hilbert space of square integrable sections of $E$, with inner product
\begin{equation*}
\langle u, v\rangle := \int_{M}h(u, v)d\mu.
\end{equation*}
We will denote the associated $L^2$-norm by
\begin{equation*}
\vert\vert u\vert\vert := \bigg{(} 
\int_M\vert u\vert^2d\mu\bigg{)}^{1/2}
\end{equation*}
where $\vert u\vert^2 = h(u, u)$.

For local Sobolev spaces of sections in $L^2(E)$, we use the notation 
$W^{k,2}_{loc}(E)$, with $k$ indicating the highest order of 
derivatives. For $k=0$, we simply
write $L^{2}_{loc}(E)$. Our potentials will be elements in
$L^{\infty}_{loc}(End E)$. We remind the reader that this 
consists of those measurable sections of $End E$ that have
finite essential supremum almost everywhere, 
over relatively compact open sets.

We will also need the distance from a point, which we denote
by $r$. That is, fixing a point $x_0 \in M$ we let 
\begin{equation}
r(x) := d(x_0, x) \label{distance}
\end{equation}
where
$d$ is the distance function induced from the Riemannian metric
$g$ on $M$, for all $x \in M$.

Given tensors $S$ and $T$ defined on bundles over $M$, we let $S * T$ denote any multilinear form obtained
from $S \otimes T$ in a universal bilinear way. Therefore, $S * T$ is obtained by starting with
$S \otimes T$, taking any linear combination of this tensor, raising and lowering indices, taking any
number of metric contractions (i.e. traces), and switching any number of factors in the product.
We then have that 
\begin{equation*}
\vert S * T\vert \leq C\vert S\vert\vert T\vert 
\end{equation*}
where $C > 0$ is a constant that will not depend on $S$ or $T$. 
For example, given a
smooth
vector field $X$ on $M$, and a smooth section $u$ of $E$. We can write
$\nabla_Xu = X * \nabla u$. To see this, one simply observes that
$\nabla_Xu = tr(X \otimes \nabla u)$, where $tr$ denotes a trace.
In particular, we see that we have the estimate
\begin{equation}
\vert \nabla_Xu\vert \leq \vert X\vert \vert \nabla u\vert. \label{star_notation_ex}
\end{equation}

Finally, we mention that during the course of many estimates constants
will change from line to line. We will often use the practise of 
denoting these new constants by the same letter.

\subsection{Bounded Geometry}\label{BG}

In this paper we will be making two assumptions on the geometry of our
Riemannian manifolds, and the vector bundles over them.

\begin{defn}\label{bouned_geometry_mfld}
Let $(M, g)$ be a smooth non-compact Riemannian manifold. We say $(M, g)$ admits 
bounded geometry if the following conditions are satisfied.
\begin{enumerate}
\item $r_{inj} > 0$

\item $\sup_{x\in M}|\nabla^kRm(x)| \leq C_k$ for $k \geq 0$, and $C_k > 0$ a constant
\end{enumerate}
where $r_{inj}$ denotes the injectivity radius of $M$, $\nabla$ is the
Levi-Civita connection, and
$Rm$ denotes the
curvature tensor.
\end{defn}

We point out that condition (1) implies that the manifold is complete.
The reader can consult chapter 2 of \cite{eldering}, for more 
on bounded geometry.

We will also need a version of bounded geometry for vector bundles over
a manifold.

\begin{defn}\label{bounded_geometry_vb}
Let $M$ be a smooth manifold, and $(E, h, \nabla)$ a Hermitian vector bundle
over $M$, with Hermitian metric $h$ and connection $\nabla$. We say
the triple $(E, h, \nabla)$ admits $k$-bounded geometry
if the following
condition is satisfied.
\begin{enumerate}
\item $\sup_{x\in M}|\nabla^jF(x)| \leq C_j$ for $0 \leq j \leq k$, and $C_j > 0$ a constant
\end{enumerate}
where $F$ denotes the curvature tensor associated to $\nabla$.

We say the triple $(E, h, \nabla)$ admits bounded geometry if it
admits $k$-bounded geometry for all $k \geq 0$.
\end{defn}

We point out to the reader that if $(M, g)$ is a smooth Riemannian manifold
and we take $(TM, g, \nabla)$, where $\nabla$ is the Levi-Civita connection.
Then the triple $(TM, g, \nabla)$ admitting bounded geometry, in the sense
of definition \ref{bounded_geometry_vb}, is weaker than $(M, g)$ admitting
bounded geometry, in the sense of definition \ref{bouned_geometry_mfld}. 
This is because definition \ref{bouned_geometry_mfld} has the extra condition
that the injectivity radius must be positive.

We have the following proposition, about derivatives of the metric
and the Christoffel symbols in the bounded geometry setting.
For the proof, the reader may consult theorem 2.4 and corollary 2.5 in
\cite{eichhorn}.

\begin{prop}\label{christoffel_derivatives}
Let $(M, g)$ be a Riemannian manifold of bounded geometry. Then there 
exists a $\delta > 0$ such that the metric and the Christoffel symbols
are bounded in normal coordinates of radius $\delta$ around each 
$x \in M$, and the bounds are uniform in $x$.
\end{prop}

Throughout this paper,
we will always impose the following two assumptions on our 
Riemannian manifolds and the Hermitian vector bundles over them.

\begin{enumerate}
\item[\textbf{(A1)}] All Riemannian manifolds $(M, g)$ that we consider will 
admit bounded geometry.

\item[\textbf{(A2)}] All Hermitian vector bundles $(E, h, \nabla)$ with Hermitian
metric $h$, and connection $\nabla$, will be assumed to admit 
$1$-bounded
geometry.
\end{enumerate}

\subsection{Cut-off functions}\label{cut_offs}

We will be making use of generalised distance functions. For this we will
need the following result of Y. A. Kordyukov, see lemma 2.1 in 
\cite{shubin}, 
and \cite{kordyukov1}, \cite{kordyukov2}.

\begin{lem}\label{gen_distance}
Let $M$ be a smooth Riemannian manifold of bounded geometry. 
There exists a smooth function
$\widetilde{d} : M \times M \rightarrow [0, \infty)$ satisfying the following
conditions:
\begin{enumerate}
\item There exists $\rho > 0$ such that 
\begin{equation*}
|\widetilde{d}(x,y) - d(x,y)| < \rho
\end{equation*}
for every $x$, $y \in M$.

\item For every multi-index $\alpha$ with $|\alpha| > 0$ there exists 
a constant $C_{\alpha} > 0$ such that
\begin{equation*}
|\partial_y^{\alpha}\widetilde{d}(x,y)| \leq C_{\alpha}, x, y \in M
\end{equation*}
where the derivative $\partial_y^{\alpha}$ is taken with respect to normal
coordinates.
\end{enumerate}
Moreover for every $\epsilon > 0$, there exists a smooth function 
$\widetilde{d}_{\epsilon} : M \times M \rightarrow [0, \infty)$ satisfying
(1) with $\rho < \epsilon$.
\end{lem}

We will be making use of lemma \ref{gen_distance} in the following way.
We once and for all fix a point $x_0 \in M$, and let 
Let $\hat{d}_{1/\epsilon}(y) := 
\widetilde{d}_{1/\epsilon}(x_0,y)$, where
$\widetilde{d}_{1/\epsilon}$ denotes the smooth generalised distance function
given by the last part of \ref{gen_distance}. We note that that
$\hat{d}_{1/\epsilon}$ is smooth on $M$.

Let $F : \R \rightarrow \R$ be a smooth function such that
\begin{equation*}
F(x) =
\begin{cases}
1 \text{ if } x \leq 1 \\
0 \text{ if } x \geq 4 \\
\text{ monotonically decreasing on } [1,4].
\end{cases}
\end{equation*}

We define $\chi_{\epsilon}(y) = F(\epsilon \hat{d}_{1/\epsilon}(y))$. 
We then see that $\chi_{\epsilon} = 1$ on $B_{2/\epsilon}(x_0)$ and
that $Supp(\chi_{\epsilon}) \subseteq B_{4/\epsilon}(x_0)$.

From lemma \ref{gen_distance} (2) we have the estimate
\begin{equation}
|\partial_y^{\alpha}\chi_{\epsilon}(y)| \leq C_{\alpha}\epsilon 
\label{cut-off_bound_1}
\end{equation}
where $\alpha$ is a multi-index, $C_{\alpha} > 0$ is a constant, and
the derivative $\partial_y^{\alpha}$ is taken with respect to normal
coordinates.

In particular, this implies we have pointwise bounds of the form
\begin{align}
|\Delta_M^k\chi_{\epsilon}| &\leq C_k\epsilon, \text{ for } k\geq 1 
\label{cut-off_bound_2} \\
|d\chi_{\epsilon}| &\leq C_0\epsilon \label{cut-off_bound_3}
\end{align}

In this paper, we will be making heavy use of the function 
$\chi_{\epsilon}$, which we can always assume exists by our assumption
(A1) from the end of section \ref{BG}.

Next we recall some well known derivative formulas for products. In the 
following, we assume $(M, g)$ is a Riemannian manifold, and 
$(E, h, \nabla)$ is a Hermitian vector bundle over $M$, with Hermitian
metric $h$ and metric connection $\nabla$.
We assume $u \in W^{4,2}_{loc}(E)$ and $f \in C_c^{\infty}(M)$.

We have the following formula for the adjoint.
\begin{equation}
\nabla^{\dagger}(f\nabla u) = f\nabla^{\dagger}\nabla u - \nabla_{(df)^\#}u. 
\label{adjoint_derivative}
\end{equation}

We also have the following formula for the Laplacian of a product.

\begin{equation}
\Delta(fu) = f\Delta u - 
2\nabla_{(df)^{\#}}u + \Delta_M(f)u.
\label{laplace_product_formula}
\end{equation}

Iterating this formula, we obtain the formula
\begin{align}
\Delta^2(fu) &= \Delta(f\Delta u) - 
2\Delta\nabla_{(df)^{\#}}u + \Delta(\Delta_M(f)u) 
\nonumber \\
&= 
f\Delta^2u - 2\nabla_{(df)^{\#}}\Delta u +
 2\Delta_M(f)\Delta u - 2\Delta\nabla_{(df)^{\#}}u \label{bi-laplace_product}
\\
&\hspace{2cm} 
 - 2\nabla_{(d\Delta_M(f))^{\#}}u + 
 (\Delta_M^2f)u. \nonumber
\end{align}

Finally, we have the following formula for a composition
\begin{equation}
\Delta(f \circ u) = f''(u)\vert du\vert^2 + f'(v)\Delta(u) 
\label{Laplace_comp}
\end{equation}

\subsection{Commutation formulae for connections}\label{commutation_formulae}

It will often be the case that we need to switch derivatives in certain
formulas we obtain. In the following subsection, we state the lemmas we 
will be using to carry out such a procedure.

The following lemma tells us how to switch covariant derivatives, see lemma 5.12 in \cite{kelleher}.
\begin{lem}\label{connection_1}
Let $E$ be a Hermitian vector bundle over a Riemannian manifold $(M, g)$, with metric compatible connection 
$\nabla$. Let $u$ denote a 
section of $E$. We have
\begin{align*}
\nabla_{i_k}\nabla_{i_{k-1}}\cdots\nabla_{i_1}\nabla_{j_1}\nabla_{j_2}\cdots\nabla_{j_k}u &=
\nabla_{i_k}\nabla_{j_k}\nabla_{i_{k-1}}\nabla_{j_{k-1}}\cdots\nabla_{i_1}\nabla_{j_1}u \\
&\hspace{0.5cm}+
\sum_{l=0}^{2k-2}\big{(}(\nabla_M^{(l)}Rm + \nabla^{(l)}F)*\nabla^{(2k-2-l)}u\big{)}.
\end{align*}
where $F$ denotes the curvature associated to $\nabla$, and $Rm$ is the Riemannian curvature.
\end{lem}

We will also need to commute derivatives with Laplacian terms. The following lemma shows us how to
do this, see corollary 5.15 in 
\cite{kelleher}.

\begin{lem}\label{connection_2}
Let $E$ be a Hermitian vector bundle over a Riemannian manifold $(M, g)$, with metric compatible connection 
$\nabla$. Let $\Delta = \nabla^{\dagger}\nabla$ denote the Bochner
Laplacian, and let $u$ be a section of $E$. We have
\begin{equation*}
\nabla^{(n)}\Delta^{(k)}u = \Delta^{(k)}\nabla^{(n)}u + 
\sum_{j=0}^{2k+n-2}\bigg{(}(\nabla_M^{(j)}Rm + \nabla^{(j)}F)*\nabla^{(2k+n-2-j)}u\bigg{)}.
\end{equation*}
\end{lem}

We will be primarily applying the above lemma for the case $n = k = 1$. 

\begin{cor}\label{connection_3}
Let $E$ be a Hermitian vector bundle over a Riemannian manifold $(M, g)$, with metric compatible connection 
$\nabla$. Let $\Delta = \nabla^{\dagger}\nabla$ denote the Bochner
Laplacian, and let $u$ be a section of $E$. We have
\begin{enumerate}
\item $\nabla\Delta u = \Delta\nabla u + (Rm + F)*\nabla u + 
\nabla(Rm + F)*u$

\end{enumerate}
\end{cor}

\section{Main results}\label{main_results}

In this section, we state the main results of the paper.

\begin{thm}\label{main_theorem_1}
Let $(M, g)$ be a complete connected Riemannian manifold, 
and let $(E, h)$ be a 
Hermitian vector bundle over $M$ with metric connection $\nabla$.
Assume $M$ and $E$ satisfy the assumptions \textbf{(A1)} and
\textbf{(A2)}. Furthermore, assume we are given a potential
$V \in L^{\infty}_{loc}(End E)$ that is self-adjoint and such that
\begin{equation*}
V(x) \geq -q(r(x))I_x \text{ for a.e } x \in M,
\end{equation*}
where $I_x : E_x \rightarrow E_x$ is the identity endomorphism, 
$r(x)$ is as in \eqref{distance}, and 
$q : [0, \infty) \rightarrow [0, \infty)$ is a non-decreasing function
such that $q(x) = O(x)$ as $x \rightarrow \infty$.

Then the operator $T := \Delta^4 + V$, with domain $C_c^{\infty}(E)$,
is essentially self-adjoint.
\end{thm}

Our next theorem will restrict to the case of the magnetic Laplacian
acting on functions. The reader who is not familiar with the magnetic
Laplacian can see section \ref{B_N} for a brief introduction.

We define the following two domains. Let 
$\mathcal{D}_1 := \{u \in L^2(M) : (\Delta_A)^4u \in L^2(M)\}$, where
$(\Delta_A)^4u$ is defined in the sense of distributions. This is 
the maximal domain of the operator $(\Delta_A)^4$. Given
$w \in C^4(M)$, let 
$\mathcal{D}_2 = \{u \in L^2(M) : wu \in L^2(M)\}$.

\begin{thm}\label{main_theorem_2}
Let $(M, g)$ be a complete connected Riemannian manifold. Let 
$w \in C^4(M)$, such that $w \geq 1$, and let $h = w^{-1}$. 
Then there exists $0 < \sigma_0 < 1$, such that if $0 < \sigma \leq \sigma_0$
and
$h$ satisfies the following pointwise estimates
\begin{align}
\vert dh\vert &\leq \sigma h^{7/4} \label{h_assump_1}\\
\vert \Delta_M h\vert &\leq \sigma h^{3/4} \label{h_assump_2}\\
\vert \Delta_M^2h\vert &\leq \sigma h^{1/2} \label{h_assump_3}\\
\vert \Delta_M \vert dh\vert^2\vert &\leq \sigma h^{3/2} \label{h_assump_4}\\
\vert d\vert dh\vert^2\vert &\leq \sigma h^{22/8} \label{h_assump_5}\\
\vert d\Delta_M h\vert &\leq \sigma h^{3/2}. \label{h_assump_6}
\end{align}
Then the operator $(\Delta_A)^4 + w$ is self adjoint on the domain
$\mathcal{D}_1 \cap \mathcal{D}_2$.
\end{thm} 

In \cite{milatovic}, Milatovic is able to find an explicit bound for
$\sigma_0$, which comes down to solving two inequalities, see 
proof of theorem 2.2 in \cite{milatovic}. In our situation, the number of
inequalities is much more than two, and this makes it very difficult to 
find an explicit bound for $\sigma_0$.

An application of the above two theorems is the following
separation result.

\begin{cor}\label{separation_cor}
Let $(M, g)$ be a complete connected Riemannian manifold satisfying
assumption \textbf{(A1)}. Let $w\in C^4(M)$ satisfy the assumptions
of theorem \ref{main_theorem_2}. Then the operator
$(\Delta_A)^4 + w$ is separated in $L^2(M)$.
\end{cor}

\section{Localised covariant derivative estimates}\label{sec_covariant_derivatives}

This section is the first of two sections on localised derivative 
estimates needed for the proof of theorem \ref{main_theorem_1}. We obtain a localised covariant derivative estimate, and then
use this to obtain a localised estimate for two covariant derivatives
of a section. These will then be put to use in section
\ref{sec_laplacian_derivatives}, where we obtain several other estimates.

We will make use of the cut-off functions
$\chi_{\epsilon}$, whose existence was explained in section 
\ref{cut_offs}. We let $k$ denote a large fixed 
positive integer.

\begin{prop}\label{covariant_est_1}
For $u \in W^{2,2}_{loc}(E)$ and
$\epsilon > 0$ sufficiently small, we have the following estimate
\begin{equation}
\vert\vert\chi_{\epsilon}^k\nabla u\vert\vert^2 \leq 
\frac{1}{2(1-k\epsilon)}\vert\vert\chi_{\epsilon}^{k+1}\Delta u\vert\vert^2 +
\frac{1+2k\epsilon}{2(1-k\epsilon)}\vert\vert \chi_{\epsilon}^{k-1}u\vert\vert^2.
\end{equation}
\end{prop}

\begin{proof}

\begin{align*}
\vert\vert \chi_{\epsilon}^k\nabla u\vert\vert^2 &= \langle\chi_{\epsilon}^{2k}
\nabla u, \nabla u\rangle \\ 
&= \langle\nabla^{\dagger}(\chi_{\epsilon}^{2k}\nabla u), u\rangle \\
&= \langle\chi_{\epsilon}^{2k}\nabla^{\dagger}\nabla u, u\rangle - \langle\nabla_{(d\chi_{\epsilon}
^{2k})^{\#}}u, u\rangle \\
&= \langle\chi_{\epsilon}^{2k}\Delta u, u\rangle - \langle 2k\chi_{\epsilon}^{2k-1}\nabla_{(d
\chi_{\epsilon})^{\#}}u, u\rangle\\
&\leq \langle\chi_{\epsilon}^{2k}\vert\Delta u\vert, \vert u\vert\rangle + 
2k\langle\chi_{\epsilon}^{2k-1}\vert d\chi_{\epsilon}\vert\vert 
\nabla u\vert, u\rangle \\
&\leq \langle\chi_{\epsilon}^{k+1}\vert\Delta u\vert, \chi_{\epsilon}^{k-1}\vert u
\vert\rangle + 
2k\epsilon\langle\chi_{\epsilon}^k\vert\nabla u\vert, \chi_{\epsilon}
^{k-1}\vert u\vert\rangle \\
&\leq \frac{1}{2}(\vert\vert\chi_{\epsilon}^{k+1}\Delta u\vert\vert^2 + \vert
\vert\chi_{\epsilon}^{k-1}u\vert\vert^2) + 
k\epsilon(\vert\vert\chi_{\epsilon}^k\nabla u\vert\vert^2 + 
\vert\vert\chi_{\epsilon}^{k-1}u\vert\vert^2).
\end{align*}

In order to obtain the second line, we are using integration by parts
noting that $\chi_{\epsilon}^{2k}$ has compact support.
To get the fifth line we are using the fact that 
$\vert\nabla_{(d\chi_{\epsilon})^{\#}}u\vert 
\leq \vert d\chi_{\epsilon}\vert\vert \nabla u\vert$, see
\eqref{star_notation_ex}, and to get the sixth line we are using
the estimate of the cut-off function \eqref{cut-off_bound_3}.
Finally, the last line follows by applying Cauchy-Schwarz and Young's
inequality.

We then find
\begin{equation}
(1 - k\epsilon)\vert\vert\chi_{\epsilon}^k\nabla u\vert\vert^2 
\leq
\frac{1}{2}\vert\vert\chi_{\epsilon}^{k+1}\Delta u\vert\vert^2 + 
(\frac{1}{2} + k\epsilon)\vert\vert\chi_{\epsilon}^{k-1}u\vert
\vert^2
\end{equation}
which immediately implies the proposition.
\end{proof}

\begin{prop}\label{covariant_est_2}
For $u \in W^{4,2}_{loc}(E)$ and
$\epsilon > 0$ sufficiently small, we have the following estimate.

\begin{align*}
\vert\vert \chi_{\epsilon}^k\nabla^2u\vert\vert^2 
&\leq 
\bigg{(}
\frac{1}{1-2k^2C^2\epsilon^2} 
\bigg{)}
\bigg{(}\frac{1}{4(1-k\epsilon)}\bigg{)}
\vert\vert \chi_{\epsilon}^{k+1}\Delta^2u\vert\vert^2 \\
&\hspace{0.5cm} + 
\bigg{(}\frac{1}{1-2k^2C^2\epsilon^2} \bigg{)}
\bigg{(} \frac{1}{2}+ C^2+ \frac{C^2}{2}\bigg{)}
\bigg{(}\frac{1}{2(1-k\epsilon)} \bigg{)}
\vert\vert\chi_{\epsilon}^{k+1}\Delta u\vert\vert^2  \\
&\hspace{0.5cm} + 
\bigg{(} \frac{1}{1-2k^2C^2\epsilon^2}
\bigg{)}
\bigg{(}\frac{1}{4(1-(k-1)\epsilon)}
\bigg{)}
\vert\vert\chi_{\epsilon}^k\Delta u\vert\vert^2 + 
\bigg{(}
\frac{1}{1-2k^2C^2\epsilon^2}
\bigg{)}
\bigg{(}
\frac{1+2k}{4(1-k\epsilon)}
\bigg{)}
\vert\vert\chi_{\epsilon}^{k-1}\Delta u\vert\vert^2  \\
&\hspace{0.5cm} +
\bigg{(}\frac{1}{1-2k^2C^2\epsilon^2}\bigg{)}
\bigg{(}\frac{C^2}{2}\bigg{)}
\vert\vert\chi_{\epsilon}^ku\vert\vert^2 
+
\bigg{(}\frac{1}{1-2k^2C^2\epsilon^2}\bigg{)}
\bigg{(}\frac{1}{2} + C^2 + \frac{C^2}{2}\bigg{)}
\bigg{(}\frac{1+2k\epsilon}{2(1-k\epsilon)} \bigg{)}
\vert\vert\chi_{\epsilon}^{k-1}u\vert\vert^2  \\
&\hspace{0.5cm} +
\bigg{(}\frac{1}{1-2k^2C^2\epsilon^2}\bigg{)}
\bigg{(}
\frac{1+2(k-1)\epsilon}{4(1-(k-1)\epsilon)} 
\bigg{)}
\vert\vert\chi_{\epsilon}^{k-2}u\vert\vert^2 
 \end{align*}
\end{prop}

\begin{proof}
We have 
\begin{align*}
\langle\chi_{\epsilon}^{2k}\nabla^2u, \nabla^2u\rangle &= \langle\nabla^{\dagger}(\chi_{\epsilon}^{2k}\nabla^2u), \nabla u\rangle \\
&= \langle\chi_{\epsilon}^{2k}\Delta\nabla u, \nabla u\rangle 
- \langle\nabla_{(d\chi_{\epsilon}^{2k})^{\#}}\nabla u, \nabla u\rangle \\
&= \langle\chi_{\epsilon}^{2k}(\nabla\Delta u - (Rm + F)*\nabla u - \nabla(Rm + F)*u), \nabla u\rangle -
 \langle\nabla_{(d\chi_{\epsilon}^{2k})^{\#}}\nabla u, \nabla u\rangle \\
&= 
\langle\chi_{\epsilon}^{2k}\nabla\Delta u, \nabla u\rangle - \langle\chi_{\epsilon}^{2k}(Rm+F)*\nabla u, \nabla u\rangle - 
\langle\chi_{\epsilon}^{2k}\nabla (Rm+F)*u, \nabla u\rangle \\
&\hspace{2.5cm} - 
 \langle\nabla_{(d\chi_{\epsilon}^{2k})^{\#}}\nabla u, \nabla u\rangle \\
&=
\langle\chi_{\epsilon}^{2k}\nabla\Delta u, \nabla u\rangle - \langle\chi_{\epsilon}^{2k}(Rm+F)*\nabla u, \nabla u\rangle - 
\langle\chi_{\epsilon}^{2k}\nabla (Rm+F)*u, \nabla u\rangle \\
&\hspace{2.4cm} -
 \langle 2k\chi_{\epsilon}^{2k-1}\nabla_{(d\chi_{\epsilon})^{\#}}\nabla u, \nabla u\rangle 
\end{align*}
where to get the first line, we are applying integration by parts noting
that $\chi_{\epsilon}^{2k}$ has compact support. To get the
the third line, we have applied the commutation formula 
from corollary \ref{connection_3}.

This implies
\begin{align*}
\vert\vert \chi_{\epsilon}^k\nabla^2u\vert\vert^2 \leq
\vert\langle\chi_{\epsilon}^{2k}\nabla\Delta u, \nabla u\rangle \vert &+ 
\vert \langle\chi_{\epsilon}^{2k}(Rm+F)*\nabla u, \nabla u\rangle \vert + 
\vert\langle\chi_{\epsilon}^{2k}\nabla (Rm+F)*u, \nabla u\rangle  \vert \\
&+
\vert \langle2k\chi_{\epsilon}^{2k-1}\nabla_{(d\chi_{\epsilon})^{\#}}\nabla u, 
\nabla u\rangle  \vert.
\end{align*}

Applying Cauchy-Schwarz and Young's inequality, we obtain

\begin{align*}
\vert\vert \chi_{\epsilon}^k\nabla^2u\vert\vert^2 &\leq 
\frac{1}{2}\vert\vert \chi_{\epsilon}^k\nabla\Delta u\vert\vert^2 +
\frac{1}{2}\vert\vert \chi_{\epsilon}^k\nabla u\vert\vert^2 +
C^2\vert\vert\chi_{\epsilon}^k\nabla u\vert\vert^2 + 
\frac{C^2}{2}\vert\vert\chi_{\epsilon}^k\nabla u\vert\vert^2 +
\frac{C^2}{2}\vert\vert\chi_{\epsilon}^ku\vert\vert^2 \\
&\hspace{1cm}+
2k^2C^2\epsilon^2\vert\vert\chi_{\epsilon}^k\nabla^2u\vert\vert^2 +
\frac{1}{2}\vert\vert\chi_{\epsilon}^{k-1}\nabla u\vert\vert^2 \\
&=
\frac{1}{2}\vert\vert \chi_{\epsilon}^k\nabla\Delta u\vert\vert^2 +
\bigg{(}\frac{1}{2} + C^2 + \frac{C^2}{2} \bigg{)}
\vert\vert\chi_{\epsilon}^k\nabla u\vert\vert^2 +
\frac{1}{2}\vert\vert\chi_{\epsilon}^{k-1}\nabla u\vert\vert^2 +
\frac{C^2}{2}\vert\vert\chi_{\epsilon}^ku\vert\vert^2 \\
&\hspace{1cm} 
+ 2k^2C^2\epsilon^2\vert\vert\chi_{\epsilon}^k\nabla^2u\vert\vert^2
\end{align*}
where for the first inequality, we have used our bounded geometry 
assumptions (A1) and (A2), see section \ref{BG}.

We can estimate the term 
$\vert\vert\chi_{\epsilon}^k\nabla\Delta u\vert\vert^2$ using
proposition \ref{covariant_est_1}.

\begin{equation*}
\vert\vert\chi_{\epsilon}^k\nabla\Delta u\vert\vert^2 \leq 
\frac{1}{2(1-k\epsilon)}
\vert\vert\chi_{\epsilon}^{k+1}\Delta^2u\vert\vert^2 + 
\frac{1+2k}{2(1-k\epsilon)}
\vert\vert \chi_{\epsilon}^{k-1}\Delta u\vert\vert^2.
\end{equation*}

We also estimate the term 
$\vert\vert\chi_{\epsilon}^{k-1}\nabla u\vert\vert^2$ using
proposition \ref{covariant_est_1}.

\begin{equation*}
\vert\vert\chi_{\epsilon}^{k-1}\nabla u\vert\vert^2 \leq 
\frac{1}{2(1-(k-1)\epsilon)}
\vert\vert\chi^k_{\epsilon}\Delta u\vert\vert^2 + 
\frac{1+2(k-1)\epsilon}{2(1-(k-1)\epsilon)}
\vert\vert\chi_{\epsilon}^{k-2}u\vert\vert^2.
\end{equation*}

Using these two estimates, along with proposition \ref{covariant_est_1}
to estimate the term 
$\vert\vert\chi_{\epsilon}^k\nabla u\vert\vert^2$, we obtain
\begin{align*}
\vert\vert\chi_{\epsilon}^k\nabla^2u\vert\vert^2 &\leq 
\frac{1}{4(1-k\epsilon)}
\vert\vert\chi_{\epsilon}^{k+1}\Delta^2u\vert\vert^2 + 
\frac{1+2k}{4(1-k\epsilon)}
\vert\vert\chi_{\epsilon}^{k-1}\Delta u\vert\vert^2 \\
&\hspace{1cm}+
\bigg{(}\frac{1}{2} + C^2 + \frac{C^2}{2}\bigg{)}
\bigg{(}\frac{1}{2(1-k\epsilon)} \bigg{)}
\vert\vert\chi_{\epsilon}^{k+1}\Delta u\vert\vert^2 \\
&\hspace{1cm} + 
\bigg{(} \frac{1}{2} + C^2 + \frac{C^2}{2}\bigg{)}
\bigg{(}
\frac{1+2k\epsilon}{2(1-k\epsilon)} 
\bigg{)}\vert\vert\chi_{\epsilon}^{k-1}u\vert\vert^2 \\
&\hspace{1cm} + \frac{1}{4(1-(k-1)\epsilon)}
\vert\vert\chi_{\epsilon}^k\Delta u\vert\vert^2 + 
\frac{1+2(k-1)\epsilon}{4(1-(k-1)\epsilon)}
\vert\vert\chi_{\epsilon}^{k-2}u\vert\vert^2 \\
&\hspace{1cm} + \frac{C^2}{2}\vert\vert\chi_{\epsilon}^ku\vert\vert^2 
+ 2k^2C^2\epsilon^2\vert\vert\chi_{\epsilon}^k\nabla^2u\vert\vert^2
\end{align*}

which gives
\begin{align*}
(1-2k^2C^2\epsilon^2)\vert\vert\chi_{\epsilon}^k\nabla^2u\vert\vert^2 &\leq 
\frac{1}{4(1-k\epsilon)}
\vert\vert\chi_{\epsilon}^{k+1}\Delta^2u\vert\vert^2 + 
\frac{1+2k}{4(1-k\epsilon)}
\vert\vert\chi_{\epsilon}^{k-1}\Delta u\vert\vert^2 \\
&\hspace{1cm}+
\bigg{(}\frac{1}{2} + C^2 + \frac{C^2}{2}\bigg{)}
\bigg{(}\frac{1}{2(1-k\epsilon)} \bigg{)}
\vert\vert\chi_{\epsilon}^{k+1}\Delta u\vert\vert^2 \\
&\hspace{1cm} + 
\bigg{(} \frac{1}{2} + C^2 + \frac{C^2}{2}\bigg{)}
\bigg{(}
\frac{1+2k\epsilon}{2(1-k\epsilon)} 
\bigg{)}\vert\vert\chi_{\epsilon}^{k-1}u\vert\vert^2 \\
&\hspace{1cm} + \frac{1}{4(1-(k-1)\epsilon)}
\vert\vert\chi_{\epsilon}^k\Delta u\vert\vert^2 + 
\frac{1+2(k-1)\epsilon}{4(1-(k-1)\epsilon)}
\vert\vert\chi_{\epsilon}^{k-2}u\vert\vert^2 \\
&\hspace{1cm} + \frac{C^2}{2}\vert\vert\chi_{\epsilon}^ku\vert\vert^2.
\end{align*}
Choosing $\epsilon$ small enough so that $1-2k^2C^2\epsilon^2 > 0$, we
obtain
\begin{align*}
\vert\vert \chi_{\epsilon}^k\nabla^2u\vert\vert^2 
&\leq 
\bigg{(}
\frac{1}{1-2k^2C^2\epsilon^2} 
\bigg{)}
\bigg{(}\frac{1}{4(1-k\epsilon)}\bigg{)}
\vert\vert \chi_{\epsilon}^{k+1}\Delta^2u\vert\vert^2 \\
&\hspace{0.5cm} + 
\bigg{(}\frac{1}{1-2k^2C^2\epsilon^2} \bigg{)}
\bigg{(} \frac{1}{2}+ C^2+ \frac{C^2}{2}\bigg{)}
\bigg{(}\frac{1}{2(1-k\epsilon)} \bigg{)}
\vert\vert\chi_{\epsilon}^{k+1}\Delta u\vert\vert^2  \\
&\hspace{0.5cm} + 
\bigg{(} \frac{1}{1-2k^2C^2\epsilon^2}
\bigg{)}
\bigg{(}\frac{1}{4(1-(k-1)\epsilon)}
\bigg{)}
\vert\vert\chi_{\epsilon}^k\Delta u\vert\vert^2 + 
\bigg{(}
\frac{1}{1-2k^2C^2\epsilon^2}
\bigg{)}
\bigg{(}
\frac{1+2k}{4(1-k\epsilon)}
\bigg{)}
\vert\vert\chi_{\epsilon}^{k-1}\Delta u\vert\vert^2  \\
&\hspace{0.5cm} +
\bigg{(}\frac{1}{1-2k^2C^2\epsilon^2}\bigg{)}
\bigg{(}\frac{C^2}{2}\bigg{)}
\vert\vert\chi_{\epsilon}^ku\vert\vert^2 
+
\bigg{(}\frac{1}{1-2k^2C^2\epsilon^2}\bigg{)}
\bigg{(}\frac{1}{2} + C^2 + \frac{C^2}{2}\bigg{)}
\bigg{(}\frac{1+2k\epsilon}{2(1-k\epsilon)} \bigg{)}
\vert\vert\chi_{\epsilon}^{k-1}u\vert\vert^2  \\
&\hspace{0.5cm} +
\bigg{(}\frac{1}{1-2k^2C^2\epsilon^2}\bigg{)}
\bigg{(}
\frac{1+2(k-1)\epsilon}{4(1-(k-1)\epsilon)} 
\bigg{)}
\vert\vert\chi_{\epsilon}^{k-2}u\vert\vert^2 
 \end{align*}
 which proves the result.
\end{proof}

\section{Derivative estimates for powers of cut-off functions}\label{derivative_est_powers_cut_offs}

The purpose of this section is to obtain certain higher order derivative
estimates for the cut-off functions $\chi_{\epsilon}$, see
section \ref{cut_offs} for their construction. These estimates will
then be used in the next section, where we obtain several more derivative
estimates for sections of a Hermitian bundle.

$k$ will denote a large fixed positive integer.

From \eqref{Laplace_comp}, we have the following formula for the
Laplacian of the cut-off function $\chi_{\epsilon}$

\begin{align}
\Delta_M(\chi_{\epsilon}^k)&= k(k-1)\chi_{\epsilon}^{k-2}\vert d\chi_{\epsilon}\vert^2 +
k\chi_{\epsilon}^{k-1}\Delta_M\chi_{\epsilon} \label{derivative_powers} \\
&= \chi_{\epsilon}^{k-2}(k(k-1)
\vert d\chi_{\epsilon}\vert^2 + k\chi_{\epsilon}\Delta_M\chi_{\epsilon}) 
\nonumber\\
&= \chi_{\epsilon}^{k-2}
G_1(\vert d\chi_{\epsilon}\vert, \Delta_M\chi_{\epsilon}, \chi_{\epsilon}). 
\label{derivative_powers_a}
\end{align}

From properties \eqref{cut-off_bound_2} and \eqref{cut-off_bound_3}, we 
immediately obtain the following claim.
\begin{claim}\label{G_1_est}
$G_1(\vert d\chi_{\epsilon}\vert, \Delta_M\chi_{\epsilon}, \chi_{\epsilon}) \leq 
C\epsilon$ for some constant $C$, which does not depend on $\epsilon$.
\end{claim}

The above formula for $\Delta_M(\chi_{\epsilon}^k)$, and claim \ref{G_1_est},
imply the following corollary.

\begin{cor}\label{laplacian_prod_est_1}
We have the estimate $\vert\Delta_M(\chi_{\epsilon}^k)\vert 
\leq C\epsilon\chi_{\epsilon}^{k-2}$ for some constant $C > 0$. 
\end{cor}

Applying $\Delta_M$ to \eqref{derivative_powers}, and making use of formulas
\eqref{laplace_product_formula} and \eqref{Laplace_comp},
we obtain the following

\begin{align}
\Delta_M^2\chi_{\epsilon}^k &= 
k(k-1)\Delta_M(\chi_{\epsilon}^{k-2}\vert d\chi_{\epsilon}\vert^2) + 
k\Delta_M(\chi_{\epsilon}^{k-1}\Delta_M\chi_{\epsilon}) \nonumber \\
&=
k(k-1)\chi_{\epsilon}^{k-2}\Delta_M(\vert d\chi_{\epsilon}\vert^2) -
k(k-1)\nabla_{(d\chi_{\epsilon}^{k-2})^{\#}}\vert d\chi_{\epsilon}\vert^2 + 
k(k-1)\vert d\chi_{\epsilon}\vert^2\delta(\chi_{\epsilon}^{k-2}) 
\nonumber \\
&\hspace{0.5cm} +
k\chi_{\epsilon}^{k-1}\Delta_M^2\chi_{\epsilon} - 
2k\nabla_{(d\chi_{\epsilon}^{k-2})^{\#}}\Delta_M\chi_{\epsilon} +
k\Delta_M\chi_{\epsilon}\Delta_M(\chi_{\epsilon}^{k-1}) 
\nonumber \\
&=
k(k-1)\chi_{\epsilon}^{k-2}\Delta_M\vert d\chi_{\epsilon}\vert^2 - 
k(k-1)(k-2)\chi_{\epsilon}^{k-3}\nabla_{(d\chi_{\epsilon})^{\#}} 
\vert d\chi_{\epsilon}\vert^2 +
k(k-1)\vert d\chi_{\epsilon}\vert^2\chi_{\epsilon}^{k-4}
G_1(\vert d\chi_{\epsilon}\vert, \Delta_M\chi_{\epsilon}, \chi_{\epsilon}) 
\nonumber \\
&\hspace{0.5cm} +
k\chi_{\epsilon}^{k-1}\Delta_M^2\chi_{\epsilon} 
- 2k(k-1)\chi_{\epsilon}^{k-2}
\nabla_{d\chi_{\epsilon}^{\#}}\Delta_M\chi_{\epsilon} + 
\chi_{\epsilon}^{k-3}\Delta_M\chi_{\epsilon}
G_1(\vert d\chi_{\epsilon}\vert, \Delta_M\chi_{\epsilon}, \chi_{\epsilon}) 
\nonumber \\
&=
\chi_{\epsilon}^{k-4}\bigg{(}
k(k-1)\chi_{\epsilon}^{2}\Delta_M\vert d\chi_{\epsilon}\vert^2
 -
k(k-1)(k-2)\chi_{\epsilon}\nabla_{(d\chi_{\epsilon})^{\#}} 
\vert d\chi_{\epsilon}\vert^2 +
k(k-1)\vert d\chi_{\epsilon}\vert^2
G_1(\vert d\chi_{\epsilon}\vert, \Delta_M\chi_{\epsilon}, \chi_{\epsilon}) 
\nonumber \\
&\hspace{0.5cm} + 
k\chi_{\epsilon}^{3}\Delta_M^2\chi_{\epsilon} -
2k(k-1)\chi_{\epsilon}^{2}
\nabla_{d\chi_{\epsilon}^{\#}}\Delta_M\chi_{\epsilon} +
\chi_{\epsilon}\Delta_M\chi_{\epsilon}
G_1(\vert d\chi_{\epsilon}\vert, \Delta_M\chi_{\epsilon}, \chi_{\epsilon})
\bigg{)} 
\nonumber \\
&=
\chi_{\epsilon}^{k-4}
G_2(\vert d\chi_{\epsilon}\vert^2, \Delta_M\chi_{\epsilon}, \chi_{\epsilon}, 
\Delta_M\vert d\chi_{\epsilon}\vert^2, 
\nabla_{(d\chi_{\epsilon})^{\#}} 
\vert d\chi_{\epsilon}\vert^2, \Delta_M^2\chi_{\epsilon}, 
\nabla_{(d\chi_{\epsilon})^{\#}}\Delta_M\chi_{\epsilon}). 
\label{derivative_powers_b}
\end{align}

The goal now is to obtain an estimate for the quantity $G_2$.
We start with the following three lemmas.

\begin{lem}\label{G2_est_1}
We have the estimate 
$| \nabla_{(d\chi_{\epsilon})^{\#}} \vert d\chi_{\epsilon}\vert^2 | \leq 
C\epsilon$, for some constant $C$ independent of $\epsilon$. 
\end{lem}
\begin{proof}
We start by computing in normal coordinates about a point $p \in M$. Write
$d\chi_{\epsilon} = \frac{\partial{\chi_{\epsilon}}}{\partial{y^i}}dy^i$, 
we then have 
$(d\chi_{\epsilon})^{\#} = g^{ij}
\frac{\partial{\chi_{\epsilon}}}{\partial{y^i}}
\frac{\partial}{\partial{y^j}}$.
We can then compute
\begin{equation*}
\nabla_{(d\chi_{\epsilon})^{\#}}d\chi_{\epsilon} = 
\bigg{(} 
g^{ij}\frac{\partial{\chi_{\epsilon}}}{\partial{y^i}}
\frac{\partial^2{\chi_{\epsilon}}}{\partial{y^j}\partial{y^k}}
- g^{ij}\frac{\partial{\chi_{\epsilon}}}{\partial{y^i}}
\frac{\partial{\chi_{\epsilon}}}{\partial{y^l}}\Gamma^l_{jk}
\bigg{)}dy^k.
\end{equation*}
At the point $p$ we have $g^{ij} = \delta^{ij}$ and $\Gamma^l_{jk} = 0$, 
since we are in normal coordinates about $p$. Therefore, at the point $p$ 
we can write
\begin{equation*}
\nabla_{(d\chi_{\epsilon})^{\#}}d\chi_{\epsilon}(p) = 
\sum_{i}\frac{\partial{\chi_{\epsilon}}}{\partial{y^i}}
\frac{\partial^2{\chi_{\epsilon}}}{\partial{y^j}\partial{y^k}}dx^k(p).
\end{equation*}
Using the fact that $\nabla$ is a metric connection, at the point $p$, we 
have
\begin{align*}
\nabla_{(d\chi_{\epsilon})^{\#}} \vert d\chi_{\epsilon}\vert^2 &= 
2\langle \nabla_{(d\chi_{\epsilon})^{\#}}d\chi_{\epsilon}, 
d\chi_{\epsilon}\rangle \\
&=
2\sum_{i, k}
\frac{\partial{\chi_{\epsilon}}}{\partial{y^i}}
\frac{\partial{\chi_{\epsilon}}}{\partial{y^k}}
\frac{\partial^2{\chi_{\epsilon}}}{\partial{y^i}\partial{y^k}}.
\end{align*}
We note
that the quantity we are trying to bound is a tensor, therefore its value
at a point is independent of the coordinates chosen to compute it. 
The bound then follows from \eqref{cut-off_bound_1}.
\end{proof}

\begin{lem}\label{G2_est_2}
We have the estimate 
$\bigg{\vert}\nabla_{(d\chi_{\epsilon})^{\#}}\Delta_M\chi_{\epsilon}\bigg{\vert}
\leq C\epsilon$ for some constant $C > 0$ independent of $\epsilon$.
\end{lem}
\begin{proof}
We fix a point $p \in M$, and we compute in normal coordinates about $p$.
In these coordinates we have the following formula for 
$\Delta_M\chi_{\epsilon}$
\begin{equation*}
\Delta_M\chi_{\epsilon} = g^{jk}\bigg{(} 
\frac{\partial^2\chi_{\epsilon}}{\partial{x^j}\partial{x^k}} - 
\frac{\partial{\chi_{\epsilon}}}{\partial{x^i}}\Gamma^{i}_{jk}.
\bigg{)}
\end{equation*}
Writing $(d\chi_{\epsilon})^{\#} = g^{pq}
\frac{\partial{\chi_{\epsilon}}}{\partial{x^p}}
\frac{\partial}{\partial{x^q}}$. We have
\begin{equation*}
\nabla_{(d\chi_{\epsilon})^{\#}}\Delta_M\chi_{\epsilon} = 
g^{pq}\frac{\partial{\chi_{\epsilon}}}{\partial{x^p}}
\frac{\partial{g^{jk}}}{\partial{x^q}}\bigg{(}
\frac{\partial^2\chi_{\epsilon}}{\partial{x^j}\partial{x^k}} - 
\frac{\partial{\chi_{\epsilon}}}{\partial{x^i}}\Gamma^{i}_{jk}
\bigg{)} +
g^{pq}g^{jk}\frac{\partial{\chi_{\epsilon}}}{\partial{x^p}}
\bigg{(}
\frac{\partial^3\chi_{\epsilon}}
{\partial{x^q}\partial{x^j}\partial{x^k}} -
\frac{\partial^2\chi_{\epsilon}}{\partial{x^q}\partial{x^i}}\Gamma^i_{jk} -
\frac{\partial{\chi_{\epsilon}}}{\partial{x^i}}
\frac{\partial{\Gamma^{i}_{jk}}}{\partial{x^q}}
\bigg{)}.
\end{equation*}
Since we are in normal coordinates, evaluating at $p$ gives
\begin{equation*}
\nabla_{(d\chi_{\epsilon})^{\#}}\Delta_M\chi_{\epsilon}(p) = 
\sum_{p, j}\frac{\partial{\chi_{\epsilon}}}{\partial{x^p}}(p)\bigg{(}
\frac{\partial^3\chi_{\epsilon}}
{\partial{x^q}\partial{x^j}\partial{x^j}}(p) - 
\frac{\partial{\chi_{\epsilon}}}{\partial{x^i}}(p)
\frac{\partial{\Gamma^{i}_{jj}}}{\partial{x^p}}(p)
\bigg{)}.
\end{equation*}
Applying \eqref{cut-off_bound_1} and lemma \ref{christoffel_derivatives} 
we get the required estimate.
\end{proof}

\begin{lem}\label{G2_est_3}
We have the estimate
$\bigg{\vert}\Delta_M\vert d\chi_{\epsilon}\vert^2\bigg{\vert}
\leq C\epsilon$
\end{lem}

The above lemma follows from theorem 1.3 in \cite{hebey}.

Using the above three lemmas, we can now estimate the quantity $G_2$.

\begin{lem}\label{G_2_est_4}
We have the following estimate 
\begin{equation*}
\vert
G_2\vert
\leq
C\epsilon
\end{equation*}
for some constant $C$ independent of $\epsilon$.
\end{lem}

\begin{proof}
It is clear we have the estimate
\begin{align*}
\vert G_2\vert \leq
k(k-1)\bigg{\vert}\Delta_M\vert d\chi_{\epsilon}\vert^2\bigg{\vert}
 &+
k(k-1)(k-2)\bigg{\vert}\nabla_{(d\chi_{\epsilon})^{\#}} 
\vert d\chi_{\epsilon}\vert^2\bigg{\vert} +
k(k-1)\vert d\chi_{\epsilon}\vert^2
\bigg{\vert}G_1(\vert d\chi_{\epsilon}\vert, \Delta_M\chi_{\epsilon}, \chi_{\epsilon})
\bigg{\vert} \\
&+ 
k\bigg{\vert}\Delta_M^2\chi_{\epsilon}\bigg{\vert} +
2k(k-1)\bigg{\vert}
\nabla_{(d\chi_{\epsilon})^{\#}}\Delta_M\chi_{\epsilon}\bigg{\vert} +
\bigg{\vert}\Delta_M\chi_{\epsilon}\bigg{\vert}
\bigg{\vert}
G_1(\vert d\chi_{\epsilon}\vert, \Delta_M\chi_{\epsilon}, \chi_{\epsilon})
\bigg{\vert}.
\end{align*}
We then estimate each term of the sum on the right.
\begin{itemize}
\item[1.] We have the estimate $\bigg{\vert}\Delta_M\vert d\chi_{\epsilon}\vert^2\bigg{\vert}
\leq C\epsilon$. This follows from lemma \ref{G2_est_3}

\item[2.] We have
$\vert\nabla_{(d\chi_{\epsilon})^{\#}} 
\vert d\chi_{\epsilon}\vert^2\vert \leq C\epsilon$ by lemma \ref{G2_est_1}.

\item[3.] The estimate 
$\vert d\chi_{\epsilon}\vert^2
\bigg{\vert}G_1(\vert d\chi_{\epsilon}\vert, \Delta_M\chi_{\epsilon}, \chi_{\epsilon})
\bigg{\vert} \leq C\epsilon$ follows from \eqref{cut-off_bound_3} and claim
\ref{G_1_est}.

\item[4.] The estimate $\bigg{\vert}\Delta_M^2\chi_{\epsilon}\bigg{\vert} \leq C\epsilon$
follows from \eqref{cut-off_bound_2}.

\item[5.] The estimate 
$\bigg{\vert}\nabla_{(d\chi_{\epsilon})^{\#}}\Delta_M\chi_{\epsilon}\bigg{\vert}
\leq C\epsilon$ follows from lemma \ref{G2_est_2}.

\item[6.] The estimate 
$\bigg{\vert}\Delta_M\chi_{\epsilon}\bigg{\vert}
\bigg{\vert}
G_1(\vert d\chi_{\epsilon}\vert, \Delta_M\chi_{\epsilon}, \chi_{\epsilon})
\bigg{\vert} \leq C\epsilon$ follows from \eqref{cut-off_bound_2} and
claim \ref{G_1_est}.
\end{itemize}
Putting these estimates together gives the result.
\end{proof}

Using the above formula for $\Delta_M^2(\chi_{\epsilon}^k)$ and lemma
\ref{G_2_est_4}, we obtain the following corollary.

\begin{cor}\label{bilaplace_prod_est_1}
We have the estimate 
$|\Delta_M^2(\chi_{\epsilon}^k)\vert \leq C\epsilon\chi_{\epsilon}^{k-4}$
for some constant $C > 0$.
\end{cor}

Applying $d$ to the above formula for $\Delta_M\chi_{\epsilon}^k$
we obtain

\begin{align}
d\Delta_M\chi_{\epsilon}^k &= (k-2)\chi_{\epsilon}^{k-3}(d\chi_{\epsilon})
(k(k-1)\vert d\chi_{\epsilon}\vert^2 + k\chi_{\epsilon}
\Delta_M\chi_{\epsilon}) \nonumber\\
&\hspace{1cm}+
\chi_{\epsilon}^{k-2}(k(k-1)d\vert d\chi_{\epsilon}\vert^2 + 
k(d\chi_{\epsilon})(\Delta_M\chi_{\epsilon}) + k\chi_{\epsilon}
d\Delta_M\chi_{\epsilon}) \nonumber\\
&=
\chi_{\epsilon}^{k-3}
\bigg{(}
 k(k-1)(k-2)d\chi_{\epsilon}\vert d\chi_{\epsilon}\vert^2 + 
k(k-2)\chi_{\epsilon}(d\chi_{\epsilon})(\Delta_M\chi_{\epsilon}) \nonumber\\
&\hspace{1cm} +
k(k-1)\chi_{\epsilon}d\vert d\chi_{\epsilon}\vert^2 + 
k\chi_{\epsilon}(d\chi_{\epsilon})(\Delta_M\chi_{\epsilon}) + 
k\chi_{\epsilon}^2d\Delta_M\chi_{\epsilon}
\bigg{)} \label{dlaplace_form_1} \\
&=
\chi_{\epsilon}^{k-3}G_3(\chi_{\epsilon}, d\chi_{\epsilon}, 
\Delta_M\chi_{\epsilon}, d|d\chi_{\epsilon}|^2, d\Delta_M\chi_{\epsilon}).
\label{dlaplace_form_2}
\end{align}

In order to estimate the quantity $G_3$, we proceed as we did in our 
estimation for $G_2$. We start with two lemmas that shows how to estimate
$|d|d\chi_{\epsilon}|^2|$ and $|d\Delta_M\chi_{\epsilon}|$.

\begin{lem}\label{G3_est_1}
We have the estimate 
$|d|d\chi_{\epsilon}|^2| \leq C\epsilon$, for some constant $C > 0$ 
independent of $\epsilon$.
\end{lem}
\begin{proof}
We will work in normal coordinates about a point $p \in M$. We can write
\begin{equation*}
d|d\chi_{\epsilon}|^2 = d\bigg{(} g^{ij}
\frac{\partial{\chi_{\epsilon}}}{\partial{x^i}}
\frac{\partial{\chi_{\epsilon}}}{\partial{x^j}}
\bigg{)} = 
\bigg{(}
\frac{\partial{g^{ij}}}{\partial{x^p}}
\frac{\partial{\chi_{\epsilon}}}{\partial{x^i}}
\frac{\partial{\chi_{\epsilon}}}{\partial{x^j}} +
g^{ij}\frac{\partial^2{\chi_{\epsilon}}}{\partial{x^p}\partial{x^i}}
\frac{\partial{\chi_{\epsilon}}}{\partial{x^j}} + 
g^{ij}\frac{\partial{\chi_{\epsilon}}}{\partial{x^i}}
\frac{\partial^2{\chi_{\epsilon}}}{\partial{x^p}\partial{x^j}}
\bigg{)}dx^p.
\end{equation*}
This then implies
\begin{align*}
|d|d\chi_{\epsilon}|^2|^2 &= 
\bigg{(}
\frac{\partial{g^{ij}}}{\partial{x^p}}
\frac{\partial{\chi_{\epsilon}}}{\partial{x^i}}
\frac{\partial{\chi_{\epsilon}}}{\partial{x^j}} +
g^{ij}\frac{\partial^2{\chi_{\epsilon}}}{\partial{x^p}\partial{x^i}}
\frac{\partial{\chi_{\epsilon}}}{\partial{x^j}} + 
g^{ij}\frac{\partial{\chi_{\epsilon}}}{\partial{x^i}}
\frac{\partial^2{\chi_{\epsilon}}}{\partial{x^p}\partial{x^j}}
\bigg{)} \\
&\hspace{1cm}
\bigg{(}
\frac{\partial{g^{ij}}}{\partial{x^q}}
\frac{\partial{\chi_{\epsilon}}}{\partial{x^i}}
\frac{\partial{\chi_{\epsilon}}}{\partial{x^j}} +
g^{ij}\frac{\partial^2{\chi_{\epsilon}}}{\partial{x^q}\partial{x^i}}
\frac{\partial{\chi_{\epsilon}}}{\partial{x^j}} + 
g^{ij}\frac{\partial{\chi_{\epsilon}}}{\partial{x^i}}
\frac{\partial^2{\chi_{\epsilon}}}{\partial{x^q}\partial{x^j}}
\bigg{)}g^{pq}.
\end{align*}
Evaluating at the point $p$, remembering that we are in normal coordinates
about $p$, we obtain
\begin{equation*}
|d|d\chi_{\epsilon}|^2|^2(p) = 4\sum_{i, p}\bigg{(}
\frac{\partial^2{\chi_{\epsilon}}}{\partial{x^p}\partial{x^i}}(p) 
\frac{\partial{\chi_{\epsilon}}}{\partial{x^i}}(p)
\bigg{)}^2.
\end{equation*}
The estimate then follows from \eqref{cut-off_bound_1}.
\end{proof}

\begin{lem}\label{G3_est_2}
We have the estimate $|d\Delta_M\chi_{\epsilon}| \leq C\epsilon$ for some
constant $C > 0$ independent of $\epsilon$.
\end{lem}
\begin{proof}
We work in normal coordinates about a point $p \in M$. We remind the reader
that we have the formula
\begin{equation*}
\Delta_M\chi_{\epsilon} = g^{jk}\bigg{(} 
\frac{\partial^2\chi_{\epsilon}}{\partial{x^j}\partial{x^k}} - 
\frac{\partial{\chi_{\epsilon}}}{\partial{x^i}}\Gamma^{i}_{jk}.
\bigg{)}.
\end{equation*}
We then see that
\begin{equation*}
d\Delta_M\chi_{\epsilon} = 
\frac{\partial{g^{jk}}}{\partial{x^l}}\bigg{(}
\frac{\partial^2\chi_{\epsilon}}{\partial{x^j}\partial{x^k}} - 
\frac{\partial{\chi_{\epsilon}}}{\partial{x^i}}\Gamma^{i}_{jk}
\bigg{)}dx^l +
g^{jk}\bigg{(}
\frac{\partial^3{\chi_{\epsilon}}}{\partial{x^l}\partial{x^j}\partial{x^k}}
- 
\frac{\partial^2{\chi_{\epsilon}}}{\partial{x^l}\partial{x^i}}\Gamma^i_{jk}
- \frac{\partial{\chi_{\epsilon}}}{\partial{x^i}}
\frac{\partial{\Gamma_{jk}^i}}{\partial{x^l}}
\bigg{)}dx^l.
\end{equation*}
This implies 
\begin{align*}
|d\Delta_M\chi_{\epsilon}|^2 &= \bigg{[}
\frac{\partial{g^{jk}}}{\partial{x^l}}\bigg{(}
\frac{\partial^2\chi_{\epsilon}}{\partial{x^j}\partial{x^k}} - 
\frac{\partial{\chi_{\epsilon}}}{\partial{x^i}}\Gamma^{i}_{jk}
\bigg{)} +
g^{jk}\bigg{(}
\frac{\partial^3{\chi_{\epsilon}}}{\partial{x^l}\partial{x^j}\partial{x^k}}
- 
\frac{\partial^2{\chi_{\epsilon}}}{\partial{x^l}\partial{x^i}}\Gamma^i_{jk}
- \frac{\partial{\chi_{\epsilon}}}{\partial{x^i}}
\frac{\partial{\Gamma_{jk}^i}}{\partial{x^l}}
\bigg{)}
\bigg{]} \\
&\hspace{1cm}
\bigg{[}
\frac{\partial{g^{jk}}}{\partial{x^p}}\bigg{(}
\frac{\partial^2\chi_{\epsilon}}{\partial{x^j}\partial{x^k}} - 
\frac{\partial{\chi_{\epsilon}}}{\partial{x^i}}\Gamma^{i}_{jk}
\bigg{)} +
g^{jk}\bigg{(}
\frac{\partial^3{\chi_{\epsilon}}}{\partial{x^p}\partial{x^j}\partial{x^k}}
- 
\frac{\partial^2{\chi_{\epsilon}}}{\partial{x^p}\partial{x^i}}\Gamma^i_{jk}
- \frac{\partial{\chi_{\epsilon}}}{\partial{x^i}}
\frac{\partial{\Gamma_{jk}^i}}{\partial{x^p}}
\bigg{)}
\bigg{]}g^{pl}.
\end{align*}
Evaluating at $p$, and using the fact that we are in normal coordinates about $p$, we obtain
\begin{equation*}
|d\Delta_M\chi_{\epsilon}|^2(p) = 
\sum_{j, l}
\bigg{(} 
\frac{\partial^3{\chi_{\epsilon}}}{\partial{x^p}\partial{x^j}\partial{x^j}}(p) 
-
\frac{\partial{\chi_{\epsilon}}}{\partial{x^i}}(p)
\frac{\partial{\Gamma_{jj}^i}}{\partial{x^l}}(p)
\bigg{)}^2.
\end{equation*}
The estimate then follows from \eqref{cut-off_bound_1} and lemma 
\ref{christoffel_derivatives}.
\end{proof}

\begin{lem}\label{G3_est_3}
We have $\vert G_3\vert \leq C\epsilon$, for some constant $C$ independent of $\epsilon$.
\end{lem}
\begin{proof}
It follows from \eqref{cut-off_bound_2}, \eqref{cut-off_bound_3}, lemmas
\ref{G3_est_1} and \ref{G3_est_2} that we can bound each term, making up the 
sum of $G_3$, by $C\epsilon$. The result follows.
\end{proof}

Using the above formula for $d\Delta_M\chi_{\epsilon}^k$ and lemma 
\ref{G3_est_3}, we have the following corollary.

\begin{cor}\label{dlaplace_prod_est_1}
We have the estimate $|d\Delta_M\chi_{\epsilon}^k| 
\leq C\epsilon\chi_{\epsilon}^{k-3}$ for some constant $C > 0$.
\end{cor}

\section{Localised Laplacian estimates}\label{sec_laplacian_derivatives}

In this section, we derive several localised estimates that involve the
Laplacian of a section. These will be crucial for the proof of
theorem \ref{main_theorem_1}.

We remind the reader that $\chi_{\epsilon}$ denotes the 
cut-off function constructed in section \ref{cut_offs}.
Furthermore, $k$ will be a large fixed positive integer.

\begin{prop}\label{Laplace_est}
Given $u \in W^{4,2}_{loc}(E)\cap L^2(E)$ and
$\epsilon > 0$ sufficiently small,
we have the following estimate
\begin{align*}
\vert\vert\chi_{\epsilon}^k\Delta u\vert\vert^2 &\leq 
\bigg{(} 
1 - \frac{C^2\epsilon^2}{2} - 
\frac{(2kC^2\epsilon^2)(1+2(2k-1)\epsilon)}
{2(1-(2k-1)\epsilon)}
\bigg{)}^{-1}
\bigg{(}
\frac{\epsilon}{2} + \frac{kC^2\epsilon^2}{1-(2k-1)\epsilon}
\bigg{)} \vert\vert\chi_{\epsilon}^{2k}\Delta^2u\vert\vert^2 \\
&\hspace{0.5cm} +
\bigg{(} 
1 - \frac{C^2\epsilon^2}{2} - 
\frac{(2kC^2\epsilon^2)(1+2(2k-1)\epsilon)}
{2(1-(2k-1)\epsilon)}
\bigg{)}^{-1}
\bigg{(}
\frac{1}{2\epsilon} + 2k + \frac{1}{2}
\bigg{)}\vert\vert u\vert\vert^2 
\end{align*}
\end{prop}

\begin{proof}
Integrating by parts and using \eqref{laplace_product_formula}, we can write
\begin{align*}
\langle \chi_{\epsilon}^k\Delta u, \chi_{\epsilon}^k\Delta u\rangle &=
\langle \Delta(\chi_{\epsilon}^k\Delta u), u\rangle \\
&=
\langle  \chi_{\epsilon}^{2k}\Delta^2u, u\rangle - 
2\langle \nabla_{(d\chi_{\epsilon}^{2k})^{\#}}\Delta u, u\rangle +
\langle (\Delta u)(\Delta_M\chi_{\epsilon}^{2k}), u\rangle.
\end{align*}

Using Cauchy-Schwartz and Youngs inequality, we obtain
\begin{align*}
\vert\vert\chi_{\epsilon}^k\Delta u\vert\vert^2 &\leq 
\vert \langle  \chi_{\epsilon}^{2k}\Delta^2u, u\rangle \vert + 
2\vert 
\langle \nabla_{(d\chi_{\epsilon}^{2k})^{\#}}\Delta u, u\rangle
\vert +
\vert
\langle (\Delta u)(\Delta_M\chi_{\epsilon}^{2k}), u\rangle
\vert \\
&\leq
\vert \langle  \chi_{\epsilon}^{2k}\Delta^2u, u\rangle \vert + 
2\vert 
\langle 2k\chi_{\epsilon}^{2k-1} 
\nabla_{(d\chi_{\epsilon})^{\#}}\Delta u, u\rangle
\vert +
\vert
\langle (\Delta u)(\Delta_M\chi_{\epsilon}^{2k}), u\rangle
\vert \\
&\leq 
\frac{\epsilon}{2}\vert\vert\chi_{\epsilon}^{2k}\Delta^2u \vert\vert^2 + 
\frac{1}{2\epsilon}\vert\vert u\vert\vert^2 +
2kC^2\epsilon^2\vert\vert\chi_{\epsilon}^{2k-1}\nabla\Delta u\vert\vert^2 + 
2k\vert\vert u\vert\vert^2 +
\frac{C^2\epsilon^2}{2}\vert\vert \chi_{\epsilon}^{2k-2}
\Delta u\vert\vert^2 + 
\frac{1}{2}\vert\vert u\vert\vert^2
\end{align*}
where to obtain the last inequality we have used \eqref{star_notation_ex},
\eqref{cut-off_bound_3}, and corollary \ref{laplacian_prod_est_1}.

By
proposition \ref{covariant_est_1}, we have
\begin{align*}
\vert\vert\chi_{\epsilon}^{2k-1}\nabla\Delta u\vert\vert^2 \leq 
\frac{1}{2(1-(2k-1)\epsilon)}
\vert\vert\chi_{\epsilon}^{2k}\Delta^2u\vert\vert^2 + 
\frac{1+2(2k-1)\epsilon}{2(1-(2k-1)\epsilon)}
\vert\vert\chi_{\epsilon}^{2k-2}\Delta u\vert\vert^2.
\end{align*}

This then gives
\begin{align*}
\vert\vert\chi_{\epsilon}^k\Delta u\vert\vert^2 &\leq 
\frac{\epsilon}{2}\vert\vert\chi_{\epsilon}^{2k}\Delta^2u\vert\vert^2 +
\frac{2kC^2\epsilon^2}{2(1-(2k-1)\epsilon)}
\vert\vert\chi_{\epsilon}^{2k}\Delta^2u\vert\vert^2 + 
\frac{(2kC^2\epsilon^2)(1+2(2k-1))\epsilon}
{2(1-(2k-1)\epsilon)}
\vert\vert\chi_{\epsilon}^{2k-2}\Delta u\vert\vert^2 \\
&\hspace{0.5cm} +
\frac{C^2\epsilon^2}{2}\vert\vert\chi_{\epsilon}^{2k-2}\Delta u\vert\vert^2
+ \bigg{(}\frac{1}{2\epsilon} + 2k + \frac{1}{2} \bigg{)}
\vert\vert u\vert\vert^2
\end{align*}
which in turn implies
\begin{align*}
\bigg{(}
1 - \frac{C^2\epsilon^2}{2} - \frac{(2kC^2\epsilon^2)(1+2(2k-1))\epsilon}
{2(1-(2k-1)\epsilon)}
\bigg{)}\vert\vert\chi_{\epsilon}^k\Delta u\vert\vert^2 &\leq
\bigg{(}
\frac{\epsilon}{2} + 
\frac{kC^2\epsilon^2}{1-(2k-1))\epsilon}
\bigg{)}\vert\vert\chi_{\epsilon}^{2k}\Delta^2u\vert\vert^2 \\
&\hspace{0.5cm} + 
\bigg{(}
\frac{1}{2\epsilon} + 2k + \frac{1}{2}
\bigg{)}\vert\vert u\vert\vert^2.
\end{align*}

Choosing $\epsilon$ small we can make it so that
$\bigg{(}
1 - \frac{C^2\epsilon^2}{2} - \frac{(2kC^2\epsilon^2)(1+2(2k-1))\epsilon}
{2(1-(2k-1)\epsilon)}
\bigg{)} > 0.$ 

Dividing the above through by
$\bigg{(}
1 - \frac{C^2\epsilon^2}{2} - \frac{(2kC^2\epsilon^2)(1+2(2k-1))\epsilon}
{2(1-(2k-1)\epsilon)}
\bigg{)}$ gives the result.

\end{proof}

We will also need various estimates on the absolute value of the inner 
product (defined on the Hermitian bundle $E$) of various
sections in $W^{4,2}_{loc}(E)\cap L^2(E)$.

\begin{lem}\label{bilaplacian_est_1}
Given $u \in W^{4,2}_{loc}(E) \cap L^2(E)$ and $\epsilon > 0$ sufficiently 
small, we have
\begin{align*}
\vert\langle \chi_{\epsilon}^{2k}\Delta^2u, 
\nabla_{(d\chi_{\epsilon}^{2k})^{\#}}\Delta u\rangle\vert &\leq 
\bigg{[}
\frac{\epsilon}{2} + \frac{2k^2C^2\epsilon}
{2(1-(2k-1)\epsilon)} +
\bigg{(}
\frac{(2k^2C^2\epsilon)(1+2(2k-1)\epsilon)}
{2(1-(2k-1)\epsilon)} 
\bigg{)} \\
&\hspace{0.5cm}
\bigg{(}
1 - \frac{C^2\epsilon^2}{2} - 
\frac{(2(2k-2)C^2\epsilon^2)(1+2(4k-5)\epsilon)}
{2(1-(4k-5)\epsilon)}
\bigg{)}^{-1} \\
&\hspace{0.5cm}
\bigg{(}
\frac{\epsilon}{2} + 
\frac{(2k-2)C^2\epsilon^2}{1-(4k-5)\epsilon}
\bigg{)}
\bigg{]}\vert\vert\chi_{\epsilon}^{2k}\Delta^2u\vert\vert^2 \\
&\hspace{0.5cm} +
\bigg{(}
\frac{(2k^2C^2\epsilon)(1+2(2k-1)\epsilon)}
{2(1-(2k-1)\epsilon)} 
\bigg{)} \\
&\hspace{0.5cm}
\bigg{(}
1 - \frac{C^2\epsilon^2}{2} - 
\frac{(2(2k-2)C^2\epsilon^2)(1+2(4k-5)\epsilon)}
{2(1-(4k-5)\epsilon)}
\bigg{)}^{-1} 
\bigg{(}
\frac{1}{2\epsilon} + 2(2k-2) + \frac{1}{2}
\bigg{)}\vert\vert u\vert\vert^2.
\end{align*}
\end{lem}

\begin{proof}
Applying Cauchy-Schwartz and Youngs inequality, we obtain
\begin{align*}
\vert\langle \chi_{\epsilon}^{2k}\Delta^2u, 
\nabla_{(d\chi_{\epsilon}^{2k})^{\#}}\Delta u\rangle\vert &\leq
\frac{\epsilon}{2}\vert\vert\chi_{\epsilon}^{2k}\Delta^2u\vert\vert^2 
+\frac{1}{2\epsilon}
\vert\vert\nabla_{(d\chi_{\epsilon}^{2k})^{\#}}\Delta u\vert\vert^2 \\
&=
\frac{\epsilon}{2}\vert\vert\chi_{\epsilon}^{2k}\Delta^2u\vert\vert^2 +
\frac{1}{2\epsilon}
\vert\vert 2k\chi_{\epsilon}^{2k-1}
\nabla_{(d\chi_{\epsilon})^{\#}}\Delta u\vert\vert^2 \\
&\leq
\frac{\epsilon}{2}\vert\vert\chi_{\epsilon}^{2k}\Delta^2u\vert\vert^2 +
2k^2C^2\epsilon\vert\vert\chi_{\epsilon}^{2k-1}
\nabla\Delta u\vert\vert^2
\end{align*}
where to get the last line we have used \eqref{star_notation_ex} and 
\eqref{cut-off_bound_3}.

Estimating $\vert\vert\chi_{\epsilon}^{2k-1}
\nabla\Delta u\vert\vert^2$ using proposition \ref{covariant_est_1} gives
\begin{align*}
\vert\langle \chi_{\epsilon}^{2k}\Delta^2u, 
\nabla_{(d\chi_{\epsilon}^{2k})^{\#}}\Delta u\rangle\vert \leq 
\frac{\epsilon}{2}\vert\vert\chi_{\epsilon}^{2k}\Delta^2u\vert\vert^2 &+
\frac{2k^2C^2\epsilon}{2(1-(2k-1)\epsilon)}
\vert\vert\chi_{\epsilon}^{2k}\Delta^2u\vert\vert^2 \\
&+
\frac{(2k^2C^2\epsilon)(1+2(2k-1)\epsilon)}
{2(1-(2k-1)\epsilon)}\vert\vert\chi_{\epsilon}^{2k-2}
\Delta u\vert\vert^2.
\end{align*}
Using proposition \ref{Laplace_est}, we can estimate 
$\vert\vert\chi_{\epsilon}^{2k-2}
\Delta u\vert\vert^2$ and obtain
\begin{align*}
\vert\vert\chi_{\epsilon}^{2k-2}
\Delta u\vert\vert^2 &\leq 
\bigg{(}
1 - \frac{C^2\epsilon^2}{2} - 
\frac{(2(2k-2)C^2\epsilon^2)(1+2(4k-5)\epsilon)}
{2(1-(4k-5)\epsilon)}
\bigg{)}^{-1}
\bigg{(}
\frac{\epsilon}{2} + 
\frac{(2k-2)C^2\epsilon^2}{1-(4k-5)\epsilon}
\bigg{)}\vert\vert\chi_{\epsilon}^{2(2k-2)}\Delta^2u\vert\vert^2 \\
&\hspace{0.5cm} +
\bigg{(}
1 - \frac{C^2\epsilon^2}{2} - 
\frac{(2(2k-2)C^2\epsilon^2)(1+2(4k-5)\epsilon)}
{2(1-(4k-5)\epsilon)}
\bigg{)}^{-1}
\bigg{(}
\frac{1}{2\epsilon} + 2(2k-2) + \frac{1}{2}
\bigg{)}\vert\vert u\vert\vert^2.
\end{align*}
Using this estimate, we then obtain
\begin{align*}
\vert\langle \chi_{\epsilon}^{2k}\Delta^2u, 
\nabla_{(d\chi_{\epsilon}^{2k})^{\#}}\Delta u\rangle\vert \leq
\bigg{[}
\frac{\epsilon}{2} &+ \frac{2k^2C^2\epsilon}
{2(1-(2k-1)\epsilon)} + \bigg{(} 
\frac{(2k^2C^2\epsilon)(1+2(2k-1)\epsilon)}
{2(1-(2k-1)\epsilon)}
\bigg{)} \\
&
\bigg{(}
1 - \frac{C^2\epsilon^2}{2} - 
\frac{(2(2k-2)C^2\epsilon^2)(1+2(4k-5)\epsilon)}
{2(1-(4k-5)\epsilon)}
\bigg{)}^{-1} \\
&
\bigg{(}
\frac{\epsilon}{2} + \frac{(2k-2)C^2\epsilon^2}
{1-(4k-5)\epsilon}
\bigg{)}
\bigg{]}\vert\vert\chi_{\epsilon}^{2k}\Delta^2u\vert\vert^2 \\
&+
\bigg{(} 
\frac{(2k^2C^2\epsilon)(1+2(2k-1)\epsilon)}
{2(1-(2k-1)\epsilon)}
\bigg{)} \\
&
\bigg{(}
1 - \frac{C^2\epsilon^2}{2} - 
\frac{(2(2k-2)C^2\epsilon^2)(1+2(4k-5)\epsilon)}
{2(1-(4k-5)\epsilon)}
\bigg{)}^{-1} \\
&
\bigg{(} 
\frac{1}{2\epsilon} + 2(2k-2) + \frac{1}{2}
\bigg{)}\vert\vert u\vert\vert^2
\end{align*}
which proves the result.
\end{proof}

\begin{cor}\label{bilaplace_est_1_a}
For $u \in W^{4,2}_{loc}\cap L^2(E)$ and
$\epsilon > 0$ sufficiently small, we can write
\begin{equation*}
\vert\langle \chi_{\epsilon}^{2k}\Delta^2u, 
\nabla_{(d\chi_{\epsilon}^{2k})^{\#}}\Delta u\rangle\vert \leq 
\epsilon C_1(\epsilon)\vert\vert\chi_{\epsilon}^{2k}\Delta^2u\vert\vert^2
+
C_2(\epsilon)\vert\vert u\vert\vert^2
\end{equation*}
where $C_1(\epsilon)$ and $C_2(\epsilon)$ are constants, depending 
on $\epsilon$, such that 
$\lim_{\epsilon \rightarrow 0} C_1(\epsilon) < \infty$ and 
$\lim_{\epsilon \rightarrow 0} C_2(\epsilon) < \infty$.
\end{cor}
\begin{proof}
From lemma \ref{bilaplacian_est_1}, we can write
\begin{align*}
\vert\langle \chi_{\epsilon}^{2k}\Delta^2u, 
\nabla_{(d\chi_{\epsilon}^{2k})^{\#}}\Delta u\rangle\vert \leq
\bigg{[}
\frac{\epsilon}{2} &+ \frac{2k^2C^2\epsilon}
{2(1-(2k-1)\epsilon)} + \bigg{(} 
\frac{(2k^2C^2\epsilon)(1+2(2k-1)\epsilon)}
{2(1-(2k-1)\epsilon)}
\bigg{)} \\
&
\bigg{(}
1 - \frac{C^2\epsilon^2}{2} - 
\frac{(2(2k-2)C^2\epsilon^2)(1+2(4k-5)\epsilon)}
{2(1-(4k-5)\epsilon)}
\bigg{)}^{-1} \\
&
\bigg{(}
\frac{\epsilon}{2} + \frac{(2k-2)C^2\epsilon^2}
{1-(4k-5)\epsilon}
\bigg{)}
\bigg{]}\vert\vert\chi_{\epsilon}^{2k}\Delta^2u\vert\vert^2 \\
&+
\bigg{(} 
\frac{(2k^2C^2\epsilon)(1+2(2k-1)\epsilon)}
{2(1-(2k-1)\epsilon)}
\bigg{)} \\
&
\bigg{(}
1 - \frac{C^2\epsilon^2}{2} - 
\frac{(2(2k-2)C^2\epsilon^2)(1+2(4k-5)\epsilon)}
{2(1-(4k-5)\epsilon)}
\bigg{)}^{-1} \\
&
\bigg{(} 
\frac{1}{2\epsilon} + 2(2k-2) + \frac{1}{2}
\bigg{)}\vert\vert u\vert\vert^2.
\end{align*}
We then write
\begin{align*}
\bigg{[}
\frac{\epsilon}{2} &+ \frac{2k^2C^2\epsilon}
{2(1-(2k-1)\epsilon)} 
+ \bigg{(} 
\frac{(2k^2C^2\epsilon)(1+2(2k-1)\epsilon)}
{2(1-(2k-1)\epsilon)}
\bigg{)} \\
&
\bigg{(}
1 - \frac{C^2\epsilon^2}{2} - 
\frac{(2(2k-2)C^2\epsilon^2)(1+2(4k-5)\epsilon)}
{2(1-(4k-5)\epsilon)}
\bigg{)}^{-1} 
\bigg{(}
\frac{\epsilon}{2} + \frac{(2k-2)C^2\epsilon^2}
{1-(4k-5)\epsilon}
\bigg{)}
\bigg{]} \\
&= \\
\epsilon\bigg{[}
\frac{1}{2} &+ \frac{2k^2C^2}
{2(1-(2k-1)\epsilon)} 
+ \bigg{(} 
\frac{(2k^2C^2)(1+2(2k-1)\epsilon)}
{2(1-(2k-1)\epsilon)}
\bigg{)} \\
&
\bigg{(}
1 - \frac{C^2\epsilon^2}{2} - 
\frac{(2(2k-2)C^2\epsilon^2)(1+2(4k-5)\epsilon)}
{2(1-(4k-5)\epsilon)}
\bigg{)}^{-1} 
\bigg{(}
\frac{\epsilon}{2} + \frac{(2k-2)C^2\epsilon^2}
{1-(4k-5)\epsilon}
\bigg{)}
\bigg{]}
\end{align*}
and define
\begin{align*}
C_1(\epsilon) &=
\bigg{[}
\frac{1}{2} + \frac{2k^2C^2}
{2(1-(2k-1)\epsilon)} 
+ \bigg{(} 
\frac{(2k^2C^2)(1+2(2k-1)\epsilon)}
{2(1-(2k-1)\epsilon)}
\bigg{)} \\
&\hspace{0.5cm}
\bigg{(}
1 - \frac{C^2\epsilon^2}{2} - 
\frac{(2(2k-2)C^2\epsilon^2)(1+2(4k-5)\epsilon)}
{2(1-(4k-5)\epsilon)}
\bigg{)}^{-1} 
\bigg{(}
\frac{\epsilon}{2} + \frac{(2k-2)C^2\epsilon^2}
{1-(4k-5)\epsilon}
\bigg{)}
\bigg{]}.
\end{align*}
It is easy to see that we have
\begin{equation*}
\lim_{\epsilon \rightarrow 0} C_1(\epsilon) = 
\frac{1}{2} + 2k^2C^2.
\end{equation*}
We then define
\begin{align*}
C_2(\epsilon) &= 
\bigg{(} 
\frac{(2k^2C^2\epsilon)(1+2(2k-1)\epsilon)}
{2(1-(2k-1)\epsilon)}
\bigg{)} 
\bigg{(}
1 - \frac{C^2\epsilon^2}{2} - 
\frac{(2(2k-2)C^2\epsilon^2)(1+2(4k-5)\epsilon)}
{2(1-(4k-5)\epsilon)}
\bigg{)}^{-1} \\
&\hspace{0.5cm}
\bigg{(} 
\frac{1}{2\epsilon} + 2(2k-2) + \frac{1}{2}
\bigg{)} \\
&=
\bigg{(} 
\frac{(2k^2C^2)(1+2(2k-1)\epsilon)}
{2(1-(2k-1)\epsilon)}
\bigg{)} 
\bigg{(}
1 - \frac{C^2\epsilon^2}{2} - 
\frac{(2(2k-2)C^2\epsilon^2)(1+2(4k-5)\epsilon)}
{2(1-(4k-5)\epsilon)}
\bigg{)}^{-1} \\
&\hspace{0.5cm}
\bigg{(} 
\frac{1}{2} + 2(2k-2)\epsilon + \frac{\epsilon}{2}
\bigg{)}.
\end{align*}
It is then easy to see that we have
\begin{equation*}
\lim_{\epsilon\rightarrow 0} C_2(\epsilon) = \frac{k^2C^2}{2}.
\end{equation*}
\end{proof}

\begin{lem}\label{bilaplace_est_2}
Given $u \in W^{4,2}_{loc}(E) \cap L^2(E)$ and $\epsilon > 0$ sufficiently small, we have
\begin{align*}
\vert \langle \chi_{\epsilon}^{2k}\Delta^2u, \Delta\nabla_{(d\chi_{\epsilon}^{2k})^{\#}}u
\rangle\vert \leq
\bigg{[}
\frac{3\epsilon}{2} &+ 
\frac{2k^2C^2\epsilon}{2(1-(2k-1)\epsilon)} +
\bigg{(}
\frac{(2k^2C^2\epsilon)(1+2(2k-1)\epsilon)}
{2(1-(2k-1)\epsilon)}
\bigg{)} \\
&
\bigg{(}
1 - \frac{C^2\epsilon^2}{2} - 
\frac{2(2k-2)C^2\epsilon^2)(1+2(4k-5)\epsilon)}
{2(1-(4k-5)\epsilon)}
\bigg{)}^{-1}
\bigg{(}
\frac{\epsilon}{2} + \frac{(2k-2)C^2\epsilon^2}
{1-(4k-5)\epsilon}
\bigg{)} \\
&+
\bigg{(}
\frac{2k^2C^2\epsilon}{2(1-(2k-1)\epsilon)}
\bigg{)}
\bigg{(}
1 - \frac{C^2\epsilon^2}{2} - 
\frac{(4kC^2\epsilon^2)(1+2(4k-1)\epsilon)}
{2(1-(4k-1)\epsilon)}
\bigg{)}^{-1} \\
&
\bigg{(}
\frac{\epsilon}{2} + \frac{2kC^2\epsilon^2}
{(1-(4k-1)\epsilon)}
\bigg{)}\bigg{]}\vert\vert\chi_{\epsilon}^{2k}\Delta^2u\vert\vert^2 \\
&+
\bigg{[}
\bigg{(}
\frac{(2k^2C^2\epsilon)(1+2(2k-1)\epsilon)}
{2(1-(2k-1)\epsilon)}
\bigg{)} \\
&
\bigg{(}
1 - \frac{C^2\epsilon^2}{2} - 
\frac{2(2k-2)C^2\epsilon^2)(1+2(4k-5)\epsilon)}
{2(1-(4k-5)\epsilon)}
\bigg{)}^{-1}
\bigg{(}
\frac{1}{2\epsilon} + 2(2k-2) + \frac{1}{2}
\bigg{)} \\
&+ 2k^2C^2\epsilon \\
&+ 
\bigg{(}
\frac{2k^2C^2\epsilon}{2(1-(2k-1)\epsilon)}
\bigg{)}
\bigg{(}
1 - \frac{C^2\epsilon^2}{2} - 
\frac{(4kC^2\epsilon^2)(1+2(4k-1)\epsilon)}
{2(1-(4k-1)\epsilon)}
\bigg{)}^{-1} \\
&+
\bigg{(}
\frac{1}{2\epsilon} + 2k + \frac{1}{2}
\bigg{)} +
\bigg{(}
\frac{(2k^2C^2\epsilon)(1+2(2k-1)\epsilon)}
{2(1-(2k-1)\epsilon)}
\bigg{)}
\bigg{]}\vert\vert u\vert\vert^2.
\end{align*}
\end{lem}

\begin{proof}
We start by using the commutation formula, 
see corollary \ref{connection_3},
to obtain
\begin{align*}
\vert \langle \chi_{\epsilon}^{2k}\Delta^2u, \Delta\nabla_{(d\chi_{\epsilon}^{2k})^{\#}}u
\rangle\vert &=
\vert \langle \chi_{\epsilon}^{2k}\Delta^2u, \nabla_{(d\chi_{\epsilon}^{2k})^{\#}}\Delta u + 
\nabla_{(d\chi_{\epsilon}^{2k})^{\#}}(Rm+F)*u + (Rm+F)*\nabla_{(d\chi_{\epsilon}^{2k})^{\#}}u
\rangle\vert \\
&\leq
\vert\langle \chi_{\epsilon}^{2k}\Delta^2u, \nabla_{(d\chi_{\epsilon}^{2k})^{\#}}\Delta u\rangle\vert
+
\vert\langle  \chi_{\epsilon}^{2k}\Delta^2u, \nabla_{(d\chi_{\epsilon}^{2k})^{\#}}(Rm+F)*u
\rangle\vert \\
&\hspace{1cm}+
\vert\langle \chi_{\epsilon}^{2k}\Delta^2u, (Rm+F)*\nabla_{(d\chi_{\epsilon}^{2k})^{\#}}u
\rangle\vert.
\end{align*}

The term 
$\vert\langle  \chi_{\epsilon}^{2k}\Delta^2u, \nabla_{(d\chi_{\epsilon}^{2k})^{\#}}(Rm+F)*u
\rangle\vert$ is estimated as follows.
\begin{equation*}
\vert\langle  \chi_{\epsilon}^{2k}\Delta^2u, \nabla_{(d\chi_{\epsilon}^{2k})^{\#}}(Rm+F)*u
\rangle\vert = 
\vert\langle  \chi_{\epsilon}^{2k}\Delta^2u, 2k\chi_{\epsilon}^{2k-1}
\nabla_{(d\chi_{\epsilon})^{\#}}(Rm+F)*u
\rangle\vert.
\end{equation*}
Apply Cauchy-Schwarz and Young's inequality to get
\begin{align*}
\vert\langle  \chi_{\epsilon}^{2k}\Delta^2u, 2k\chi_{\epsilon}^{2k-1}
\nabla_{(d\chi_{\epsilon})^{\#}}(Rm+F)*u
\rangle\vert &\leq 
\frac{\epsilon}{2}\vert\vert\chi_{\epsilon}^{2k}\Delta^2u\vert\vert^2 +
\frac{1}{2\epsilon}\vert\vert 2k\chi_{\epsilon}^{2k-1}
\nabla_{(d\chi_{\epsilon})^{\#}}(Rm+F)*u\vert\vert^2 \\
&\leq 
\frac{\epsilon}{2}\vert\vert\chi_{\epsilon}^{2k}\Delta^2u\vert\vert^2 +
2k^2C^2\epsilon\vert\vert\chi_{\epsilon}^{2k-1}u\vert\vert^2
\end{align*}
where to get the second inequality we have used our bounded
geometry assumption (A1) and (A2).

The term 
$\vert\langle \chi_{\epsilon}^{2k}\Delta^2u, (Rm+F)*\nabla_{(d\chi_{\epsilon}^{2k})^{\#}}u
\rangle\vert$ can be estimated as follows.

\begin{align*}
\vert\langle \chi_{\epsilon}^{2k}\Delta^2u, (Rm+F)*\nabla_{(d
\chi_{\epsilon}^{2k})^{\#}}u
\rangle\vert &\leq \frac{\epsilon}{2}
\vert\vert\chi_{\epsilon}^{2k}\Delta^2u\vert\vert^2 + \frac{1}{2\epsilon}
\vert\vert (Rm+F)*\nabla_{(d\chi_{\epsilon}^{2k})^{\#}}u\vert\vert^2 
 \\
&\leq
\frac{\epsilon}{2}
\vert\vert\chi_{\epsilon}^{2k}\Delta^2u\vert\vert^2 + 
\frac{1}{2\epsilon}
\vert\vert (Rm+F)*2k
\chi_{\epsilon}^{2k-1}\nabla_{(d\chi_{\epsilon})^{\#}}u\vert\vert^2 
 \\
&\leq 
\frac{\epsilon}{2}
\vert\vert\chi_{\epsilon}^{2k}\Delta^2u\vert\vert^2 +
2k^2C^2\epsilon\vert\vert\chi_{\epsilon}^{2k-1}\nabla u\vert\vert^2 
 \\
&\leq
\frac{\epsilon}{2}
\vert\vert\chi_{\epsilon}^{2k}\Delta^2u\vert\vert^2 + 
\frac{2k^2C^2\epsilon}{2(2-(2k-1)\epsilon)}
\vert\vert\chi_{\epsilon}^{2k}\Delta u\vert\vert^2 
 \\
&\hspace{0.5cm}+ 
\frac{(2k^2C^2\epsilon)(1+2(2k-1)\epsilon)}
{2(1-(2k-1)\epsilon)}\vert\vert u\vert\vert^2
\end{align*}
where to get the first inequality we have used Cauchy-Schwarz and Young's 
inequality, to get the third inequality we have used our bounded
geometry assumption (A1) and (A2),
and to get the last inequality we have used proposition 
\ref{covariant_est_1}.

The term $\vert\langle \chi_{\epsilon}^{2k}\Delta^2u, 
\nabla_{(d\chi_{\epsilon}^{2k})^{\#}}
\Delta u\rangle\vert$ was estimated in the previous lemma, and
we can use proposition \ref{Laplace_est} to estimate the term 
$\vert\vert\chi_{\epsilon}^{2k}\Delta u\vert\vert^2$. This gives
\begin{align*}
\vert \langle \chi_{\epsilon}^{2k}\Delta^2u, \Delta\nabla_{(d\chi_{\epsilon}^{2k})^{\#}}u
\rangle\vert \leq
\bigg{[}
\frac{3\epsilon}{2} &+ 
\frac{2k^2C^2\epsilon}{2(1-(2k-1)\epsilon)} +
\bigg{(}
\frac{(2k^2C^2\epsilon)(1+2(2k-1)\epsilon)}
{2(1-(2k-1)\epsilon)}
\bigg{)} \\
&
\bigg{(}
1 - \frac{C^2\epsilon^2}{2} - 
\frac{2(2k-2)C^2\epsilon^2)(1+2(4k-5)\epsilon)}
{2(1-(4k-5)\epsilon)}
\bigg{)}^{-1}
\bigg{(}
\frac{\epsilon}{2} + \frac{(2k-2)C^2\epsilon^2}
{1-(4k-5)\epsilon}
\bigg{)} \\
&+
\bigg{(}
\frac{2k^2C^2\epsilon}{2(1-(2k-1)\epsilon)}
\bigg{)}
\bigg{(}
1 - \frac{C^2\epsilon^2}{2} - 
\frac{(4kC^2\epsilon^2)(1+2(4k-1)\epsilon)}
{2(1-(4k-1)\epsilon)}
\bigg{)}^{-1} \\
&
\bigg{(}
\frac{\epsilon}{2} + \frac{2kC^2\epsilon^2}
{(1-(4k-1)\epsilon)}
\bigg{)}\bigg{]}\vert\vert\chi_{\epsilon}^{2k}\Delta^2u\vert\vert^2 \\
&+
\bigg{[}
\bigg{(}
\frac{(2k^2C^2\epsilon)(1+2(2k-1)\epsilon)}
{2(1-(2k-1)\epsilon)}
\bigg{)} \\
&
\bigg{(}
1 - \frac{C^2\epsilon^2}{2} - 
\frac{2(2k-2)C^2\epsilon^2)(1+2(4k-5)\epsilon)}
{2(1-(4k-5)\epsilon)}
\bigg{)}^{-1}
\bigg{(}
\frac{1}{2\epsilon} + 2(2k-2) + \frac{1}{2}
\bigg{)} \\
&+ 2k^2C^2\epsilon \\
&+ 
\bigg{(}
\frac{2k^2C^2\epsilon}{2(1-(2k-1)\epsilon)}
\bigg{)}
\bigg{(}
1 - \frac{C^2\epsilon^2}{2} - 
\frac{(4kC^2\epsilon^2)(1+2(4k-1)\epsilon)}
{2(1-(4k-1)\epsilon)}
\bigg{)}^{-1} \\
&+
\bigg{(}
\frac{1}{2\epsilon} + 2k + \frac{1}{2}
\bigg{)} +
\bigg{(}
\frac{(2k^2C^2\epsilon)(1+2(2k-1)\epsilon)}
{2(1-(2k-1)\epsilon)}
\bigg{)}
\bigg{]}\vert\vert u\vert\vert^2
\end{align*}
which proves the result.
\end{proof}

\begin{cor}\label{bilaplace_est_2_a}
For $u \in W^{4,2}_{loc}(E)\cap L^2(E)$ and $\epsilon > 0$ sufficiently 
small, we can write
\begin{equation*}
\vert \langle \chi_{\epsilon}^{2k}\Delta^2u, \Delta\nabla_{(d
\chi_{\epsilon}^{2k})^{\#}}u
\rangle\vert \leq 
\epsilon C_1(\epsilon)\vert\vert\chi_{\epsilon}^{2k}\Delta^2u\vert\vert^2
+
C_2(\epsilon)\vert\vert u\vert\vert^2
\end{equation*}
where $C_1(\epsilon)$ and $C_2(\epsilon)$ are constants, depending 
on $\epsilon$, such that 
$\lim_{\epsilon \rightarrow 0} C_1(\epsilon) < \infty$ and 
$\lim_{\epsilon \rightarrow 0} C_2(\epsilon) < \infty$.
\end{cor}

\begin{proof}
We use lemma \ref{bilaplace_est_2}, and write 
\begin{align*}
\vert \langle \chi_{\epsilon}^{2k}\Delta^2u, \Delta\nabla_{(d\chi_{\epsilon}^{2k})^{\#}}u
\rangle\vert \leq
\bigg{[}
\frac{3\epsilon}{2} &+ 
\frac{2k^2C^2\epsilon}{2(1-(2k-1)\epsilon)} +
\bigg{(}
\frac{(2k^2C^2\epsilon)(1+2(2k-1)\epsilon)}
{2(1-(2k-1)\epsilon)}
\bigg{)} \\
&
\bigg{(}
1 - \frac{C^2\epsilon^2}{2} - 
\frac{2(2k-2)C^2\epsilon^2)(1+2(4k-5)\epsilon)}
{2(1-(4k-5)\epsilon)}
\bigg{)}^{-1}
\bigg{(}
\frac{\epsilon}{2} + \frac{(2k-2)C^2\epsilon^2}
{1-(4k-5)\epsilon}
\bigg{)} \\
&+
\bigg{(}
\frac{2k^2C^2\epsilon}{2(1-(2k-1)\epsilon)}
\bigg{)}
\bigg{(}
1 - \frac{C^2\epsilon^2}{2} - 
\frac{(4kC^2\epsilon^2)(1+2(4k-1)\epsilon)}
{2(1-(4k-1)\epsilon)}
\bigg{)}^{-1} \\
&
\bigg{(}
\frac{\epsilon}{2} + \frac{2kC^2\epsilon^2}
{(1-(4k-1)\epsilon)}
\bigg{)}\bigg{]}\vert\vert\chi_{\epsilon}^{2k}\Delta^2u\vert\vert^2 \\
&+
\bigg{[}
\bigg{(}
\frac{(2k^2C^2\epsilon)(1+2(2k-1)\epsilon)}
{2(1-(2k-1)\epsilon)}
\bigg{)} \\
&
\bigg{(}
1 - \frac{C^2\epsilon^2}{2} - 
\frac{2(2k-2)C^2\epsilon^2)(1+2(4k-5)\epsilon)}
{2(1-(4k-5)\epsilon)}
\bigg{)}^{-1}
\bigg{(}
\frac{1}{2\epsilon} + 2(2k-2) + \frac{1}{2}
\bigg{)} \\
&+ 2k^2C^2\epsilon \\
&+ 
\bigg{(}
\frac{2k^2C^2\epsilon}{2(1-(2k-1)\epsilon)}
\bigg{)}
\bigg{(}
1 - \frac{C^2\epsilon^2}{2} - 
\frac{(4kC^2\epsilon^2)(1+2(4k-1)\epsilon)}
{2(1-(4k-1)\epsilon)}
\bigg{)}^{-1} \\
&+
\bigg{(}
\frac{1}{2\epsilon} + 2k + \frac{1}{2}
\bigg{)} +
\bigg{(}
\frac{(2k^2C^2\epsilon)(1+2(2k-1)\epsilon)}
{2(1-(2k-1)\epsilon)}
\bigg{)}
\bigg{]}\vert\vert u\vert\vert^2.
\end{align*}

Then observe that we can write
\begin{align*}
\bigg{[}
\frac{3\epsilon}{2} &+ 
\frac{2k^2C^2\epsilon}{2(1-(2k-1)\epsilon)} +
\bigg{(}
\frac{(2k^2C^2\epsilon)(1+2(2k-1)\epsilon)}
{2(1-(2k-1)\epsilon)}
\bigg{)} \\
&
\bigg{(}
1 - \frac{C^2\epsilon^2}{2} - 
\frac{2(2k-2)C^2\epsilon^2)(1+2(4k-5)\epsilon)}
{2(1-(4k-5)\epsilon)}
\bigg{)}^{-1}
\bigg{(}
\frac{\epsilon}{2} + \frac{(2k-2)C^2\epsilon^2}
{1-(4k-5)\epsilon}
\bigg{)} \\
&+
\bigg{(}
\frac{2k^2C^2\epsilon}{2(1-(2k-1)\epsilon)}
\bigg{)}
\bigg{(}
1 - \frac{C^2\epsilon^2}{2} - 
\frac{(4kC^2\epsilon^2)(1+2(4k-1)\epsilon)}
{2(1-(4k-1)\epsilon)}
\bigg{)}^{-1} \\
&
\bigg{(}
\frac{\epsilon}{2} + \frac{2kC^2\epsilon^2}
{(1-(4k-1)\epsilon)}
\bigg{)}\bigg{]} \\
= \\
\epsilon\bigg{[}
\frac{3}{2} &+ 
\frac{2k^2C^2}{2(1-(2k-1)\epsilon)} +
\bigg{(}
\frac{(2k^2C^2)(1+2(2k-1)\epsilon)}
{2(1-(2k-1)\epsilon)}
\bigg{)} \\
&
\bigg{(}
1 - \frac{C^2\epsilon^2}{2} - 
\frac{2(2k-2)C^2\epsilon^2)(1+2(4k-5)\epsilon)}
{2(1-(4k-5)\epsilon)}
\bigg{)}^{-1}
\bigg{(}
\frac{\epsilon}{2} + \frac{(2k-2)C^2\epsilon^2}
{1-(4k-5)\epsilon}
\bigg{)} \\
&+
\bigg{(}
\frac{2k^2C^2}{2(1-(2k-1)\epsilon)}
\bigg{)}
\bigg{(}
1 - \frac{C^2\epsilon^2}{2} - 
\frac{(4kC^2\epsilon^2)(1+2(4k-1)\epsilon)}
{2(1-(4k-1)\epsilon)}
\bigg{)}^{-1} \\
&
\bigg{(}
\frac{\epsilon}{2} + \frac{2kC^2\epsilon^2}
{(1-(4k-1)\epsilon)}
\bigg{)}\bigg{]}.
\end{align*}

Defining 
\begin{align*}
C_1(\epsilon) =
\frac{3}{2} &+ 
\frac{2k^2C^2}{2(1-(2k-1)\epsilon)} +
\bigg{(}
\frac{(2k^2C^2)(1+2(2k-1)\epsilon)}
{2(1-(2k-1)\epsilon)}
\bigg{)} \\
&
\bigg{(}
1 - \frac{C^2\epsilon^2}{2} - 
\frac{2(2k-2)C^2\epsilon^2)(1+2(4k-5)\epsilon)}
{2(1-(4k-5)\epsilon)}
\bigg{)}^{-1}
\bigg{(}
\frac{\epsilon}{2} + \frac{(2k-2)C^2\epsilon^2}
{1-(4k-5)\epsilon}
\bigg{)} \\
&+
\bigg{(}
\frac{2k^2C^2}{2(1-(2k-1)\epsilon)}
\bigg{)}
\bigg{(}
1 - \frac{C^2\epsilon^2}{2} - 
\frac{(4kC^2\epsilon^2)(1+2(4k-1)\epsilon)}
{2(1-(4k-1)\epsilon)}
\bigg{)}^{-1} \\
&
\bigg{(}
\frac{\epsilon}{2} + \frac{2kC^2\epsilon^2}
{(1-(4k-1)\epsilon)}
\bigg{)}.
\end{align*}
It is easy to see that 
$\lim_{\epsilon\rightarrow 0}C_1(\epsilon) = \frac{3}{2} + k^2C$.

We then define $C_2(\epsilon)$ to be the coefficient of 
$\vert\vert u\vert\vert^2$
\begin{align*}
C_2(\epsilon) &=
\bigg{(}
\frac{(2k^2C^2\epsilon)(1+2(2k-1)\epsilon)}
{2(1-(2k-1)\epsilon)}
\bigg{)} \\
&\hspace{0.5cm}
\bigg{(}
1 - \frac{C^2\epsilon^2}{2} - 
\frac{2(2k-2)C^2\epsilon^2)(1+2(4k-5)\epsilon)}
{2(1-(4k-5)\epsilon)}
\bigg{)}^{-1}
\bigg{(}
\frac{1}{2\epsilon} + 2(2k-2) + \frac{1}{2}
\bigg{)} \\
&\hspace{0.5cm}+ 2k^2C^2\epsilon +
 \bigg{(}
\frac{2k^2C^2\epsilon}{2(1-(2k-1)\epsilon)}
\bigg{)}
\bigg{(}
1 - \frac{C^2\epsilon^2}{2} - 
\frac{(4kC^2\epsilon^2)(1+2(4k-1)\epsilon)}
{2(1-(4k-1)\epsilon)}
\bigg{)}^{-1} \\
&\hspace{0.5cm}+
\bigg{(}
\frac{1}{2\epsilon} + 2k + \frac{1}{2}
\bigg{)} +
\bigg{(}
\frac{(2k^2C^2\epsilon)(1+2(2k-1)\epsilon)}
{2(1-(2k-1)\epsilon)}
\bigg{)}.
\end{align*}
We then observe that writing
\begin{align*}
&\bigg{(}
\frac{(2k^2C^2\epsilon)(1+2(2k-1)\epsilon)}
{2(1-(2k-1)\epsilon)}
\bigg{)} 
\bigg{(}
1 - \frac{C^2\epsilon^2}{2} - 
\frac{2(2k-2)C^2\epsilon^2)(1+2(4k-5)\epsilon)}
{2(1-(4k-5)\epsilon)}
\bigg{)}^{-1}
\bigg{(}
\frac{1}{2\epsilon} + 2(2k-2) + \frac{1}{2}
\bigg{)} \\
&\hspace{0.5cm}+ 2k^2C^2\epsilon +
 \bigg{(}
\frac{2k^2C^2\epsilon}{2(1-(2k-1)\epsilon)}
\bigg{)}
\bigg{(}
1 - \frac{C^2\epsilon^2}{2} - 
\frac{(4kC^2\epsilon^2)(1+2(4k-1)\epsilon)}
{2(1-(4k-1)\epsilon)}
\bigg{)}^{-1} \\
&\hspace{0.5cm}+
\bigg{(}
\frac{1}{2\epsilon} + 2k + \frac{1}{2}
\bigg{)} +
\bigg{(}
\frac{(2k^2C^2\epsilon)(1+2(2k-1)\epsilon)}
{2(1-(2k-1)\epsilon)}
\bigg{)} \\
&= \\
&\bigg{(}
\frac{(2k^2C^2)(1+2(2k-1)\epsilon)}
{2(1-(2k-1)\epsilon)}
\bigg{)} 
\bigg{(}
1 - \frac{C^2\epsilon^2}{2} - 
\frac{2(2k-2)C^2\epsilon^2)(1+2(4k-5)\epsilon)}
{2(1-(4k-5)\epsilon)}
\bigg{)}^{-1}
\bigg{(}
\frac{1}{2} + 2(2k-2)\epsilon + \frac{\epsilon}{2}
\bigg{)} \\
&\hspace{0.5cm}+ 2k^2C^2\epsilon +
 \bigg{(}
\frac{2k^2C^2}{2(1-(2k-1)\epsilon)}
\bigg{)}
\bigg{(}
1 - \frac{C^2\epsilon^2}{2} - 
\frac{(4kC^2\epsilon^2)(1+2(4k-1)\epsilon)}
{2(1-(4k-1)\epsilon)}
\bigg{)}^{-1} \\
&\hspace{0.5cm}+
\bigg{(}
\frac{1}{2} + 2k\epsilon + \frac{\epsilon}{2}
\bigg{)} +
\bigg{(}
\frac{(2k^2C^2\epsilon)(1+2(2k-1)\epsilon)}
{2(1-(2k-1)\epsilon)}
\bigg{)}.
\end{align*}
we have $\lim_{\epsilon \rightarrow 0}C_2(\epsilon) = k^2C^2$.

This proves the corollary.

\end{proof}

\begin{lem}\label{bilaplace_est_3}
Given $u \in W^{4,2}_{loc}(E)$ and $\epsilon > 0$ sufficiently small, we have
\begin{align*}
\vert\langle \chi_{\epsilon}^{2k}\Delta^2u, 
\nabla_{(d\Delta_M\chi_{\epsilon}^{2k})^{\#}}u\rangle\vert &\leq
\bigg{[}
\frac{\epsilon}{2} + 
\bigg{(}\frac{C^2\epsilon}{4(1-(2k-3)\epsilon)} \bigg{)}
\bigg{(}
1 - \frac{C^2\epsilon^2}{2} - 
\frac{(2(2k-2)C^2\epsilon^2)(1+2(4k-5)\epsilon)}
{2(1-(4k-5)\epsilon)}
\bigg{)}^{-1} \\
&\hspace{0.5cm}
\bigg{(} 
\frac{\epsilon}{2} + 
\frac{(2k-2)C^2\epsilon^2}{1-(4k-5)\epsilon)} 
\bigg{)}
\bigg{]}\vert\vert\chi_{\epsilon}^{2k}\Delta^2u\vert\vert^2 \\
&\hspace{0.5cm} + 
\bigg{[}
\bigg{(}\frac{C^2\epsilon}{4(1-(2k-3)\epsilon)} \bigg{)}
\bigg{(}
1 - \frac{C^2\epsilon^2}{2} - 
\frac{(2(2k-2)C^2\epsilon^2)(1+2(4k-5)\epsilon)}
{2(1-(4k-5)\epsilon)}
\bigg{)}^{-1} \\
&\hspace{0.5cm} 
\bigg{(} 
\frac{1}{2\epsilon} + 2(2k-2) + \frac{1}{2}
\bigg{)} 
+
\bigg{(}
\frac{(C^2\epsilon)(1+2(2k-3)\epsilon)}
{4(1-(2k-3)\epsilon)}
\bigg{)} 
\bigg{]}\vert\vert u\vert\vert^2
\end{align*}
\end{lem}
\begin{proof}

Applying Cauchy-Schwarz and Young's inequality, we obtain
\begin{align*}
\vert\langle \chi_{\epsilon}^{2k}\Delta^2u, 
\nabla_{(d\Delta_M\chi_{\epsilon}^{2k})^{\#}}u\rangle\vert &\leq 
\frac{\epsilon}{2}\vert\vert\chi_{\epsilon}^{2k}\Delta^2u\vert\vert^2 
+ \frac{1}{2\epsilon}\vert\vert 
\nabla_{(d\Delta_M\chi_{\epsilon}^{2k})^{\#}}u\vert\vert^2 \\
&\leq
\frac{\epsilon}{2}\vert\vert\chi_{\epsilon}^{2k}\Delta^2u\vert\vert^2 
+
\frac{C^2\epsilon}{2}\vert\vert\chi_{\epsilon}^{2k-3}\nabla u\vert\vert^2
\end{align*}
where we have used corollary \ref{dlaplace_prod_est_1} to get the 
second inequality.

Using proposition \ref{covariant_est_1}, we can estimate the term
$\vert\vert \chi_{\epsilon}^{2k-3}\nabla u\vert\vert^2$ and obtain
\begin{align*}
\vert\langle \chi_{\epsilon}^{2k}\Delta^2u, 
\nabla_{(d\Delta_M\chi_{\epsilon}^{2k})^{\#}}u\rangle\vert \leq 
\frac{\epsilon}{2}\vert\vert\chi_{\epsilon}^{2k}\Delta^2u\vert\vert^2 
&+ \frac{C^2\epsilon}{4(1-(2k-3)\epsilon)}
\vert\vert\chi_{\epsilon}^{2k-2}\Delta u\vert\vert^2 \\
&+
\frac{(C^2\epsilon)(1+2(2k-3)\epsilon)}
{4(1-(2k-3)\epsilon)}\vert\vert 
 \chi_{\epsilon}^{2k-4}u\vert\vert^2.
\end{align*}
Estimating the term $\vert\vert\chi_{\epsilon}^{2k-2}\Delta u\vert\vert^2$
using proposition \ref{Laplace_est}, and substituting it into the above proves the lemma.
\end{proof}

\begin{cor}\label{bilaplace_est_3_a}
For $u \in W^{4,2}_{loc}(E) \cap L^2(E)$ and $\epsilon > 0$ sufficiently 
small, we can write
\begin{equation*}
\vert\langle \chi_{\epsilon}^{2k}\Delta^2u, 
\nabla_{(d\Delta_M\chi_{\epsilon}^{2k})^{\#}}u\rangle\vert \leq 
\epsilon C_1(\epsilon)\vert\vert\chi_{\epsilon}^{2k}\Delta^2u\vert\vert^2
+
C_2(\epsilon)\vert\vert u\vert\vert^2
\end{equation*}
where $C_1(\epsilon)$ and $C_2(\epsilon)$ are constants, depending 
on $\epsilon$, such that 
$\lim_{\epsilon \rightarrow 0} C_1(\epsilon) < \infty$ and 
$\lim_{\epsilon \rightarrow 0} C_2(\epsilon) < \infty$.
\end{cor}

\begin{proof}
By lemma \ref{bilaplace_est_3} we have
\begin{align*}
\vert\langle \chi_{\epsilon}^{2k}\Delta^2u, 
\nabla_{(d\Delta_M\chi_{\epsilon}^{2k})^{\#}}u\rangle\vert &\leq
\bigg{[}
\frac{\epsilon}{2} + 
\bigg{(}\frac{C^2\epsilon}{4(1-(2k-3)\epsilon)} \bigg{)}
\bigg{(}
1 - \frac{C^2\epsilon^2}{2} - 
\frac{(2(2k-2)C^2\epsilon^2)(1+2(4k-5)\epsilon)}
{2(1-(4k-5)\epsilon)}
\bigg{)}^{-1} \\
&\hspace{0.5cm}
\bigg{(} 
\frac{\epsilon}{2} + 
\frac{(2k-2)C^2\epsilon^2}{1-(4k-5)\epsilon)} 
\bigg{)}
\bigg{]}\vert\vert\chi_{\epsilon}^{2k}\Delta^2u\vert\vert^2 \\
&\hspace{0.5cm} + 
\bigg{[}
\bigg{(}\frac{C^2\epsilon}{4(1-(2k-3)\epsilon)} \bigg{)}
\bigg{(}
1 - \frac{C^2\epsilon^2}{2} - 
\frac{(2(2k-2)C^2\epsilon^2)(1+2(4k-5)\epsilon)}
{2(1-(4k-5)\epsilon)}
\bigg{)}^{-1} \\
&\hspace{0.5cm} 
\bigg{(} 
\frac{1}{2\epsilon} + 2(2k-2) + \frac{1}{2}
\bigg{)} 
+
\bigg{(}
\frac{(C^2\epsilon)(1+2(2k-3)\epsilon)}
{4(1-(2k-3)\epsilon)}
\bigg{)} 
\bigg{]}\vert\vert u\vert\vert^2.
\end{align*}
Write
\begin{align*}
&\bigg{[}
\frac{\epsilon}{2} + 
\bigg{(}\frac{C^2\epsilon}{4(1-(2k-3)\epsilon)} \bigg{)}
\bigg{(}
1 - \frac{C^2\epsilon^2}{2} - 
\frac{(2(2k-2)C^2\epsilon^2)(1+2(4k-5)\epsilon)}
{2(1-(4k-5)\epsilon)}
\bigg{)}^{-1} 
\bigg{(} 
\frac{\epsilon}{2} + 
\frac{(2k-2)C^2\epsilon^2}{1-(4k-5)\epsilon)} 
\bigg{)}
\bigg{]} \\
&= \\
&\epsilon\bigg{[}
\frac{1}{2} + 
\bigg{(}\frac{C^2}{4(1-(2k-3)\epsilon)} \bigg{)}
\bigg{(}
1 - \frac{C^2\epsilon^2}{2} - 
\frac{(2(2k-2)C^2\epsilon^2)(1+2(4k-5)\epsilon)}
{2(1-(4k-5)\epsilon)}
\bigg{)}^{-1} 
\bigg{(} 
\frac{\epsilon}{2} + 
\frac{(2k-2)C^2\epsilon^2}{1-(4k-5)\epsilon)} 
\bigg{)}
\bigg{]}.
\end{align*}
Defining
\begin{align*}
C_1(\epsilon) &=
\frac{1}{2} + 
\bigg{(}\frac{C^2}{4(1-(2k-3)\epsilon)} \bigg{)}
\bigg{(}
1 - \frac{C^2\epsilon^2}{2} - 
\frac{(2(2k-2)C^2\epsilon^2)(1+2(4k-5)\epsilon)}
{2(1-(4k-5)\epsilon)}
\bigg{)}^{-1} \\
&\hspace{0.5cm}
\bigg{(} 
\frac{\epsilon}{2} + 
\frac{(2k-2)C^2\epsilon^2}{1-(4k-5)\epsilon)} 
\bigg{)}
\end{align*}
we have $\lim{\epsilon\rightarrow 0}C_1(\epsilon) = \frac{1}{2}$.

We define
\begin{align*}
C_2(\epsilon) &=
\bigg{(}\frac{C^2\epsilon}{4(1-(2k-3)\epsilon)} \bigg{)}
\bigg{(}
1 - \frac{C^2\epsilon^2}{2} - 
\frac{(2(2k-2)C^2\epsilon^2)(1+2(4k-5)\epsilon)}
{2(1-(4k-5)\epsilon)}
\bigg{)}^{-1} 
\bigg{(} 
\frac{1}{2\epsilon} + 2(2k-2) + \frac{1}{2}
\bigg{)} \\
&\hspace{0.5cm}+
\bigg{(}
\frac{(C^2\epsilon)(1+2(2k-3)\epsilon)}
{4(1-(2k-3)\epsilon)}
\bigg{)} \\
&=
\bigg{(}\frac{C^2}{4(1-(2k-3)\epsilon)} \bigg{)}
\bigg{(}
1 - \frac{C^2\epsilon^2}{2} - 
\frac{(2(2k-2)C^2\epsilon^2)(1+2(4k-5)\epsilon)}
{2(1-(4k-5)\epsilon)}
\bigg{)}^{-1} 
\bigg{(} 
\frac{1}{2} + 2(2k-2)\epsilon + \frac{\epsilon}{2}
\bigg{)} \\
&\hspace{0.5cm}+
\bigg{(}
\frac{(C^2\epsilon)(1+2(2k-3)\epsilon)}
{4(1-(2k-3)\epsilon)}
\bigg{)}.
\end{align*}
It is easy to see $\lim_{\epsilon \rightarrow 0}C_2(\epsilon) 
= \frac{C^2}{8}$. This finishes the proof.
\end{proof}

\begin{lem}\label{bilaplace_est_4}
Given $u \in W^{4,2}_{loc}(E) \cap L^2(E)$ and $\epsilon > 0$ sufficiently 
small, we have the following
estimate
\begin{align*}
\vert\langle \chi_{\epsilon}^{2k}\Delta^2u, 
\Delta_M(\chi_{\epsilon}^{2k})(\Delta u)\rangle\vert &\leq 
\bigg{[}
\frac{\epsilon}{2} + \bigg{(}\frac{C^2\epsilon}{2}\bigg{)}
\bigg{(}
1 - \frac{C^2\epsilon^2}{2} - 
\frac{(2(2k-2)C^2\epsilon^2)(1+2(4k-5)\epsilon)}
{2(1-(4k-5)\epsilon)}
\bigg{)}^{-1} \\
&\hspace{0.5cm}
\bigg{(} 
\frac{\epsilon}{2} + \frac{(2k-2)C^2\epsilon^2}
{(1-(4k-5)\epsilon)}
\bigg{)}
\bigg{]}\vert\vert\chi_{\epsilon}^{2k}\Delta^2u\vert\vert^2 \\
&+ 
\bigg{(}\frac{C^2\epsilon}{2}\bigg{)}
\bigg{(}
1 - \frac{C^2\epsilon^2}{2} - 
\frac{(2(2k-2)C^2\epsilon^2)(1+2(4k-5)\epsilon)}
{2(1-(4k-5)\epsilon)}
\bigg{)}^{-1} \\
&\hspace{0.5cm}
\bigg{(}
\frac{1}{2\epsilon} + 2(2k-2) + \frac{1}{2}
\bigg{)}\vert\vert u\vert\vert^2.
\end{align*}
\end{lem}

\begin{proof}
Applying Cauchy-Schwarz and Young's inequality gives
\begin{align*}
\vert\langle \chi_{\epsilon}^{2k}\Delta^2u, 
\Delta_M(\chi_{\epsilon}^{2k})(\Delta u)\rangle\vert &\leq 
\frac{\epsilon}{2}\vert\vert\chi_{\epsilon}^{2k}\Delta^2u\vert\vert^2 + 
\frac{1}{2\epsilon}\vert\vert(\Delta_M\chi_{\epsilon}^{2k})\Delta u\vert
\vert^2 \\
&\leq 
\frac{\epsilon}{2}\vert\vert\chi_{\epsilon}^{2k}\Delta^2u\vert\vert^2
+
\frac{C^2\epsilon}{2}\vert\vert\chi_{\epsilon}^{2k-2}\Delta u\vert\vert^2
\end{align*}
where to get the second inequality, we have applied 
corollary \ref{laplacian_prod_est_1} to estimate the 
$\Delta_M(\chi_{\epsilon}^{2k})$ term.

Applying proposition \ref{Laplace_est} to estimate the 
$\vert\vert\chi_{\epsilon}^{2k-2}\Delta u\vert\vert^2$ term gives
\begin{align*}
\vert\langle \chi_{\epsilon}^{2k}\Delta^2u, 
\Delta_M(\chi_{\epsilon}^{2k})(\Delta u)\rangle\vert &\leq 
\frac{\epsilon}{2}\vert\vert\chi_{\epsilon}^{2k}\Delta^2u\vert\vert^2 +
\bigg{(}\frac{C^2\epsilon}{2}\bigg{)}
\bigg{(}
1 - \frac{C^2\epsilon^2}{2} - 
\frac{(2(2k-2)C^2\epsilon^2)(1+2(4k-5)\epsilon)}
{2(1-(4k-5)\epsilon)}
\bigg{)}^{-1} \\
&\hspace{0.5cm}
\bigg{(}
\frac{\epsilon}{2} + 
\frac{(2k-2)C^2\epsilon^2}{(1-(4k-5)\epsilon)}
\bigg{)}\vert\vert\chi_{\epsilon}^{2(2k-2)}\Delta^2u\vert\vert^2 \\
&\hspace{0.5cm} +
\bigg{(}\frac{C^2\epsilon}{2}\bigg{)}
\bigg{(}
1 - \frac{C^2\epsilon^2}{2} - 
\frac{(2(2k-2)C^2\epsilon^2)(1+2(4k-5)\epsilon)}
{2(1-(4k-5)\epsilon)}
\bigg{)}^{-1} \\
&\hspace{0.5cm}
\bigg{(}
\frac{1}{2\epsilon} + 2(2k-2) + \frac{1}{2}
\bigg{)}\vert\vert u\vert\vert^2.
\end{align*}
Using the fact that $\chi_{\epsilon}^{2(2k-2)} \leq \chi_{\epsilon}^{2k}$
gives the statement of the lemma.
\end{proof}

\begin{cor}\label{bilaplace_est_4_a}
For $\epsilon > 0$ sufficiently small, we can write
\begin{equation*}
\vert\langle \chi_{\epsilon}^{2k}\Delta^2u, 
\Delta_M(\chi_{\epsilon}^{2k})(\Delta u)\rangle\vert \leq 
\epsilon C_1(\epsilon)\vert\vert\chi_{\epsilon}^{2k}\Delta^2u\vert\vert^2
+
C_2(\epsilon)\vert\vert u\vert\vert^2
\end{equation*}
where $C_1(\epsilon)$ and $C_2(\epsilon)$ are constants, depending 
on $\epsilon$, such that 
$\lim_{\epsilon \rightarrow 0} C_1(\epsilon) < \infty$ and 
$\lim_{\epsilon \rightarrow 0} C_2(\epsilon) < \infty$.
\end{cor}

The proof of this corollary follows exactly the same lines as the proof
of corollary \ref{bilaplace_est_3_a}.

\begin{lem}\label{bilaplace_est_5}
Given $\epsilon$, we have the following estimate
\begin{align*}
\vert\langle \chi_{\epsilon}^{2k}\Delta^2u, u\Delta_M^2\chi_{\epsilon}^{2k}\rangle\vert 
\leq
\frac{\epsilon}{2}\vert\vert\chi_{\epsilon}^{2k}\Delta^2u\vert\vert^2 +
\frac{C^2\epsilon}{2}\vert\vert u\vert\vert^2
\end{align*}
\end{lem}
\begin{proof}
Apply Cauchy-Schwarz and Young's inequality to obtain.
\begin{align*}
\vert\langle \chi_{\epsilon}^{2k}\Delta^2u, u\Delta_M^2\chi_{\epsilon}^{2k}\rangle\vert 
&\leq
\frac{\epsilon}{2}\vert\vert\chi_{\epsilon}^{2k}\Delta^2u\vert\vert^2 + 
\frac{1}{2\epsilon}\vert\vert u\Delta_M^2\chi_{\epsilon}^{2k}\vert\vert^2 \\
&\leq
\frac{\epsilon}{2}\vert\vert\chi_{\epsilon}^{2k}\Delta^2u\vert\vert^2 +
\frac{C^2\epsilon}{2}\vert\vert\chi_{\epsilon}^{2k-4} u\vert\vert^2 \\
&\leq 
\frac{\epsilon}{2}\vert\vert\chi_{\epsilon}^{2k}\Delta^2u\vert\vert^2 +
\frac{C^2\epsilon}{2}\vert\vert u\vert\vert^2
\end{align*}
where to get the second inequality we have applied proposition
\ref{bilaplace_prod_est_1}, and to get the third inequality we have
used the fact that $\chi_{\epsilon}^{2k-4} \leq 1$. 

\end{proof}

From this lemma, it is straightforward to see that we have the following
corollary.

\begin{cor}\label{bilaplace_est_5_a}
For $\epsilon > 0$ sufficiently small, we can write
\begin{equation*}
\vert\langle \chi_{\epsilon}^{2k}\Delta^2u, u\Delta_M^2\chi_{\epsilon}^{2k}
\rangle\vert \leq 
\epsilon C_1(\epsilon)\vert\vert\chi_{\epsilon}^{2k}\Delta^2u\vert\vert^2
+
C_2(\epsilon)\vert\vert u\vert\vert^2
\end{equation*}
where $C_1(\epsilon)$ and $C_2(\epsilon)$ are constants, depending 
on $\epsilon$, such that 
$\lim_{\epsilon \rightarrow 0} C_1(\epsilon) < \infty$ and 
$\lim_{\epsilon \rightarrow 0} C_2(\epsilon) < \infty$.
\end{cor}

The following proposition can be seen as a Bilaplacian version of
Milatovic's Laplacian estimate, given in lemma 3.6 of 
\cite{milatovic}.

\begin{prop}\label{bilaplace_est_6}
Given $u \in W^{4,2}_{loc}(E) \cap L^2(E)$ and $\epsilon > 0$ sufficiently 
small, we have the following
estimate
\begin{equation}
\vert\vert\chi_{\epsilon}^{2k}\Delta^2 u\vert\vert^2 \leq
\bigg{(}\frac{1}{\big{(}1-\epsilon C_1(\epsilon)\big{)}}\bigg{)}
\vert\vert\Delta^2(\chi_{\epsilon}^{2k}u)\vert\vert^2 +
\frac{C_2(\epsilon, \epsilon_1)}{\big{(}1-\epsilon C_1(\epsilon)\big{)}}
\vert\vert u\vert\vert^2
\end{equation}
where $C_1(\epsilon)$, $C_2(\epsilon)$ 
are constants depending on $\epsilon$ such that,
$\lim_{\epsilon \rightarrow 0} C_1(\epsilon) < \infty$ and
$\lim_{\epsilon \rightarrow 0} C_2(\epsilon) < \infty$.
\end{prop}

\begin{proof}
Using formula \eqref{bi-laplace_product} 
we expand $\langle \Delta^2(\chi_{\epsilon}^{2k}u), 
\Delta^2(\chi_{\epsilon}^{2k}u) \rangle$ and obtain
\begin{align*}
\langle \Delta^2(\chi_{\epsilon}^{2k}u), 
\Delta^2(\chi_{\epsilon}^{2k}u) \rangle &=
\bigg{\langle} 
\chi_{\epsilon}^{2k}\Delta^2u - 2\nabla_{(d\chi_{\epsilon}^{2k})^{\#}}\Delta u 
+ 2(\Delta_M\chi_{\epsilon}^{2k})(\Delta u) - 
2\Delta\nabla_{(d\chi_{\epsilon}^{2k})^{\#}}u \\
&\hspace{1cm} - 
2\nabla_{(d\Delta_M\chi_{\epsilon}^{2k})^{\#}}u + (\Delta_M^2\chi_{\epsilon}^{2k})u, 
\chi_{\epsilon}^{2k}\Delta^2u - 2\nabla_{(d\chi_{\epsilon}^{2k})^{\#}}\Delta u 
+ 2(\Delta_M\chi_{\epsilon}^{2k})(\Delta u) \\
&\hspace{1cm} - 
2\Delta\nabla_{(d\chi_{\epsilon}^{2k})^{\#}}u 
 - 
2\nabla_{(d\Delta_M\chi_{\epsilon}^{2k})^{\#}}u + (\Delta_M^2\chi_{\epsilon}^{2k})u
\bigg{\rangle} \\
&=
\langle \chi_{\epsilon}^{2k}\Delta^2u, \chi_{\epsilon}^{2k}\Delta^2u \rangle  
- 4Re\langle \chi_{\epsilon}^{2k}\Delta^2u, \nabla_{(d\chi_{\epsilon}^{2k})^{\#}}\Delta u \rangle 
-4Re \langle \chi_{\epsilon}^{2k}\Delta^2u, \Delta\nabla_{(d\chi_{\epsilon}^{2k})^{\#}}u\rangle \\
&\hspace{1cm}
-4Re \langle \chi_{\epsilon}^{2k}\Delta^2u, \nabla_{
(d\Delta_M\chi_{\epsilon}^{2k})^{\#}}u\rangle 
-4Re  \langle \chi_{\epsilon}^{2k}\Delta^2u, (\Delta_M\chi_{\epsilon}^{2k})(\Delta u) \rangle \\
&\hspace{1cm}
+ 2Re\langle \chi_{\epsilon}^{2k}\Delta^2u, u\Delta_M^2\chi_{\epsilon}^{2k}\rangle
+ 4\langle  \nabla_{(d\chi_{\epsilon}^{2k})^{\#}}\Delta u, \nabla_{(d\chi_{\epsilon}^{2k})^{\#}}\Delta u\rangle \\
&\hspace{1cm} 
+ 8Re\langle\nabla_{(d\chi_{\epsilon}^{2k})^{\#}}\Delta u, 
\Delta\nabla_{(d\chi_{\epsilon}^{2k})^{\#}}u\rangle +
8Re\langle\nabla_{(d\chi_{\epsilon}^{2k})^{\#}}\Delta u, 
\nabla_{(d\Delta_M\chi_{\epsilon}^{2k})^{\#}}u\rangle \\
&\hspace{1cm} +
8Re\langle\nabla_{(d\chi_{\epsilon}^{2k})^{\#}}\Delta u, 
(\Delta_M\chi_{\epsilon}^{2k})(\Delta u)\rangle -
4Re\langle\nabla_{(d\chi_{\epsilon}^{2k})^{\#}}\Delta u, u\Delta_M^2\chi_{\epsilon}^{2k}
\rangle \\
&\hspace{1cm} + 
4Re\langle\Delta\nabla_{(d\chi_{\epsilon}^{2k})^{\#}}u, 
\Delta\nabla_{(d\chi_{\epsilon}^{2k})^{\#}}u\rangle +
8Re\langle \Delta\nabla_{(d\chi_{\epsilon}^{2k})^{\#}}u, 
\nabla_{(d\Delta_M\chi_{\epsilon}^{2k})^{\#}}u \rangle \\
&\hspace{1cm} +
8Re\langle \Delta\nabla_{(d\chi_{\epsilon}^{2k})^{\#}}u, 
(\Delta_M\chi_{\epsilon}^{2k})(\Delta u) \rangle -
4Re\langle \Delta\nabla_{(d\chi_{\epsilon}^{2k})^{\#}}u, 
u\Delta_M^2\chi_{\epsilon}^{2k}\rangle \\
&\hspace{1cm} +  
4\langle \nabla_{(d\Delta_M\chi_{\epsilon}^{2k})^{\#}}u, 
\nabla_{(d\Delta_M\chi_{\epsilon}^{2k})^{\#}}u\rangle +
8Re\langle \nabla_{(d\Delta_M\chi_{\epsilon}^{2k})^{\#}}u, 
(\Delta_M\chi_{\epsilon}^{2k})\Delta u\rangle \\
&\hspace{1cm} -
4Re\langle \nabla_{(d\Delta_M\chi_{\epsilon}^{2k})^{\#}}u, 
u\Delta_M^2(\chi_{\epsilon}^{2k})\rangle +
4\langle (\Delta_M\chi_{\epsilon}^{2k})(\Delta u), 
(\Delta_M\chi_{\epsilon}^{2k})(\Delta u)
\rangle \\
&\hspace{1cm} -
4Re\langle (\Delta_M\chi_{\epsilon}^{2k})(\Delta u), u\Delta_M^2\chi_{\epsilon}^{2k}\rangle 
+
\langle u\Delta_M^2\chi_{\epsilon}^{2k}, u\Delta_M^2\chi_{\epsilon}^{2k}\rangle 
\end{align*}
This implies we can write
\begin{align}
\langle \chi_{\epsilon}^{2k}\Delta^2u, \chi_{\epsilon}^{2k}\Delta^2u \rangle  
&=
\langle \Delta^2(\chi_{\epsilon}^{2k}u), \Delta^2(\chi_{\epsilon}^{2k}u) \rangle \label{bilaplace_est_6_0}\\
&\hspace{1cm}+ 
4Re\langle \chi_{\epsilon}^{2k}\Delta^2u, \nabla_{(d\chi_{\epsilon}^{2k})^{\#}}\Delta u \rangle 
+4Re \langle \chi_{\epsilon}^{2k}\Delta^2u, \Delta\nabla_{(d\chi_{\epsilon}^{2k})^{\#}}u\rangle \nonumber\\
&\hspace{1cm}
+4Re \langle \chi_{\epsilon}^{2k}\Delta^2u, \nabla_{
(d\Delta_M\chi_{\epsilon}^{2k})^{\#}}u\rangle 
+4Re  \langle \chi_{\epsilon}^{2k}\Delta^2u, (\Delta_M\chi_{\epsilon}^{2k})(\Delta u) \rangle \nonumber\\
&\hspace{1cm}
- 2Re\langle \chi_{\epsilon}^{2k}\Delta^2u, u\Delta_M^2\chi_{\epsilon}^{2k}\rangle
- 4\langle  \nabla_{(d\chi_{\epsilon}^{2k})^{\#}}\Delta u, \nabla_{(d\chi_{\epsilon}^{2k})^{\#}}\Delta u\rangle \nonumber\\
&\hspace{1cm} 
- 8Re\langle\nabla_{(d\chi_{\epsilon}^{2k})^{\#}}\Delta u, 
\Delta\nabla_{(d\chi_{\epsilon}^{2k})^{\#}}u\rangle 
-8Re\langle\nabla_{(d\chi_{\epsilon}^{2k})^{\#}}\Delta u, 
\nabla_{(d\Delta_M\chi_{\epsilon}^{2k})^{\#}}u\rangle \nonumber\\ 
&\hspace{1cm} 
-8Re\langle\nabla_{(d\chi_{\epsilon}^{2k})^{\#}}\Delta u, 
(\Delta_M\chi_{\epsilon}^{2k})(\Delta u)\rangle 
+4Re\langle\nabla_{(d\chi_{\epsilon}^{2k})^{\#}}\Delta u, u\Delta_M^2\chi_{\epsilon}^{2k} 
\rangle \nonumber\\
&\hspace{1cm}  
-4Re\langle\Delta\nabla_{(d\chi_{\epsilon}^{2k})^{\#}}u, 
\Delta\nabla_{(d\chi_{\epsilon}^{2k})^{\#}}u\rangle 
-8Re\langle \Delta\nabla_{(d\chi_{\epsilon}^{2k})^{\#}}u,
\nabla_{(d\Delta_M\chi_{\epsilon}^{2k})^{\#}}u \rangle \nonumber\\
&\hspace{1cm} 
-8Re\langle \Delta\nabla_{(d\chi_{\epsilon}^{2k})^{\#}}u, 
(\Delta_M\chi_{\epsilon}^{2k})(\Delta u) \rangle 
+4Re\langle \Delta\nabla_{(d\chi_{\epsilon}^{2k})^{\#}}u, 
u\Delta_M^2\chi_{\epsilon}^{2k}\rangle \nonumber\\
&\hspace{1cm} 
-4\langle \nabla_{(d\Delta_M\chi_{\epsilon}^{2k})^{\#}}u, 
\nabla_{(d\Delta_M\chi_{\epsilon}^{2k})^{\#}}u\rangle 
-8Re\langle \nabla_{(d\Delta_M\chi_{\epsilon}^{2k})^{\#}}u, 
(\Delta_M\chi_{\epsilon}^{2k})\Delta u\rangle \nonumber\\
&\hspace{1cm} 
+4Re\langle \nabla_{(d\Delta_M\chi_{\epsilon}^{2k})^{\#}}u, 
u\Delta_M^2(\chi_{\epsilon}^{2k})\rangle 
-4\langle (\Delta_M\chi_{\epsilon}^{2k})(\Delta u), 
(\Delta_M\chi_{\epsilon}^{2k})(\Delta u)
\rangle \nonumber\\ 
&\hspace{1cm} 
+4Re\langle (\Delta_M\chi_{\epsilon}^{2k})(\Delta u), u\Delta_M^2\chi_{\epsilon}^{2k}\rangle 
-\langle u\Delta_M^2\chi_{\epsilon}^{2k}, u\Delta_M^2\chi_{\epsilon}^{2k}\rangle \nonumber
\end{align}

Using corollaries \ref{bilaplace_est_1_a}, \ref{bilaplace_est_2_a}, 
\ref{bilaplace_est_3_a}, \ref{bilaplace_est_4_a}, 
and \ref{bilaplace_est_5_a}.
The terms

\begin{align}
&4Re\langle \chi_{\epsilon}^{2k}\Delta^2u, \nabla_{(d\chi_{\epsilon}
^{2k})^{\#}}\Delta u \rangle \\
&4Re \langle \chi_{\epsilon}^{2k}\Delta^2u, \Delta\nabla_{(d\chi_{\epsilon}^
{2k})^{\#}}u\rangle \\
&4Re \langle \chi_{\epsilon}^{2k}\Delta^2u, \nabla_{
(d\Delta_M\chi_{\epsilon}
^{2k})^{\#}}u\rangle \\
&4Re  \langle \chi_{\epsilon}^{2k}\Delta^2u, (\Delta_M\chi_{\epsilon}^{2k})
(\Delta u) \rangle \\
&- 2Re\langle \chi_{\epsilon}^{2k}\Delta^2u, u\Delta_M^2\chi_{\epsilon}^{2k}
\rangle
\end{align}
can all be bounded above by 
\begin{align*}
\epsilon C_1(\epsilon)\vert\vert\chi_{\epsilon}^{2k}\Delta^2u\vert\vert^2 
+
C_2(\epsilon)\vert\vert u\vert\vert^2
\end{align*}
where $\lim_{\epsilon \rightarrow 0}C_1(\epsilon) < \infty$, and
$\lim_{\epsilon \rightarrow 0}C_2(\epsilon) < \infty$.

The terms 

\begin{align}
&- 4\langle  \nabla_{(d\chi_{\epsilon}^{2k})^{\#}}\Delta u, \nabla_{(d
\chi_{\epsilon}^{2k})^{\#}}\Delta u\rangle \label{bilaplace_est_6_a} \\
&-4Re\langle\Delta\nabla_{(d\chi_{\epsilon}^{2k})^{\#}}u, 
\Delta\nabla_{(d\chi_{\epsilon}^{2k})^{\#}}u\rangle 
\label{bilaplace_est_6_b} \\
&-4\langle \nabla_{(d\Delta_M\chi_{\epsilon}^{2k})^{\#}}u, 
\nabla_{(d\Delta_M\chi_{\epsilon}^{2k})^{\#}}u\rangle 
\label{bilaplace_est_6_c} \\
&-4\langle (\Delta_M\chi_{\epsilon}^{2k})(\Delta u), 
(\Delta_M\chi_{\epsilon}
^{2k})(\Delta u)\rangle \label{bilaplace_est_6_d} \\
&2\langle u\Delta_M^2\chi_{\epsilon}^{2k}, u\Delta_M^2\chi_{\epsilon}^{2k}
\rangle \label{bilaplace_est_6_e}
\end{align}
can all be bounded above by
\begin{align*}
\epsilon C_1(\epsilon)\vert\vert\chi_{\epsilon}^{2k}\Delta^2u\vert\vert^2 + 
C_2(\epsilon)\vert\vert u\vert\vert^2
\end{align*}
where $\lim_{\epsilon \rightarrow 0}C_1(\epsilon) < \infty$, and
$\lim_{\epsilon \rightarrow 0}C_2(\epsilon) < \infty$.

The proof of this follows exactly how we proved 
lemmas \ref{bilaplacian_est_1} to \ref{bilaplace_est_5}. We will outline
how to do the case of 
$-4Re\langle\Delta\nabla_{(d\chi_{\epsilon}^{2k})^{\#}}u, 
\Delta\nabla_{(d\chi_{\epsilon}^{2k})^{\#}}u\rangle$. 

We have
\begin{align*}
\vert\vert\Delta\nabla_{(d\chi_{\epsilon}^{2k})^{\#}}u\vert\vert^2  &= 
\vert\vert \nabla_{(d\chi_{\epsilon}^{2k})^{\#}}\Delta u + 
\nabla_{(d\chi_{\epsilon}^{2k})^{\#}}(Rm + F) * u + 
(Rm + F)*\nabla_{(d\chi_{\epsilon}^{2k})^{\#}}u\vert\vert^2 \\
&\leq C\bigg{(}
\vert\vert \nabla_{(d\chi_{\epsilon}^{2k})^{\#}}\Delta u\vert\vert^2 +
\vert\vert \nabla_{(d\chi_{\epsilon}^{2k})^{\#}}(Rm + F) * u\vert\vert^2
+ \vert\vert (Rm + F)*\nabla_{(d\chi_{\epsilon}^{2k})^{\#}}u\vert\vert^2
\bigg{)} \\
&=
C\bigg{(}
\vert\vert 
2k\chi_{\epsilon}^{2k-1} \nabla_{(d\chi_{\epsilon})^{\#}}\Delta u
\vert\vert^2 +
\vert\vert 2k\chi_{\epsilon}^{2k-1}
\nabla_{(d\chi_{\epsilon})^{\#}}(Rm + F) * u\vert\vert^2 \\
&\hspace{1cm}+
\vert\vert 2k\chi_{\epsilon}^{2k-1}
(Rm + F)*\nabla_{(d\chi_{\epsilon})^{\#}}u\vert\vert^2
\bigg{)} \\
&\leq 
4k^2C^2\epsilon^2\vert\vert 
\chi_{\epsilon}^{2k-1} \nabla\Delta u\vert\vert^2 +
4k^2C^2\epsilon^2\vert\vert\chi_{\epsilon}^{2k-1}u\vert\vert^2 + 
4k^2C^2\epsilon^2\vert\vert\nabla u\vert\vert^2
\end{align*}
where to get the first equality we have applied the commutation formula
given by corollary \ref{connection_3}, and to get the last inequality 
we have used \eqref{cut-off_bound_3} and our bounded geometry assumptions
(A1) and (A2).

The way to proceed now is to estimate the term 
$\vert\vert \chi_{\epsilon}^{2k-1} \nabla\Delta u\vert\vert^2$ and
$\vert\vert\nabla u\vert\vert^2$, using proposition \ref{covariant_est_1}.
In doing so, we obtain
\begin{align*}
\vert\vert\Delta\nabla_{(d\chi_{\epsilon}^{2k})^{\#}}u\vert\vert^2 &\leq
\frac{4k^2C^2\epsilon^2}{2(1-(2k-1)\epsilon)}
\vert\vert\chi_{\epsilon}^{2k}\Delta^2u\vert\vert^2 + 
\frac{(4k^2C^2\epsilon^2)(1+(2k-1)\epsilon)}
{2(1-(2k-1)\epsilon)}\vert\vert
\chi_{\epsilon}^{2k-2}\Delta u\vert\vert^2 \\
&\hspace{0.5cm} + 
\frac{4k^2C^2\epsilon^2}{2(1-(2k-1)\epsilon)}
\vert\vert\chi_{\epsilon}^{2k}\Delta u\vert\vert^2 +
\frac{(4k^2C^2\epsilon^2)(1+(2k-1)\epsilon)}
{2(1-(2k-1)\epsilon)}\vert\vert
\chi_{\epsilon}^{2k-2}u\vert\vert^2 \\
&\hspace{0.5cm} +
4k^2C^2\epsilon^2\vert\vert\chi_{\epsilon}^{2k-1}u\vert\vert^2 \\
&\leq
\frac{4k^2C^2\epsilon^2}{2(1-(2k-1)\epsilon)}
\vert\vert\chi_{\epsilon}^{2k}\Delta^2u\vert\vert^2 \\
&\hspace{0.5cm}+ 
\bigg{[}
\bigg{(}\frac{(4k^2C^2\epsilon^2)(1+(2k-1)\epsilon)}
{2(1-(2k-1)\epsilon)} \bigg{)}
+
\bigg{(}
\frac{4k^2C^2\epsilon^2}{2(1-(2k-1)\epsilon)} 
\bigg{)}
\bigg{]}\vert\vert \chi_{\epsilon}^{2k-2}\Delta u\vert\vert^2 \\
&\hspace{0.5cm}+ 
\bigg{[}
\bigg{(}
\frac{(4k^2C^2\epsilon^2)(1+(2k-1)\epsilon)}
{2(1-(2k-1)\epsilon)}
\bigg{)}
+ 
4k^2C^2\epsilon^2
\bigg{]}\vert\vert u\vert\vert^2
\end{align*}
where to get the second inequality we have used the fact that
$\chi_{\epsilon}^{2k} \leq \chi_{\epsilon}^{2k-2} \leq 1$.

Using proposition \ref{Laplace_est}, we
estimate $\vert\vert\chi_{\epsilon}^{2k}\Delta u\vert\vert^2$
and $\vert\vert\chi_{\epsilon}^{2k}\Delta u\vert\vert^2$ to obtain
\begin{align*}
\vert\vert\Delta\nabla_{(d\chi_{\epsilon}^{2k})^{\#}}u\vert\vert^2 &\leq
\frac{4k^2C^2\epsilon^2}{2(1-(2k-1)\epsilon)}
\vert\vert\chi_{\epsilon}^{2k}\Delta^2u\vert\vert^2 
+
\bigg{[}
\bigg{(}\frac{(4k^2C^2\epsilon^2)(1+(2k-1)\epsilon)}
{2(1-(2k-1)\epsilon)} \bigg{)}
+
\bigg{(}
\frac{4k^2C^2\epsilon^2}{2(1-(2k-1)\epsilon)} 
\bigg{)}
\bigg{]} \\
&\hspace{0.5cm}
\bigg{(}
1 - \frac{C^2\epsilon^2}{2} - 
\frac{(2(2k-2)C^2\epsilon^2)(1+2(4k-5)\epsilon)}
{2(1-(4k-5)\epsilon)}
\bigg{)}^{-1} \\
&\hspace{0.5cm}
\bigg{(}
\frac{\epsilon}{2} + \frac{(2k-2)C^2\epsilon^2}
{1-(4k-5)\epsilon}
\bigg{)}\vert\vert\chi_{\epsilon}^{2(2k-2)}\Delta^2u\vert\vert^2 \\
&\hspace{0.5cm} +
\bigg{[}
\bigg{(}\frac{(4k^2C^2\epsilon^2)(1+(2k-1)\epsilon)}
{2(1-(2k-1)\epsilon)} \bigg{)}
+
\bigg{(}
\frac{4k^2C^2\epsilon^2}{2(1-(2k-1)\epsilon)} 
\bigg{)}
\bigg{]} \\
&\hspace{0.5cm}
\bigg{(}
1 - \frac{C^2\epsilon^2}{2} - 
\frac{(2(2k-2)C^2\epsilon^2)(1+2(4k-5)\epsilon)}
{2(1-(4k-5)\epsilon)}
\bigg{)}^{-1}
\bigg{(}
\frac{1}{2\epsilon} + 2(2k-2) + \frac{1}{2}
\bigg{)}\vert\vert u\vert\vert^2 \\
&\leq
\bigg{[}
\bigg{(}
\frac{4k^2C^2\epsilon^2}{2(1-(2k-1)\epsilon)}
\bigg{)} +
\bigg{[}
\bigg{(}\frac{(4k^2C^2\epsilon^2)(1+(2k-1)\epsilon)}
{2(1-(2k-1)\epsilon)} \bigg{)}
+
\bigg{(}
\frac{4k^2C^2\epsilon^2}{2(1-(2k-1)\epsilon)} 
\bigg{)}
\bigg{]} \\
&\hspace{0.5cm}
\bigg{(}
1 - \frac{C^2\epsilon^2}{2} - 
\frac{(2(2k-2)C^2\epsilon^2)(1+2(4k-5)\epsilon)}
{2(1-(4k-5)\epsilon)}
\bigg{)}^{-1} \\
&\hspace{0.5cm}
\bigg{(}
\frac{\epsilon}{2} + \frac{(2k-2)C^2\epsilon^2}
{1-(4k-5)\epsilon}
\bigg{)}
\bigg{]}\vert\vert\chi_{\epsilon}^{2k}\Delta^2u\vert\vert^2 \\
&\hspace{0.5cm}+
\Bigg{[}
\bigg{[}
\bigg{(}\frac{(4k^2C^2\epsilon^2)(1+(2k-1)\epsilon)}
{2(1-(2k-1)\epsilon)} \bigg{)}
+
\bigg{(}
\frac{4k^2C^2\epsilon^2}{2(1-(2k-1)\epsilon)} 
\bigg{)}
\bigg{]} \\
&\hspace{0.5cm}
\bigg{(}
1 - \frac{C^2\epsilon^2}{2} - 
\frac{(2(2k-2)C^2\epsilon^2)(1+2(4k-5)\epsilon)}
{2(1-(4k-5)\epsilon)}
\bigg{)}^{-1}
\bigg{(}
\frac{1}{2\epsilon} + 2(2k-2) + \frac{1}{2}
\bigg{)} \\
&\hspace{0.5cm}+
\bigg{[}
\bigg{(}
\frac{(4k^2C^2\epsilon^2)(1+(2k-1)\epsilon)}
{2(1-(2k-1)\epsilon)}
\bigg{)}
+ 
4k^2C^2\epsilon^2
\bigg{]}
\Bigg{]}\vert\vert u\vert\vert^2
\end{align*}
where to get the second inequality we have used the fact that
$\chi_{\epsilon}^{2(2k-2)} \leq \chi_{\epsilon}^{2k}$.

We then observe that we can write the coefficient of the 
$\vert\vert\chi_{\epsilon}^{2k}\Delta^2u\vert\vert^2$ term as
\begin{align*}
&\bigg{[}
\bigg{(}
\frac{4k^2C^2\epsilon^2}{2(1-(2k-1)\epsilon)}
\bigg{)} +
\bigg{[}
\bigg{(}\frac{(4k^2C^2\epsilon^2)(1+(2k-1)\epsilon)}
{2(1-(2k-1)\epsilon)} \bigg{)}
+
\bigg{(}
\frac{4k^2C^2\epsilon^2}{2(1-(2k-1)\epsilon)} 
\bigg{)}
\bigg{]} \\
&\hspace{0.5cm}
\bigg{(}
1 - \frac{C^2\epsilon^2}{2} - 
\frac{(2(2k-2)C^2\epsilon^2)(1+2(4k-5)\epsilon)}
{2(1-(4k-5)\epsilon)}
\bigg{)}^{-1} 
\bigg{(}
\frac{\epsilon}{2} + \frac{(2k-2)C^2\epsilon^2}
{1-(4k-5)\epsilon}
\bigg{)}
\bigg{]} \\
&= \\
&\epsilon\bigg{[}
\bigg{(}
\frac{4k^2C^2\epsilon}{2(1-(2k-1)\epsilon)}
\bigg{)} +
\bigg{[}
\bigg{(}\frac{(4k^2C^2\epsilon)(1+(2k-1)\epsilon)}
{2(1-(2k-1)\epsilon)} \bigg{)}
+
\bigg{(}
\frac{4k^2C^2\epsilon}{2(1-(2k-1)\epsilon)} 
\bigg{)}
\bigg{]} \\
&\hspace{0.5cm}
\bigg{(}
1 - \frac{C^2\epsilon^2}{2} - 
\frac{(2(2k-2)C^2\epsilon^2)(1+2(4k-5)\epsilon)}
{2(1-(4k-5)\epsilon)}
\bigg{)}^{-1} 
\bigg{(}
\frac{\epsilon}{2} + \frac{(2k-2)C^2\epsilon^2}
{1-(4k-5)\epsilon}
\bigg{)}
\bigg{]}.
\end{align*}
We then define
\begin{align*}
C_1(\epsilon) &=
\bigg{[}
\bigg{(}
\frac{4k^2C^2\epsilon}{2(1-(2k-1)\epsilon)}
\bigg{)} +
\bigg{[}
\bigg{(}\frac{(4k^2C^2\epsilon)(1+(2k-1)\epsilon)}
{2(1-(2k-1)\epsilon)} \bigg{)}
+
\bigg{(}
\frac{4k^2C^2\epsilon}{2(1-(2k-1)\epsilon)} 
\bigg{)}
\bigg{]} \\
&\hspace{0.5cm}
\bigg{(}
1 - \frac{C^2\epsilon^2}{2} - 
\frac{(2(2k-2)C^2\epsilon^2)(1+2(4k-5)\epsilon)}
{2(1-(4k-5)\epsilon)}
\bigg{)}^{-1} 
\bigg{(}
\frac{\epsilon}{2} + \frac{(2k-2)C^2\epsilon^2}
{1-(4k-5)\epsilon}
\bigg{)}
\bigg{]}.
\end{align*}
It is easy to see $\lim_{\epsilon \rightarrow 0} C_1(\epsilon) = 0$.

We define
\begin{align*}
C_2(\epsilon) &=
\Bigg{[}
\bigg{[}
\bigg{(}\frac{(4k^2C^2\epsilon^2)(1+(2k-1)\epsilon)}
{2(1-(2k-1)\epsilon)} \bigg{)}
+
\bigg{(}
\frac{4k^2C^2\epsilon^2}{2(1-(2k-1)\epsilon)} 
\bigg{)}
\bigg{]} \\
&\hspace{0.5cm}
\bigg{(}
1 - \frac{C^2\epsilon^2}{2} - 
\frac{(2(2k-2)C^2\epsilon^2)(1+2(4k-5)\epsilon)}
{2(1-(4k-5)\epsilon)}
\bigg{)}^{-1}
\bigg{(}
\frac{1}{2\epsilon} + 2(2k-2) + \frac{1}{2}
\bigg{)} \\
&\hspace{0.5cm}+
\bigg{[}
\bigg{(}
\frac{(4k^2C^2\epsilon^2)(1+(2k-1)\epsilon)}
{2(1-(2k-1)\epsilon)}
\bigg{)}
+ 
4k^2C^2\epsilon^2
\bigg{]}
\Bigg{]}.
\end{align*}
It is also easy to see that 
$\lim_{\epsilon \rightarrow 0} C_2(\epsilon) = 0$.

A similar approach can be used to establish the required estimates
for \eqref{bilaplace_est_6_a}, 
\eqref{bilaplace_est_6_c}, \eqref{bilaplace_est_6_d}, 
\eqref{bilaplace_est_6_e}.


The next step is to look at the ten terms
\begin{align}
&- 8Re\langle\nabla_{(d\chi_{\epsilon}^{2k})^{\#}}\Delta u, 
\Delta\nabla_{(d\chi_{\epsilon}^{2k})^{\#}}u\rangle \\
&-8Re\langle\nabla_{(d\chi_{\epsilon}^{2k})^{\#}}\Delta u, 
\nabla_{(d\Delta_M\chi_{\epsilon}^{2k})^{\#}}u\rangle \\
&-8Re\langle\nabla_{(d\chi_{\epsilon}^{2k})^{\#}}\Delta u,
(\Delta_M\chi_{\epsilon}^{2k})(\Delta u)\rangle  \\
&4Re\langle\nabla_{(d\chi_{\epsilon}^{2k})^{\#}}\Delta u, u
\Delta_M^2\chi_{\epsilon}^{2k} \\
&-8Re\langle \Delta\nabla_{(d\chi_{\epsilon}^{2k})^{\#}}u, 
\nabla_{(d\Delta_M\chi_{\epsilon}^{2k})^{\#}}u \rangle \\
&-8Re\langle \Delta\nabla_{(d\chi_{\epsilon}^{2k})^{\#}}u, 
(\Delta_M\chi_{\epsilon}^{2k})(\Delta u) \rangle \\
&4Re\langle \Delta\nabla_{(d\chi_{\epsilon}^{2k})^{\#}}u, 
u\Delta_M^2\chi_{\epsilon}^{2k}\rangle \\
&-8Re\langle \nabla_{(d\Delta_M\chi_{\epsilon}^{2k})^{\#}}u, 
(\Delta_M\chi_{\epsilon}^{2k})\Delta u\rangle \\
&4Re\langle \nabla_{(d\Delta_M\chi_{\epsilon}^{2k})^{\#}}u, 
u\Delta_M^2(\chi_{\epsilon}^{2k})\rangle \\
&-4\langle (\Delta_M\chi_{\epsilon}^{2k})(\Delta u), 
(\Delta_M\chi_{\epsilon}^{2k})
(\Delta u)\rangle. 
\end{align}
These can all also be bounded above by a term of the form
$\epsilon C_1(\epsilon)
\vert\vert \chi_{\epsilon}^{2k}\Delta^2u\vert\vert^2 + 
C_2(\epsilon)\vert\vert u\vert\vert^2$, with 
$\lim_{\epsilon\rightarrow 0}C_1(\epsilon) < \infty$ and
$\lim_{\epsilon\rightarrow 0}C_2(\epsilon) < \infty$.

To see this, one applies
Cauchy-Schwarz together with Young's inequality, and then uses
the fact that we obtained such a bound for each of
\eqref{bilaplace_est_6_a}, \eqref{bilaplace_est_6_b}, 
\eqref{bilaplace_est_6_c}, \eqref{bilaplace_est_6_d}, 
\eqref{bilaplace_est_6_e}.

Let us give an example of how to do this with the term 
$-8Re\langle \nabla_{(d\chi_{\epsilon}^{2k})^{\#}}\Delta u, 
\Delta\nabla_{(d\chi_{\epsilon}^{2k})^{\#}}u \rangle$.

Start by applying Cauchy-Schwartz and Youngs inequality to obtain
\begin{align*}
8\vert\langle \nabla_{(d\chi_{\epsilon}^{2k})^{\#}}\Delta u, 
\Delta\nabla_{(d\chi_{\epsilon}^{2k})^{\#}}u \rangle\vert \leq 
{4}\vert\vert \nabla_{(d\chi_{\epsilon}^{2k})^{\#}}\Delta u\vert\vert^2
+ 4\vert\vert\Delta\nabla_{(d\chi_{\epsilon}^{2k})^{\#}}u\vert\vert^2.
\end{align*}

We then see that 
$\vert\vert \nabla_{(d\chi_{\epsilon}^{2k})^{\#}}\Delta u\vert\vert^2$
and 
$\vert\vert\Delta\nabla_{(d\chi_{\epsilon}^{2k})^{\#}}u\vert\vert^2$
are the two terms from \eqref{bilaplace_est_6_a} and 
\eqref{bilaplace_est_6_b} respectively, which can be bounded above
by $\epsilon C_1(\epsilon)
\vert\vert \chi_{\epsilon}^{2k}\Delta^2u\vert\vert^2 + 
C_2(\epsilon)\vert\vert u\vert\vert^2$. Hence we obtain
\begin{align*}
8\vert\langle \nabla_{(d\chi_{\epsilon}^{2k})^{\#}}\Delta u, 
\Delta\nabla_{(d\chi_{\epsilon}^{2k})^{\#}}u \rangle\vert &\leq
8\epsilon C_1(\epsilon)
\vert\vert \chi_{\epsilon}^{2k}\Delta^2u\vert\vert^2 + 
8C_2(\epsilon)\vert\vert u\vert\vert^2 \\
&=
\epsilon C_1(\epsilon)
\vert\vert \chi_{\epsilon}^{2k}\Delta^2u\vert\vert^2 + 
C_2(\epsilon)\vert\vert u\vert\vert^2
\end{align*}
where to get the second equality we are simply absorbing the 8 into the 
constant $C_1(\epsilon)$ and $C_2(\epsilon)$, and denoting these new
constants again by $C_1(\epsilon)$ and $C_2(\epsilon)$. This shows
what we wanted to show.

Putting all this together into \eqref{bilaplace_est_6_0}, we obtain
\begin{equation*}
\vert\vert\chi_{\epsilon}^{2k}\Delta^2 u\vert\vert^2 \leq 
\vert\vert\Delta^2(\chi_{\epsilon}^{2k}u)\vert\vert^2 +
\epsilon C_1(\epsilon)
\vert\vert\chi_{\epsilon}^{2k}\Delta^2 u\vert\vert^2  +
C_2(\epsilon)
\vert\vert u\vert\vert^2
\end{equation*}
which implies
\begin{equation*}
\big{(}1-\epsilon C_1(\epsilon)\big{)}
\vert\vert\chi_{\epsilon}^{2k}\Delta^2 u\vert\vert^2 \leq 
\vert\vert\Delta^2(\chi_{\epsilon}^{2k}u)\vert\vert^2 +
C_2(\epsilon)
\vert\vert u\vert\vert^2.
\end{equation*}
Choosing $\epsilon$ small enough so that $1-\epsilon C_1(\epsilon) > 0$, 
we obtain
\begin{equation*}
\vert\vert\chi_{\epsilon}^{2k}\Delta^2 u\vert\vert^2 \leq 
\frac{1}{\big{(}1-\epsilon C_1(\epsilon)\big{)}}
\vert\vert\Delta^2(\chi_{\epsilon}^{2k}u)\vert\vert^2 +
\frac{C_2(\epsilon)}{\big{(}1-\epsilon C_1(\epsilon)\big{)}}
\vert\vert u\vert\vert^2
\end{equation*}
which establishes the proposition.

\end{proof}

\section{Proof of theorem \ref{main_theorem_1}}
\label{proof_main_theorem_1}

In this section we prove theorem \ref{main_theorem_1}. We will start
with an important lemma, and then move on to the proof of the theorem.

Let $T = \Delta^4 + V$ be as in the statement of theorem 
\ref{main_theorem_1}. We define the minimal operator associated 
to $T$ by $T_{min}u := Tu$ with domain $D_{min} = C_c^{\infty}(E)$.
We then define the maximal operator associated to $T$ as the adjoint
of the minimal operator. That is, $T_{max} := (T_{min})^*$, where
for a linear densely defined operator $L$, we let $L^*$ denote the 
adjoint. The domain of the operator $T_{max}$ can be defined
distributionally as
\begin{equation}
D_{max} = \{u \in L^2(E) : Tu \in L^2(E)\} \label{max_domain}
\end{equation}
where $Tu$ is to be understood in the distributional sense, and
we have that $T_{max}u := Tu$ for $u \in D_{max}$.

The following lemma can be seen as a Bilaplacian version of 
Milatovic's lemma 4.1 in \cite{milatovic}.

\begin{lem}\label{main_lemma_1}
Assume that $V$ satisfies the hypotheses of 
theorem \ref{main_theorem_1}. Assume $u \in Dom(T_{max})$ and 
$T_{max}u = i\lambda u$, for some $\lambda \in \R$. Then given 
$\epsilon > 0$ sufficiently small, we have the following estimate
\begin{equation*}
\vert\vert \Delta^2(\chi_{\epsilon}^{2k}u)\vert\vert^2 \leq 
\frac{C_2(\epsilon)}{1-2\epsilon C_1(\epsilon)}\vert\vert u\vert\vert^2
+ 
2
\langle (q\circ r)(u), \chi_{\epsilon}^{4k}u\rangle
\end{equation*}
where $C_1(\epsilon)$ and $C_2(\epsilon)$ are constants depending on 
$\epsilon$ such that
$\lim_{\epsilon\rightarrow 0}C_1(\epsilon) < \infty$ and
$\lim_{\epsilon\rightarrow 0}C_2(\epsilon) < \infty$.
\end{lem}

\begin{proof}
$T_{\max}u = i\lambda u$ implies $\Delta^4u + Vu = i\lambda u$. As 
$V \in L^{\infty}_{loc}(End\mathcal{E})$ and $u \in L^2(E)$, elliptic regularity, see theorem 10.3.6 in \cite{nicolaescu}, implies $u \in W^{8,2}_{loc}(E)$. Thus, integrating by
parts gives 
\begin{align*}
i\lambda\langle u, \chi_{\epsilon}^{4k}u\rangle &= 
\langle \Delta^4u + Vu, \chi_{\epsilon}^{4k}u\rangle \\
&= \langle \Delta^4u, \chi_{\epsilon}^{4k}u\rangle + 
\langle Vu, \chi_{\epsilon}^{4k}u\rangle \\
&= \langle \Delta^2u, \Delta^2(\chi_{\epsilon}^{4k}u)\rangle + 
\langle Vu, \chi_{\epsilon}^{4k}u\rangle.
\end{align*}

Using formula \eqref{bi-laplace_product}, we can write
\begin{align*}
\Delta^2(\chi_{\epsilon}^{4k}u) = 
\chi_{\epsilon}^{4k}\Delta^2u &- 2\nabla_{(d\chi_{\epsilon}^{4k})^{\#}}\Delta u +
 2\Delta_M(\chi_{\epsilon}^{4k})\Delta u - 2\Delta\nabla_{(d\chi_{\epsilon}^{4k})^{\#}}u 
\\
&
 - 2\nabla_{(d\Delta_M(\chi_{\epsilon}^{4k}))^{\#}}u + 
 (\Delta_M^2\chi_{\epsilon}^{4k})u. 
\end{align*}
Substituting this into the above, we get
\begin{align}
i\lambda\langle u, \chi_{\epsilon}^{4k}u\rangle = 
\langle \Delta^2u, \chi_{\epsilon}^{4k}\Delta^2u\rangle &- 
2\langle \Delta^2u, \Delta\nabla_{(d\chi_{\epsilon}^{4k})^{\#}}u\rangle 
-2\langle \Delta^2u, \nabla_{(d\chi_{\epsilon}^{4k})^{\#}}\Delta u\rangle 
\label{main_lemma_1_eqn} \\
&+
2\langle \Delta^2u, \Delta_M(\chi_{\epsilon}^{4k})\Delta u \rangle 
-2\langle \Delta^2u, \nabla_{(d\Delta_M(\chi_{\epsilon}^{4k}))^{\#}}u \rangle 
\nonumber \\
&+ 
\langle \Delta^2u, (\Delta_M^2\chi_{\epsilon}^{4k})u\rangle + 
\langle Vu, \chi_{\epsilon}^{4k}u\rangle. \nonumber
\end{align}

Taking real parts of the above equation gives
\begin{align*}
0 = \vert\vert\chi_{\epsilon}^{2k}\Delta^2u\vert\vert^2 &- 
2Re\langle \Delta^2u, \Delta\nabla_{(d\chi_{\epsilon}^{4k})^{\#}}u\rangle \\
&-2Re\langle\Delta^2u, \nabla_{(d\chi_{\epsilon}^{4k})^{\#}}\Delta u\rangle +
2Re\langle \Delta^2u, \Delta_M(\chi_{\epsilon}^{4k})\Delta u \rangle \\
&-2Re\langle \Delta^2u, \nabla_{(d\Delta_M(\chi_{\epsilon}^{4k}))^{\#}}u \rangle 
+Re\langle \Delta^2u, (\Delta_M^2\chi_{\epsilon}^{4k})u\rangle \\
&+ \langle Vu, \chi_{\epsilon}^{4k}u\rangle
\end{align*}
which in turn implies
\begin{align*}
\vert\vert\chi_{\epsilon}^{2k}\Delta^2u\vert\vert^2 &= 
2Re\langle \Delta^2u, \Delta\nabla_{(d\chi_{\epsilon}^{4k})^{\#}}u\rangle 
+ 2Re\langle\Delta^2u, \nabla_{(d\chi_{\epsilon}^{4k})^{\#}}\Delta u\rangle \\
&\hspace{0.5cm} 
-2Re\langle \Delta^2u, \Delta_M(\chi_{\epsilon}^{4k})\Delta u \rangle 
+2Re\langle \Delta^2u, \nabla_{(d\Delta_M(\chi_{\epsilon}^{4k}))^{\#}}u \rangle \\
&\hspace{0.5cm}
-Re\langle \Delta^2u, (\Delta_M^2\chi_{\epsilon}^{4k})u\rangle -
\langle Vu, \chi_{\epsilon}^{4k}u\rangle.
\end{align*}

Using the commutation formula from corollary \ref{connection_3}, we write 
the above as 
\begin{align}
\vert\vert \chi_{\epsilon}^{2k}\Delta^2u\vert\vert^2 &= 2Re\langle \Delta^2u, \nabla_{(d\chi_{\epsilon}^{4k})^{\#}}\Delta u\rangle 
+ 2Re\langle \Delta^2u, \nabla_{(d\chi_{\epsilon}^{4k})^{\#}}(Rm+F)*u\rangle \label{main_lemma_1_a}  \\
&\hspace{0.5cm} + 
2Re\langle \Delta^2u, (Rm+F)*\nabla_{(d\chi_{\epsilon}^{4k})^{\#}}\rangle + 
2Re\langle\Delta^2u, \nabla_{(d\chi_{\epsilon}^{4k})^{\#}}\Delta u\rangle \nonumber  \\
&\hspace{0.5cm} 
 -2Re\langle \Delta^2u, \Delta_M(\chi_{\epsilon}^{4k})\Delta u \rangle 
+2Re\langle \Delta^2u, \nabla_{(d\Delta_M(\chi_{\epsilon}^{4k}))^{\#}}u \rangle \nonumber  \\
&\hspace{0.5cm} 
-Re\langle \Delta^2u, (\Delta_M^2\chi_{\epsilon}^{4k})u\rangle 
-
\langle Vu, \chi_{\epsilon}^{4k}u\rangle \nonumber  
\end{align}

We now claim that the first seven terms in the above equation can be
bounded above by
$\epsilon C_1(\epsilon) \vert\vert\chi_{\epsilon}^{2k}\Delta^2u
\vert\vert^2 + C_2(\epsilon)\vert\vert u\vert\vert^2$, where
$C_1(\epsilon)$ and $C_2(\epsilon)$ are constants depending on
$\epsilon$. Furthermore, we have that
$\lim_{\epsilon \rightarrow 0}C_1(\epsilon) < \infty$ and
$\lim_{\epsilon \rightarrow 0}C_2(\epsilon) < \infty$. 

The proof of this
follows exactly how we proved lemmas \ref{bilaplacian_est_1} to 
\ref{bilaplace_est_5}, and corollaries \ref{bilaplace_est_1_a} to 
\ref{bilaplace_est_5_a}. We briefly give the details for the term
$2Re\langle \Delta^2u, \nabla_{(d\Delta_M(\chi_{\epsilon}^{4k}))^{\#}}u 
\rangle$.

Recall, by formula \eqref{dlaplace_form_2}, we can write 
$d\Delta_M(\chi_{\epsilon}^{4k}) = \chi_{\epsilon}^{4k-3}G_3$, where
$G_3$ is defined by formula \eqref{dlaplace_form_1}.
Therefore, we can write 
\begin{align*}
\langle \Delta^2u, \nabla_{(d\Delta_M(\chi_{\epsilon}^{4k}))^{\#}}u 
\rangle &= 
\langle \Delta^2u, 
\chi_{\epsilon}^{4k-3}
\nabla_{(G_3)^{\#}}u 
\rangle \\
&=
\langle \chi_{\epsilon}^{2k}\Delta^2u, 
\chi_{\epsilon}^{2k-3}
\nabla_{(G_3)^{\#}}u 
\rangle.
\end{align*}
Applying Cauchy-Schwarz and Young's inequality, we obtain
\begin{align*}
\vert \langle \Delta^2u, \nabla_{(d\Delta_M(\chi_{\epsilon}^{4k}))^{\#}}u 
\rangle
\vert &\leq
\frac{\epsilon}{2}\vert\vert\chi_{\epsilon}^{2k}\Delta^2u\vert\vert^2 +
\frac{1}{2\epsilon}\vert\vert\chi_{\epsilon}^{2k-3}
\nabla_{(G_3)^{\#}}u\vert\vert^2 \\
&\leq \frac{\epsilon}{2}\vert\vert\chi_{\epsilon}^{2k}\Delta^2u\vert\vert^2
+
\frac{C^2\epsilon}{2}\vert\vert\chi_{\epsilon}^{2k-3}\nabla u\vert\vert^2
\end{align*}
where to get the second inequality, we have used \eqref{star_notation_ex} 
and applied 
lemma \ref{G3_est_3}.

We then proceed by estimating the 
$\vert\vert\chi_{\epsilon}^{2k-3}\nabla u\vert\vert^2$ term by using
proposition \ref{covariant_est_1}. One then immediately gets that
$\vert \langle \Delta^2u, \nabla_{(d\Delta_M(\chi_{\epsilon}^{4k}))^{\#}}u 
\rangle
\vert$
is bounded above by a quantity of the form 
$\epsilon C_1(\epsilon) \vert\vert\chi_{\epsilon}^{2k}\Delta^2u
\vert\vert^2 + C_2(\epsilon)\vert\vert u\vert\vert^2$, with
$\lim_{\epsilon \rightarrow 0}C_1(\epsilon) < \infty$ and
$\lim_{\epsilon \rightarrow 0}C_2(\epsilon) < \infty$. The required 
estimate for 
$2Re\langle \Delta^2u, \nabla_{(d\Delta_M(\chi_{\epsilon}^{4k}))^{\#}}u 
\rangle$ then follows.

The next step in the proof of the lemma is to observe that, using 
equation \eqref{bilaplace_est_6_0}, we can rewrite equation 
\eqref{main_lemma_1_a} to obtain

\begin{align}
\langle \Delta^2(\chi_{\epsilon}^{2k}u), \Delta^2(\chi_{\epsilon}^{2k}u) \rangle  
&=
-4Re\langle \chi_{\epsilon}^{2k}\Delta^2u, \nabla_{(d\chi_{\epsilon}^{2k})^{\#}}\Delta u \rangle \label{main_lemma_1_b} \\
&\hspace{1cm}
-4Re \langle \chi_{\epsilon}^{2k}\Delta^2u, \Delta\nabla_{(d\chi_{\epsilon}^{2k})^{\#}}u\rangle \nonumber \\
&\hspace{1cm}
-4Re \langle \chi_{\epsilon}^{2k}\Delta^2u, \nabla_{
(d\Delta_M\chi_{\epsilon}^{2k})^{\#}}u\rangle 
-4Re  \langle \chi_{\epsilon}^{2k}\Delta^2u, (\Delta_M\chi_{\epsilon}^{2k})(\Delta u) \rangle \nonumber\\
&\hspace{1cm}
+ 2Re\langle \chi_{\epsilon}^{2k}\Delta^2u, u\Delta_M^2\chi_{\epsilon}^{2k}\rangle
+ 4\langle  \nabla_{(d\chi_{\epsilon}^{2k})^{\#}}\Delta u, \nabla_{(d\chi_{\epsilon}^{2k})^{\#}}\Delta u\rangle \nonumber\\
&\hspace{1cm} 
+ 8Re\langle\nabla_{(d\chi_{\epsilon}^{2k})^{\#}}\Delta u, 
\Delta\nabla_{(d\chi_{\epsilon}^{2k})^{\#}}u\rangle 
+8Re\langle\nabla_{(d\chi_{\epsilon}^{2k})^{\#}}\Delta u, 
\nabla_{(d\Delta_M\chi_{\epsilon}^{2k})^{\#}}u\rangle \nonumber\\ 
&\hspace{1cm} 
+8Re\langle\nabla_{(d\chi_{\epsilon}^{2k})^{\#}}\Delta u, 
(\Delta_M\chi_{\epsilon}^{2k})(\Delta u)\rangle 
-4Re\langle\nabla_{(d\chi_{\epsilon}^{2k})^{\#}}\Delta u, u\Delta_M^2\chi_{\epsilon}^{2k} 
\rangle \nonumber\\
&\hspace{1cm}  
+4Re\langle\Delta\nabla_{(d\chi_{\epsilon}^{2k})^{\#}}u, 
\Delta\nabla_{(d\chi_{\epsilon}^{2k})^{\#}}u\rangle 
+8Re\langle \Delta\nabla_{(d\chi_{\epsilon}^{2k})^{\#}}u,
\nabla_{(d\Delta_M\chi_{\epsilon}^{2k})^{\#}}u \rangle \nonumber\\
&\hspace{1cm} 
+8Re\langle \Delta\nabla_{(d\chi_{\epsilon}^{2k})^{\#}}u, 
(\Delta_M\chi_{\epsilon}^{2k})(\Delta u) \rangle 
-4Re\langle \Delta\nabla_{(d\chi_{\epsilon}^{2k})^{\#}}u, 
u\Delta_M^2\chi_{\epsilon}^{2k}\rangle \nonumber\\
&\hspace{1cm} 
+4\langle \nabla_{(d\Delta_M\chi_{\epsilon}^{2k})^{\#}}u, 
\nabla_{(d\Delta_M\chi_{\epsilon}^{2k})^{\#}}u\rangle 
+8Re\langle \nabla_{(d\Delta_M\chi_{\epsilon}^{2k})^{\#}}u, 
(\Delta_M\chi_{\epsilon}^{2k})\Delta u\rangle \nonumber\\
&\hspace{1cm} 
-4Re\langle \nabla_{(d\Delta_M\chi_{\epsilon}^{2k})^{\#}}u, 
u\Delta_M^2(\chi_{\epsilon}^{2k})\rangle 
+4\langle (\Delta_M\chi_{\epsilon}^{2k})(\Delta u), 
(\Delta_M\chi_{\epsilon}^{2k})(\Delta u)
\rangle \nonumber\\ 
&\hspace{1cm} 
-4Re\langle (\Delta_M\chi_{\epsilon}^{2k})(\Delta u), u\Delta_M^2\chi_{\epsilon}^{2k}\rangle 
+\langle u\Delta_M^2\chi_{\epsilon}^{2k}, u\Delta_M^2\chi_{\epsilon}^{2k}\rangle \nonumber \\
&\hspace{1cm}+ 2Re\langle \Delta^2u, \nabla_{(d\chi_{\epsilon}^{4k})^{\#}}\Delta u\rangle 
+ 2Re\langle \Delta^2u, \nabla_{(d\chi_{\epsilon}^{4k})^{\#}}(Rm+F)*u\rangle \nonumber  \\
&\hspace{1cm} + 
2Re\langle \Delta^2u, (Rm+F)*\nabla_{(d\chi_{\epsilon}^{4k})^{\#}}\rangle + 
2Re\langle\Delta^2u, \nabla_{(d\chi_{\epsilon}^{4k})^{\#}}\Delta u\rangle \nonumber  \\
&\hspace{1cm} 
 -2Re\langle \Delta^2u, \Delta_M(\chi_{\epsilon}^{4k})\Delta u \rangle 
+2Re\langle \Delta^2u, \nabla_{(d\Delta_M(\chi_{\epsilon}^{4k}))^{\#}}u \rangle \nonumber  \\
&\hspace{1cm} 
-Re\langle \Delta^2u, (\Delta_M^2\chi_{\epsilon}^{4k})u\rangle 
-
\langle Vu, \chi_{\epsilon}^{4k}u\rangle. \nonumber  
\end{align}

We now note that the first twenty terms can be bounded above by
$\epsilon C_1(\epsilon)
\vert\vert\chi_{\epsilon}^{2k}\Delta^2u\vert\vert^2 +
+ C_2(\epsilon)\vert\vert u\vert\vert^2$, this was shown
in corollaries \ref{bilaplace_est_1_a} to \ref{bilaplace_est_5_a}.
The next seven terms can also be bounded by the same quantity, this was
shown above. Furthermore, these constants that depend on $\epsilon$
satisfy the following limit conditions, 
$\lim_{\epsilon\rightarrow 0}C_1(\epsilon) < \infty$ and 
$\lim_{\epsilon\rightarrow 0}C_2(\epsilon) < \infty$.

This means we can obtain the following estimate
\begin{equation*}
\vert\vert \Delta^2(\chi_{\epsilon}^{2k}u)\vert\vert^2 \leq 
\epsilon C_1(\epsilon)
\vert\vert\chi_{\epsilon}^{2k}\Delta^2u\vert\vert^2 
+ C_2(\epsilon)\vert\vert u\vert\vert^2 - 
\langle Vu, \chi_{\epsilon}^{4k}u\rangle. 
\end{equation*}

Using our assumption on the potential $V$, we know that 
\begin{equation*}
-\langle Vu, \chi_{\epsilon}^{4k}u\rangle \leq 
\langle (q\circ r)(u), \chi_{\epsilon}^{4k}u\rangle
\end{equation*}
which implies
\begin{equation*}
\vert\vert \Delta^2(\chi_{\epsilon}^{2k}u)\vert\vert^2 \leq 
\epsilon C_1(\epsilon)
\vert\vert\chi_{\epsilon}^{2k}\Delta^2u\vert\vert^2 
+ C_2(\epsilon)\vert\vert u\vert\vert^2 +
\langle (q\circ r)(u), \chi_{\epsilon}^{4k}u\rangle.
\end{equation*}
Applying proposition \ref{bilaplace_est_6}, we can estimate
the term $\vert\vert\chi_{\epsilon}^{2k}\Delta^2u\vert\vert^2$ in the above
equation and obtain
\begin{equation*}
\vert\vert \Delta^2(\chi_{\epsilon}^{2k}u)\vert\vert^2 \leq 
\frac{\epsilon C_1(\epsilon)}{1-\epsilon C_1(\epsilon)}
\vert\vert \Delta^2(\chi_{\epsilon}^{2k}u)\vert\vert^2 + 
\frac{\epsilon  C_1(\epsilon) C_2(\epsilon)}{1-\epsilon  C_1(\epsilon)}
\vert\vert u\vert\vert^2 + C_2(\epsilon)\vert\vert u\vert\vert^2 +
\langle (q\circ r)(u), \chi_{\epsilon}^{4k}u\rangle
\end{equation*}
which implies
\begin{equation*}
\bigg{(} 
1 - \frac{\epsilon C_1(\epsilon)}{1-\epsilon C_1(\epsilon)}
\bigg{)}\vert\vert \Delta^2(\chi_{\epsilon}^{2k}u)\vert\vert^2 \leq
\frac{C_2(\epsilon)}{1-\epsilon C_1(\epsilon)}\vert\vert u\vert\vert^2 +
\langle (q\circ r)(u), \chi_{\epsilon}^{4k}u\rangle.
\end{equation*}
By choosing $\epsilon$ sufficiently small, we can make it so that
$1 - \frac{\epsilon C_1(\epsilon)}{1-\epsilon C_1(\epsilon)} = 
\frac{1- 2\epsilon C_1(\epsilon)}{1-\epsilon C_1(\epsilon)} > 0$.
Dividing the above equation by this, we obtain
\begin{equation*}
\vert\vert \Delta^2(\chi_{\epsilon}^{2k}u)\vert\vert^2 \leq 
\frac{C_2(\epsilon)}{1-2\epsilon C_1(\epsilon)}\vert\vert u\vert\vert^2
+ 
\frac{1- \epsilon C_1(\epsilon)}{1-2\epsilon C_1(\epsilon)}
\langle (q\circ r)(u), \chi_{\epsilon}^{4k}u\rangle.
\end{equation*}

Observing that we can write 
$\frac{1- \epsilon C_1(\epsilon)}{1-2\epsilon C_1(\epsilon)} =
1 + \frac{\epsilon C_1(\epsilon)}{1-2\epsilon C_1(\epsilon)}$. It is 
easy to see that, for $\epsilon$ sufficiently small, 
$\frac{1- \epsilon C_1(\epsilon)}{1-2\epsilon C_1(\epsilon)} < 2$.
Here we have used the fact that we know that 
$\lim_{\epsilon \rightarrow 0}C_1(\epsilon) < \infty$ and
$\lim_{\epsilon \rightarrow 0}C_2(\epsilon) < \infty$.
Putting this observation into the above estimate gives
\begin{equation*}
\vert\vert \Delta^2(\chi_{\epsilon}^{2k}u)\vert\vert^2 \leq 
\frac{C_2(\epsilon)}{1-2\epsilon C_1(\epsilon)}\vert\vert u\vert\vert^2
+ 
2
\langle (q\circ r)(u), \chi_{\epsilon}^{4k}u\rangle
\end{equation*}
provided $\epsilon$ is sufficiently small. This finishes the proof of the
lemma.

\end{proof}

Armed with the above lemma, we can now give the proof of 
theorem \ref{main_theorem_1}.

\begin{proof}[\textbf{proof of theorem \ref{main_theorem_1}}]

We will follow the strategy Milatovich employs in \cite{milatovic}.

Suppose $u \in Dom(T_{max})$ satisfies $T_{\max}u = i\lambda u$, for 
some $\lambda \in \R$. The essential 
self adjointness of $T\vert_{C_{c}^{\infty}}$ will follow if we can show 
that $u = 0$.

For $\epsilon > 0$ define 
\begin{equation*}
G_{\epsilon} = \{x \in M : r(x) \leq \frac{4}{\epsilon} \},
\end{equation*}
where $r(x)$ is as in \eqref{distance}. Let 
$q$ be as in the statement of theorem \ref{main_theorem_1}. We then have
\begin{align*}
\langle (q\circ r)u, \chi_{\epsilon}^{4k}u\rangle &\leq 
\int_{G_{\epsilon}}q(r(x))\vert u(x)\vert^2d\mu \\
&\leq 
q\bigg{(}\frac{4}{\epsilon}\bigg{)}\vert\vert u\vert\vert^2.
\end{align*}

Using lemma \ref{main_lemma_1}, and the assumption that $q(s) = O(s)$, 
we 
have
\begin{align}
\vert\vert \Delta^2(\chi_{\epsilon}^{2k}u)\vert\vert^2 &\leq 
\bigg{(}\frac{C_2(\epsilon)}{1-2\epsilon C_1(\epsilon)} 
\bigg{)}\vert\vert u\vert\vert^2 + 
\frac{C}{\epsilon}\vert\vert u\vert\vert^2 \nonumber \\
&= \bigg{(}\frac{C_2(\epsilon)}{1-2\epsilon C_1(\epsilon)} + 
\frac{C}{\epsilon} \bigg{)}\vert\vert u\vert\vert^2 
\label{square_est}
\end{align}
for some constant $C > 0$.

Taking 
imaginary parts in equation \eqref{main_lemma_1_eqn}, we obtain
 \begin{align*}
 \lambda \langle u, \chi_{\epsilon}^{4k}u\rangle = &-2Im\langle \Delta^2u, \Delta\nabla_{(d\chi_{\epsilon}^{4k})^{\#}}u\rangle
 -2Im\langle \Delta^2u, \nabla_{(d\chi_{\epsilon}^{4k})^{\#}}\Delta u\rangle  \\
 &+ 
 2Im\langle \Delta^2u,  \Delta_M(\chi_{\epsilon}^{4k})\Delta u \rangle  
 -2Im\langle \Delta^2u, \nabla_{(d\Delta_M(\chi_{\epsilon}^{4k}))^{\#}}u \rangle \\
 &+ 
 Im\langle \Delta^2u, (\Delta_M^2\chi_{\epsilon}^{4k})u\rangle. 
 \end{align*}

Simplifying the right hand side of the above, using 
\eqref{derivative_powers_a}, \eqref{derivative_powers_b}
and \eqref{dlaplace_form_2}, we see that we can write

\begin{align}
\lambda \langle u, \chi_{\epsilon}^{4k}u\rangle = &
-2Im\langle 2\chi_{\epsilon}^{2k}\Delta^2u, \Delta\nabla_{(d\chi_{\epsilon}^{2k})^{\#}}u\rangle
 -2Im\langle 2\chi_{\epsilon}^{2k}\Delta^2u, 
 \nabla_{(d\chi_{\epsilon}^{2k})^{\#}}\Delta u\rangle  \label{main_thm_1_a} \\
 &+ 
 2Im\langle \chi_{\epsilon}^{2k}\Delta^2u,  \chi_{\epsilon}^{2k-2}G_1\Delta u \rangle  
 -2Im\langle \chi_{\epsilon}^{2k}\Delta^2u, \chi_{\epsilon}^{2k-3}
 \nabla_{(G_3)^{\#}}u \rangle \nonumber \\
 &+ 
 Im\langle \chi_{\epsilon}^{2k}\Delta^2u, 
 \chi_{\epsilon}^{2k-4}G_2u\rangle. \nonumber
\end{align}

In the proof of lemma \ref{main_lemma_1} (see \eqref{main_lemma_1_a}) we 
explained how we could bound each of the terms on the right hand side of 
\eqref{main_thm_1_a}
by 
$\epsilon C_1(\epsilon)\vert\vert \chi_{\epsilon}^{2k}\Delta^2\vert\vert^2
+ C_2(\epsilon)\vert\vert u\vert\vert^2$, where 
$\lim_{\epsilon\rightarrow 0}C_1(\epsilon) < \infty$ and
$\lim_{\epsilon\rightarrow 0}C_2(\epsilon) < \infty$.

Combining this with proposition \ref{bilaplace_est_6}, we see that each 
of the terms on the right hand side of \eqref{main_thm_1_a} can be bounded 
above by
\begin{equation*}
\frac{\epsilon C_1(\epsilon)}{1-\epsilon C_1(\epsilon}\vert\vert
\Delta^2(\chi_{\epsilon}^{2k}u)\vert\vert^2 + 
\bigg{(}
\frac{\epsilon C_1(\epsilon)C_2(\epsilon)}{1-\epsilon C_1(\epsilon)}
+ C_2(\epsilon)\bigg{)}
\vert\vert u\vert\vert^2.
\end{equation*}

Using this bound in \eqref{main_thm_1_a}, we obtain
\begin{align*}
\vert\lambda\vert\vert\vert \chi_{\epsilon}^{2k}u\vert\vert^2 \leq 
\frac{\epsilon C_1(\epsilon)}{1-\epsilon C_1(\epsilon}\vert\vert
\Delta^2(\chi_{\epsilon}^{2k}u)\vert\vert^2 + 
\bigg{(}
\frac{\epsilon C_1(\epsilon)C_2(\epsilon)}{1-\epsilon C_1(\epsilon)}
+ C_2(\epsilon)\bigg{)}
\vert\vert u\vert\vert^2.
\end{align*} 
 
We can then bound the first term on the right hand side of the above
inequality, by using \eqref{square_est}, to obtain
\begin{align*}
\vert\lambda\vert\vert\vert \chi_{\epsilon}^{2k}u\vert\vert^2 \leq 
\bigg{(}
\frac{\epsilon C_1(\epsilon)C_2(\epsilon)}
{(1-\epsilon C_1(\epsilon))(1-2\epsilon C_1(\epsilon))} + 
\frac{CC_1(\epsilon)}{1-\epsilon C_1(\epsilon)} + 
\frac{\epsilon C_1(\epsilon)C_2(\epsilon)}{1-\epsilon C_1(\epsilon)}
+ C_2(\epsilon)
\bigg{)}\vert\vert u\vert\vert^2
\end{align*}


Letting $\epsilon \rightarrow 0$ in the above inequality, we obtain
\begin{align*}
\lim_{\epsilon \rightarrow 0}
\vert\lambda\vert\vert\vert \chi_{\epsilon}^{2k}u\vert\vert^2 \leq 
\widetilde{C}\vert\vert u\vert\vert^2
\end{align*} 
where $\widetilde{C} := 
\lim_{\epsilon \rightarrow 0}
(CC_1(\epsilon) + C_2(\epsilon)) < \infty$.

By the dominated convergence theorem, we have that
$\lim_{\epsilon \rightarrow 0}
\vert\lambda\vert\vert\vert \chi_{\epsilon}^{2k}u\vert\vert^2 = 
\vert\lambda\vert\vert\vert u\vert\vert^2$. Hence we obtain
\begin{equation*}
\vert\lambda\vert\vert\vert u\vert\vert^2 \leq 
\widetilde{C}\vert\vert u\vert\vert^2.
\end{equation*}
 As $\vert\lambda\vert$ may be chosen arbitrarily large, we see that the above 
 inequality implies $u = 0$.
 
\end{proof}

\section{Localised derivative estimates for the magnetic Laplacian}
\label{localised_derivative_mag_lap}

In this section, we obtain several localised derivative estimates
that will be needed for the proof of theorem \ref{main_theorem_2}.
These can be seen as analogous, for the magnetic Laplacian, of the 
estimates we obtained in
section \ref{sec_covariant_derivatives} and 
\ref{sec_laplacian_derivatives}.

We will denote the magnetic differential $d_A$ on functions by
$\nabla$. 

We assume we have a function $w \in C^4(M)$ with $w \geq 1$. 
Let $h = w^{-1}$, so that $0 \leq h \leq 1$.

Using \eqref{Laplace_comp} and \eqref{laplace_product_formula},
we have the following formulas for the Laplacian and Bilaplacian of 
$h^{1/2}$.

\begin{equation}
\Delta_M(h^{1/2}) = 
\frac{-1}{4}h^{-3/2}|dh|^2 + \frac{1}{2}h^{-1/2}\Delta_M h
\label{derivative_h_1}
\end{equation}

\begin{align}
\Delta_M^2(h^{1/2}) = &-\frac{1}{4}h^{-3/2}\Delta_M \vert dh\vert^2  
- \frac{3}{4}h^{-5/2}\langle dh, d\vert dh\vert^2\rangle -
\frac{15}{16}h^{-7/2}\vert dh\vert^2\vert dh\vert^2 
\label{derivative_h_2}\\
& + \frac{3}{8}h^{-5/2}\vert dh\vert^2\Delta_M h + 
\frac{1}{2}h^{-1/2}\Delta_M^2h + \frac{1}{2}h^{-3/2}\langle dh, 
d\Delta_M h\rangle
+ \frac{3}{4}h^{-3/2}(\Delta_M h)\vert dh\vert^2 \nonumber\\
& -\frac{1}{2}h^{-3/2}(\Delta_M h)^2 \nonumber
\end{align}

In the statement of theorem \ref{main_theorem_2}, we had various 
derivative assumptions that we imposed on $h$. As we will be making
use of these assumptions in obtaining localised derivative estimates,
we list them here for convenience.

\begin{itemize}
\item $\vert dh\vert \leq \sigma h^{7/4}$

\item $\vert \Delta_M h\vert \leq \sigma h^{3/4}$

\item $\vert \Delta_M^2h\vert \leq \sigma h^{1/2}$

\item $\vert \Delta_M \vert dh\vert^2\vert \leq \sigma h^{3/2}$

\item $\vert d\vert dh\vert^2\vert \leq \sigma h^{22/8}$

\item $\vert d\Delta_M h\vert \leq \sigma h^{3/2}$ 
\end{itemize}

With these assumptions, it is easy to obtain the following estimates.

\begin{itemize}
\item $\vert dh^{1/2}\vert \leq \sigma h^{5/4}$

\item $\vert \Delta_M h^{1/2}\vert \leq \sigma h^{1/4}$ if $0 < \sigma \leq 1$

\item $\vert \Delta_M^2(h^{1/2})\vert \leq \sigma$ if 
$286\sigma + 18\sigma^2 + 15\sigma^3 \leq 4$

\item $\vert d\Delta_M h^{1/2}\vert \leq \sigma h$
\end{itemize}

The conditions $0 < \sigma \leq 1$ and 
$286\sigma + 18\sigma^2 + 15\sigma^3 \leq 4$ will be satisfied in our
context, as we will taking $\sigma$ small. Therefore, they 
can be safely ignored.

We now obtain several localised derivative estimates that will be useful 
in the proof of theorem \ref{main_theorem_2}.

\begin{prop}\label{mag_covariant_est_1}
Given $u \in C_c^{\infty}(M)$ and
$\sigma > 0$ sufficiently small, we have the following estimate
\begin{equation*}
\vert\vert h^k\nabla u\vert\vert^2 \leq \frac{1}{2(1-k\sigma)}
\vert\vert h^{k+1}\Delta_Au\vert\vert^2 +
 \frac{1+2k\sigma}{2(1-k\sigma)}\vert\vert h^{k-1}u\vert\vert^2.
\end{equation*}
\end{prop}

The proof of the above proposition follows exactly the same lines
as the proof of proposition \ref{covariant_est_1}.

We also have the following estimate, which follows from applying 
proposition \ref{mag_covariant_est_1} to $\Delta_Au$.

\begin{prop}\label{mag_covariant_est_2}
Given $u \in C_c^{\infty}(M)$ and
$\sigma > 0$ sufficiently small, we have the following estimate
\begin{equation*}
\vert\vert h^k\nabla\Delta_A u\vert\vert^2 \leq 
\frac{1}{2(1-k\sigma)}\vert\vert h^k\Delta_A^2u\vert\vert^2 + 
\bigg{(}\frac{1}{1-k\sigma} \bigg{)}\bigg{(}\frac{1}{2} + k\sigma\bigg{)}
\vert\vert h^k\Delta_Au\vert\vert^2
\end{equation*}
\end{prop}

The next step is to obtain an estimate for the magnetic Laplacian, 
analogous to 
proposition \ref{Laplace_est}. The estimate we need will be slightly 
different to that of \ref{Laplace_est}, and so we give the proof of this 
estimate.

\begin{lem}\label{mag_laplace_est_1}
Given $u \in C_c^{\infty}(M)$ and
$\sigma > 0$ sufficiently small, we have the following estimate
\begin{align*}
\vert\vert h^{1/2}\Delta_Au\vert\vert^2 &\leq 
\bigg{(}
\frac{\sqrt{\sigma}}{2} + \frac{\sigma^2}{2(1-\frac{7}{4}\sigma)}
 \bigg{)}
\bigg{(}
1 - \sigma^2\bigg{(}\frac{1}{1-\frac{7}{4}\sigma} \bigg{)}
\bigg{(}\frac{1}{2} + \frac{7\sigma}{4} \bigg{)} - \frac{\sigma^2}{2}
\bigg{)}^{-1}\vert\vert h^{1/2}\Delta_A^2u\vert\vert^2 \\ 
&\hspace{0.5cm} +
\bigg{(}\frac{1}{2\sqrt{\sigma}} + \frac{3}{2} \bigg{)}
\bigg{(}
1 - \sigma^2\bigg{(}\frac{1}{1-\frac{7}{4}\sigma} \bigg{)}
\bigg{(}\frac{1}{2} + \frac{7\sigma}{4} \bigg{)} - \frac{\sigma^2}{2}
\bigg{)}^{-1}\vert\vert u\vert\vert^2.
\end{align*}
\end{lem}

\begin{proof}
We start by integrating by parts, and using 
\eqref{laplace_product_formula} to obtain
\begin{align*}
\langle h^{1/2}\Delta_Au, h^{1/2}\Delta_Au\rangle &= 
\langle \Delta_A(h\Delta_Au), u\rangle \\
&= \langle h\Delta_A^2u, u\rangle - 
2\langle \nabla_{(dh)^{\#}}\Delta_Au, u\rangle +
\langle (\Delta_M h)(\Delta_Au), u\rangle.
\end{align*}
 Applying Cauchy-Schwarz and Young's inequality, we obtain
\begin{align*}
\langle h^{1/2}\Delta_Au, h^{1/2}\Delta_Au\rangle &\leq 
\frac{\sqrt{\sigma}}{2}\vert\vert h\Delta_A^2u\vert\vert^2 + 
\frac{1}{2\sqrt{\sigma}}\vert\vert u\vert\vert^2 + 
\vert\vert \nabla_{(dh)^{\#}}\Delta_Au\vert\vert^2 + 
\vert\vert u\vert\vert^2 + 
\frac{1}{2}\vert\vert (\Delta_M h)\Delta_Au\vert\vert^2 +
\frac{1}{2}\vert\vert u\vert\vert^2 \\
&\leq 
\frac{\sqrt{\sigma}}{2}\vert\vert h\Delta_A^2u\vert\vert^2 + 
\sigma^2\vert\vert h^{7/4}\nabla\Delta_Au\vert\vert^2 + 
\frac{\sigma^2}{2}\vert\vert h^{3/4}\Delta_Au\vert\vert^2 + 
\bigg{(} 
\frac{1}{2\sqrt{\sigma}} + 1 + \frac{1}{2}
\bigg{)}\vert\vert u\vert\vert^2
\end{align*}
where to get the second inequality we have applied the assumptions on
$\vert dh\vert$ and $\vert \Delta_M h\vert$, and used 
\eqref{star_notation_ex}.

Applying proposition \ref{mag_covariant_est_2} to estimate the
$\vert\vert h^{7/4}\nabla\Delta_Au\vert\vert^2$ term, we get
\begin{align*}
\langle h^{1/2}\Delta_Au, h^{1/2}\Delta_Au\rangle &\leq 
\bigg{(} 
\frac{\sqrt{\sigma}}{2} + \frac{\sigma^2}{2(1-\frac{7}{4}\sigma)}
\bigg{)}
\vert\vert h^{1/2}\Delta_A^2u\vert\vert^2 + 
(\sigma^2)\bigg{(}\frac{1}{1-\frac{7}{4}\sigma} \bigg{)}
\bigg{(}\frac{1}{2} + \frac{7\sigma}{2} \bigg{)}
\vert\vert h^{1/2}\Delta_Au\vert\vert^2 \\
&\hspace{0.5cm}+
\bigg{(}\frac{\sigma^2}{2} \bigg{)}\vert\vert h^{1/2}\Delta_Au\vert\vert^2
+ \bigg{(}\frac{1}{2\sqrt{\sigma}} + 1 + \frac{1}{2} \bigg{)}
\vert\vert u\vert\vert^2.
\end{align*}

This implies
\begin{align*}
\bigg{(}
1 - (\sigma^2)\bigg{(}\frac{1}{1-\frac{7}{4}\sigma} \bigg{)}
\bigg{(}\frac{1}{2} + \frac{7\sigma}{4} \bigg{)} - \frac{\sigma^2}{2}
\bigg{)}\vert\vert h^{1/2}\Delta_Au\vert\vert^2 &\leq
\bigg{(} 
\frac{\sqrt{\sigma}}{2} + \frac{\sigma^2}{2(1-\frac{7}{4}\sigma)}
\bigg{)}
\vert\vert h^{1/2}\Delta_A^2u\vert\vert^2 \\
&\hspace{0.5cm}
+ 
\bigg{(}\frac{1}{2\sqrt{\sigma}} + 1 + \frac{1}{2} \bigg{)}
\vert\vert u\vert\vert^2.
\end{align*}
Choosing $\sigma > 0$ small enough so that
$1 - (\sigma^2)\bigg{(}\frac{1}{1-\frac{7}{4}\sigma} \bigg{)}
\bigg{(}\frac{1}{2} + \frac{7\sigma}{4} \bigg{)} - \frac{\sigma^2}{2} > 0$,
we obtain
\begin{align*}
\vert\vert h^{1/2}\Delta_Au\vert\vert^2 &\leq
\bigg{(} 
\frac{\sqrt{\sigma}}{2} + \frac{\sigma^2}{2(1-\frac{7}{4}\sigma)}
\bigg{)}
\bigg{(}
1 - (\sigma^2)\bigg{(}\frac{1}{1-\frac{7}{4}\sigma} \bigg{)}
\bigg{(}\frac{1}{2} + \frac{7\sigma}{4} \bigg{)} - \frac{\sigma^2}{2}
\bigg{)}^{-1}\vert\vert h^{1/2}\Delta_A^2u\vert\vert^2 \\
&\hspace{0.5cm}
+ 
\bigg{(}\frac{1}{2\sqrt{\sigma}} + 1 + \frac{1}{2} \bigg{)}
\bigg{(}
1 - (\sigma^2)\bigg{(}\frac{1}{1-\frac{7}{4}\sigma} \bigg{)}
\bigg{(}\frac{1}{2} + \frac{7\sigma}{4} \bigg{)} - \frac{\sigma^2}{2}
\bigg{)}^{-1}\vert\vert u\vert\vert^2
\end{align*}
which proves the lemma.
\end{proof}

We will also need the following estimate. The proof of which follows
in exactly the same manner as the proof of the above lemma.

\begin{lem}\label{mag_laplace_est_2}
Given $u \in C_c^{\infty}(M)$ and
$\sigma > 0$ sufficiently small, we have the following estimate
\begin{align*}
\vert\vert h^{1/4}\Delta_Au\vert\vert^2 &\leq 
\bigg{(} 
\frac{\sqrt{\sigma}}{2} + \frac{\sigma^2}{2(1-\frac{5}{4}\sigma)}
\bigg{)}
\bigg{(}
1 - (\sigma^2)\bigg{(}\frac{1}{1-\frac{5}{4}\sigma} \bigg{)}
\bigg{(}\frac{1}{2} + \frac{5\sigma}{4} \bigg{)} - \frac{\sigma^2}{2}
\bigg{)}^{-1}\vert\vert h^{1/2}\Delta_A^2u\vert\vert^2 \\
&\hspace{0.5cm} +
\bigg{(}
\frac{1}{2\sqrt{\sigma}} + \frac{3}{2}
\bigg{)}
\bigg{(}
1 - (\sigma^2)\bigg{(}\frac{1}{1-\frac{5}{4}\sigma} \bigg{)}
\bigg{(}\frac{1}{2} + \frac{5\sigma}{4} \bigg{)} - \frac{\sigma^2}{2}
\bigg{)}^{-1}\vert\vert u\vert\vert^2.
\end{align*}
\end{lem}

We need the following estimates

\begin{lem}\label{mag_main_est_1}
Given $u \in C_c^{\infty}(M)$ and
$\sigma > 0$ sufficiently small, we have the following estimate
\begin{align*}
\vert \langle h^{1/2}\Delta_A^2u, \nabla_{(dh^{1/2})^{\#}}\Delta_Au\rangle
\vert &\leq 
\bigg{[}
\frac{\sigma}{2} + \bigg{(}\frac{\sigma}{2}\bigg{)}
\bigg{(}
\frac{1}{2(1-\frac{5}{4}\sigma)}\bigg{)} + \\
&\hspace{0.5cm}
\bigg{(}\frac{\sigma}{2}\bigg{)}
\bigg{(}\frac{1}{1-\frac{5}{4}\sigma} \bigg{)}
\bigg{(}\frac{1}{2} + \frac{5}{4}\sigma\bigg{)}
\bigg{(}
\frac{\sqrt{\sigma}}{2} + \frac{\sigma^2}{2(1-\frac{7}{4}\sigma)}
 \bigg{)} \\
&\hspace{0.5cm} 
\bigg{(}
1 - \sigma^2\bigg{(}\frac{1}{1-\frac{7}{4}\sigma} \bigg{)}
\bigg{(}\frac{1}{2} + \frac{7\sigma}{4} \bigg{)} - \frac{\sigma^2}{2}
\bigg{)}^{-1}\bigg{]}
\vert\vert h^{1/2}
\Delta_A^2u\vert\vert^2 \\
&\hspace{0.5cm}+
\bigg{(}\frac{\sigma}{2}\bigg{)}
\bigg{(}\frac{1}{1-\frac{5}{4}\sigma} \bigg{)}
\bigg{(}\frac{1}{2} + \frac{5}{4}\sigma\bigg{)}
\bigg{(}\frac{1}{2\sqrt{\sigma}} + \frac{3}{2} \bigg{)} \\
&\hspace{0.5cm}
\bigg{(}
1 - \sigma^2\bigg{(}\frac{1}{1-\frac{7}{4}\sigma} \bigg{)}
\bigg{(}\frac{1}{2} + \frac{7\sigma}{4} \bigg{)} - \frac{\sigma^2}{2}
\bigg{)}^{-1}\vert\vert u\vert\vert^2.
\end{align*}
\end{lem}

\begin{proof}
By Cauchy-Schwarz and Young's inequality we have
\begin{align}
\vert \langle h^{1/2}\Delta_A^2u, \nabla_{(dh^{1/2})^{\#}}\Delta_Au\rangle\vert 
&\leq 
\frac{\sigma}{2}\vert\vert h^{1/2}\Delta_A^2u\vert\vert^2 + 
\frac{1}{2\sigma}\vert\vert\nabla_{(dh^{1/2})^{\#}}\Delta_Au\vert\vert^2 
\nonumber \\
&\leq \frac{\sigma}{2}\vert\vert h^{1/2}\Delta_A^2u\vert\vert^2 +
\frac{\sigma}{2}\vert\vert h^{5/4}\nabla\Delta_Au\vert\vert^2
\end{align}
where to get the second inequality we have used \eqref{star_notation_ex},
and
the fact that 
$\vert dh^{1/2}\vert \leq \sigma h^{5/4}$.

We then use proposition \ref{mag_covariant_est_2} to estimate the
$\vert\vert h^{5/4}\nabla\Delta_Au\vert\vert^2$ term, and obtain
\begin{align*}
\vert \langle h^{1/2}\Delta_A^2u, \nabla_{(dh^{1/2})^{\#}}\Delta_Au\rangle
\vert &\leq
\frac{\sigma}{2}\vert\vert h^{1/2}\Delta_A^2u\vert\vert^2 + 
\bigg{(}\frac{\sigma}{2}\bigg{)}
\bigg{(}
\frac{1}{2(1-\frac{5}{4}\sigma)}\bigg{)}
\vert\vert h^{5/4}\Delta_A^2u\vert\vert^2 \\
&\hspace{0.5cm}
+ 
\bigg{(}\frac{\sigma}{2}\bigg{)}
\bigg{(}\frac{1}{1-\frac{5}{4}\sigma} \bigg{)}
\bigg{(}\frac{1}{2} + \frac{5}{4}\sigma\bigg{)}
\vert\vert h^{5/4}\Delta_Au\vert\vert^2 \\
&\leq
\frac{\sigma}{2}\vert\vert h^{1/2}\Delta_A^2u\vert\vert^2 + 
\bigg{(}\frac{\sigma}{2}\bigg{)}
\bigg{(}
\frac{1}{2(1-\frac{5}{4}\sigma)}\bigg{)}
\vert\vert h^{1/2}\Delta_A^2u\vert\vert^2 \\
&\hspace{0.5cm}
+ 
\bigg{(}\frac{\sigma}{2}\bigg{)}
\bigg{(}\frac{1}{1-\frac{5}{4}\sigma} \bigg{)}
\bigg{(}\frac{1}{2} + \frac{5}{4}\sigma\bigg{)}
\vert\vert h^{1/2}\Delta_Au\vert\vert^2 
\end{align*}
where to get the second inequality we have used $h^{5/4} \leq h^{1/2}$.
Applying lemma \ref{mag_laplace_est_1} to estimate the 
$\vert\vert h^{1/2}\Delta_Au\vert\vert^2$ term, we obtain
\begin{align*}
\vert \langle h^{1/2}\Delta_A^2u, \nabla_{(dh^{1/2})^{\#}}\Delta_Au\rangle
\vert &\leq
\frac{\sigma}{2}\vert\vert h^{1/2}\Delta_A^2u\vert\vert^2 + 
\bigg{(}\frac{\sigma}{2}\bigg{)}
\bigg{(}
\frac{1}{2(1-\frac{5}{4}\sigma)}\bigg{)}
\vert\vert h^{1/2}\Delta_A^2u\vert\vert^2 \\
&\hspace{0.5cm}
+ 
\bigg{(}\frac{\sigma}{2}\bigg{)}
\bigg{(}\frac{1}{1-\frac{5}{4}\sigma} \bigg{)}
\bigg{(}\frac{1}{2} + \frac{5}{4}\sigma\bigg{)}
\bigg{(}
\frac{\sqrt{\sigma}}{2} + \frac{\sigma^2}{2(1-\frac{7}{4}\sigma)}
 \bigg{)} \\
&\hspace{0.5cm} 
\bigg{(}
1 - \sigma^2\bigg{(}\frac{1}{1-\frac{7}{4}\sigma} \bigg{)}
\bigg{(}\frac{1}{2} + \frac{7\sigma}{4} \bigg{)} - \frac{\sigma^2}{2}
\bigg{)}^{-1}\vert\vert h^{1/2}\Delta_A^2u\vert\vert^2 \\ 
&\hspace{0.5cm} +
\bigg{(}\frac{\sigma}{2}\bigg{)}
\bigg{(}\frac{1}{1-\frac{5}{4}\sigma} \bigg{)}
\bigg{(}\frac{1}{2} + \frac{5}{4}\sigma\bigg{)}
\bigg{(}\frac{1}{2\sqrt{\sigma}} + \frac{3}{2} \bigg{)} \\
&\hspace{0.5cm}
\bigg{(}
1 - \sigma^2\bigg{(}\frac{1}{1-\frac{7}{4}\sigma} \bigg{)}
\bigg{(}\frac{1}{2} + \frac{7\sigma}{4} \bigg{)} - \frac{\sigma^2}{2}
\bigg{)}^{-1}\vert\vert u\vert\vert^2 \\
&=
\bigg{[}
\frac{\sigma}{2} + \bigg{(}\frac{\sigma}{2}\bigg{)}
\bigg{(}
\frac{1}{2(1-\frac{5}{4}\sigma)}\bigg{)} + \\
&\hspace{0.5cm}
\bigg{(}\frac{\sigma}{2}\bigg{)}
\bigg{(}\frac{1}{1-\frac{5}{4}\sigma} \bigg{)}
\bigg{(}\frac{1}{2} + \frac{5}{4}\sigma\bigg{)}
\bigg{(}
\frac{\sqrt{\sigma}}{2} + \frac{\sigma^2}{2(1-\frac{7}{4}\sigma)}
 \bigg{)} \\
&\hspace{0.5cm} 
\bigg{(}
1 - \sigma^2\bigg{(}\frac{1}{1-\frac{7}{4}\sigma} \bigg{)}
\bigg{(}\frac{1}{2} + \frac{7\sigma}{4} \bigg{)} - \frac{\sigma^2}{2}
\bigg{)}^{-1}\bigg{]}
\vert\vert h^{1/2}
\Delta_A^2u\vert\vert^2 \\
&\hspace{0.5cm}+
\bigg{(}\frac{\sigma}{2}\bigg{)}
\bigg{(}\frac{1}{1-\frac{5}{4}\sigma} \bigg{)}
\bigg{(}\frac{1}{2} + \frac{5}{4}\sigma\bigg{)}
\bigg{(}\frac{1}{2\sqrt{\sigma}} + \frac{3}{2} \bigg{)} \\
&\hspace{0.5cm}
\bigg{(}
1 - \sigma^2\bigg{(}\frac{1}{1-\frac{7}{4}\sigma} \bigg{)}
\bigg{(}\frac{1}{2} + \frac{7\sigma}{4} \bigg{)} - \frac{\sigma^2}{2}
\bigg{)}^{-1}\vert\vert u\vert\vert^2.
\end{align*}
This proves the lemma.

\end{proof}

We then obtain the following corollary.

\begin{cor}\label{mag_main_est_1_a}
Given $u \in C_c^{\infty}(M)$ and
$\sigma > 0$ sufficiently small, we have
\begin{align*}
\vert \langle h^{1/2}\Delta_A^2u, \nabla_{(dh^{1/2})^\#}\Delta_Au\rangle\vert
\leq
\sigma C_1(\sigma)\vert\vert  h^{1/2}\Delta_A^2u\vert\vert^2 + 
C_2(\sigma)\vert\vert u\vert\vert^2
\end{align*}
where $C_1(\sigma)$ and $C_2(\sigma)$ are constants, depending 
on $\sigma$, such that
$\lim_{\sigma \rightarrow 0} C_1(\sigma) < \infty$ and 
$\lim_{\sigma \rightarrow 0}C_2(\sigma) = 0$.
\end{cor}

\begin{proof}
Using lemma \ref{mag_main_est_1}, we can write
\begin{align*}
\vert \langle h^{1/2}\Delta_A^2u, \nabla_{(dh^{1/2})^{\#}}\Delta_Au\rangle
\vert &\leq 
\bigg{[}
\frac{\sigma}{2} + \bigg{(}\frac{\sigma}{2}\bigg{)}
\bigg{(}
\frac{1}{2(1-\frac{5}{4}\sigma)}\bigg{)} + \\
&\hspace{0.5cm}
\bigg{(}\frac{\sigma}{2}\bigg{)}
\bigg{(}\frac{1}{1-\frac{5}{4}\sigma} \bigg{)}
\bigg{(}\frac{1}{2} + \frac{5}{4}\sigma\bigg{)}
\bigg{(}
\frac{\sqrt{\sigma}}{2} + \frac{\sigma^2}{2(1-\frac{7}{4}\sigma)}
 \bigg{)} \\
&\hspace{0.5cm} 
\bigg{(}
1 - \sigma^2\bigg{(}\frac{1}{1-\frac{7}{4}\sigma} \bigg{)}
\bigg{(}\frac{1}{2} + \frac{7\sigma}{4} \bigg{)} - \frac{\sigma^2}{2}
\bigg{)}^{-1}\bigg{]}
\vert\vert h^{1/2}
\Delta_A^2u\vert\vert^2 \\
&\hspace{0.5cm}+
\bigg{(}\frac{\sigma}{2}\bigg{)}
\bigg{(}\frac{1}{1-\frac{5}{4}\sigma} \bigg{)}
\bigg{(}\frac{1}{2} + \frac{5}{4}\sigma\bigg{)}
\bigg{(}\frac{1}{2\sqrt{\sigma}} + \frac{3}{2} \bigg{)} \\
&\hspace{0.5cm}
\bigg{(}
1 - \sigma^2\bigg{(}\frac{1}{1-\frac{7}{4}\sigma} \bigg{)}
\bigg{(}\frac{1}{2} + \frac{7\sigma}{4} \bigg{)} - \frac{\sigma^2}{2}
\bigg{)}^{-1}\vert\vert u\vert\vert^2.
\end{align*}
Writing the coefficient of $\vert\vert h^{1/2}\Delta_A^2u\vert\vert^2$
as
\begin{align*}
&\frac{\sigma}{2} + \bigg{(}\frac{\sigma}{2}\bigg{)}
\bigg{(}
\frac{1}{2(1-\frac{5}{4}\sigma)}\bigg{)} \\
&+ 
\bigg{(}\frac{\sigma}{2}\bigg{)}
\bigg{(}\frac{1}{1-\frac{5}{4}\sigma} \bigg{)}
\bigg{(}\frac{1}{2} + \frac{5}{4}\sigma\bigg{)}
\bigg{(}
\frac{\sqrt{\sigma}}{2} + \frac{\sigma^2}{2(1-\frac{7}{4}\sigma)}
 \bigg{)} 
\bigg{(}
1 - \sigma^2\bigg{(}\frac{1}{1-\frac{7}{4}\sigma} \bigg{)}
\bigg{(}\frac{1}{2} + \frac{7\sigma}{4} \bigg{)} - \frac{\sigma^2}{2}
\bigg{)}^{-1} \\
&=
\sigma
\bigg{[}
\frac{1}{2} + \bigg{(}\frac{1}{2}\bigg{)}
\bigg{(}
\frac{1}{2(1-\frac{5}{4}\sigma)}\bigg{)} \\
&+ 
\bigg{(}\frac{1}{2}\bigg{)}
\bigg{(}\frac{1}{1-\frac{5}{4}\sigma} \bigg{)}
\bigg{(}\frac{1}{2} + \frac{5}{4}\sigma\bigg{)}
\bigg{(}
\frac{\sqrt{\sigma}}{2} + \frac{\sigma^2}{2(1-\frac{7}{4}\sigma)}
 \bigg{)} 
\bigg{(}
1 - \sigma^2\bigg{(}\frac{1}{1-\frac{7}{4}\sigma} \bigg{)}
\bigg{(}\frac{1}{2} + \frac{7\sigma}{4} \bigg{)} - \frac{\sigma^2}{2}
\bigg{)}^{-1}
\bigg{]}
\end{align*}
and defining
\begin{align*}
C_1(\sigma) &=
\frac{1}{2} + \bigg{(}\frac{1}{2}\bigg{)}
\bigg{(}
\frac{1}{2(1-\frac{5}{4}\sigma)}\bigg{)} \\
&\hspace{0.5cm}+ 
\bigg{(}\frac{1}{2}\bigg{)}
\bigg{(}\frac{1}{1-\frac{5}{4}\sigma} \bigg{)}
\bigg{(}\frac{1}{2} + \frac{5}{4}\sigma\bigg{)}
\bigg{(}
\frac{\sqrt{\sigma}}{2} + \frac{\sigma^2}{2(1-\frac{7}{4}\sigma)}
 \bigg{)} 
\bigg{(}
1 - \sigma^2\bigg{(}\frac{1}{1-\frac{7}{4}\sigma} \bigg{)}
\bigg{(}\frac{1}{2} + \frac{7\sigma}{4} \bigg{)} - \frac{\sigma^2}{2}
\bigg{)}^{-1}.
\end{align*}
It is easy to see that $\lim_{\sigma \rightarrow 0}C_1(\sigma) = 1$.

We then define $C_2(\sigma)$ as coefficient of $\vert\vert u\vert\vert^2$ 
\begin{align*}
C_2(\sigma) &=
\bigg{(}\frac{\sigma}{2}\bigg{)}
\bigg{(}\frac{1}{1-\frac{5}{4}\sigma} \bigg{)}
\bigg{(}\frac{1}{2} + \frac{5}{4}\sigma\bigg{)}
\bigg{(}\frac{1}{2\sqrt{\sigma}} + \frac{3}{2} \bigg{)} \\
&\hspace{0.5cm}
\bigg{(}
1 - \sigma^2\bigg{(}\frac{1}{1-\frac{7}{4}\sigma} \bigg{)}
\bigg{(}\frac{1}{2} + \frac{7\sigma}{4} \bigg{)} - \frac{\sigma^2}{2}
\bigg{)}^{-1} \\
&=
\bigg{(}\frac{\sqrt{\sigma}}{2}\bigg{)}
\bigg{(}\frac{1}{1-\frac{5}{4}\sigma} \bigg{)}
\bigg{(}\frac{1}{2} + \frac{5}{4}\sigma\bigg{)}
\bigg{(}\frac{1}{2} + \frac{3\sqrt{\sigma}}{2} \bigg{)} \\
&\hspace{0.5cm}
\bigg{(}
1 - \sigma^2\bigg{(}\frac{1}{1-\frac{7}{4}\sigma} \bigg{)}
\bigg{(}\frac{1}{2} + \frac{7\sigma}{4} \bigg{)} - \frac{\sigma^2}{2}
\bigg{)}^{-1}
\end{align*}
from which it is easy to see that 
$\lim_{\sigma \rightarrow 0}C_2(\sigma) = 0$.
\end{proof}

\begin{lem}\label{mag_main_est_2}
Given $u \in C_c^{\infty}(M)$ and
$\sigma > 0$ sufficiently small, we have the following estimate
\begin{align*}
\vert\vert \Delta_A\nabla_{({dh^{1/2}}^{\#})}u\vert\vert^2 &\leq
\bigg{[}
\frac{\sigma^2}{1-\frac{5}{4}\sigma} + (2\sigma^2)
\bigg{(}\frac{1+\frac{5}{4}\sigma}{2(1-\frac{5}{4}\sigma)} \bigg{)}
\bigg{(}\frac{\sqrt{\sigma}}{2} + \frac{\sigma^2}{2(1-\frac{5}{4}\sigma)} 
\bigg{)} \\
&\hspace{0.5cm}
\bigg{(}
1 - \sigma^2\bigg{(}\frac{1}{1-\frac{5}{4}\sigma} \bigg{)}
\bigg{(}1 + \frac{5\sigma}{4} \bigg{)} - \frac{\sigma^2}{2}
\bigg{)}^{-1} \\
&\hspace{0.5cm}
+
\bigg{(}\frac{C^2\sigma^2}{1-\frac{5}{4}\sigma} \bigg{)}
\bigg{(}\frac{\sqrt{\sigma}}{2} + \frac{\sigma^2}{2(1-\frac{7}{4}\sigma)} 
\bigg{)} \\
&\hspace{0.5cm}
\bigg{(}
1 - \sigma^2\bigg{(}\frac{1}{1-\frac{7}{4}\sigma} \bigg{)}
\bigg{(}\frac{1}{2} + \frac{7\sigma}{4} \bigg{)} - \frac{\sigma^2}{2}
\bigg{)}^{-1}
\bigg{]}\vert\vert h^{1/2}\Delta_A^2u\vert\vert^2 \\
&\hspace{0.5cm} +
\bigg{[}
(2\sigma^2)
\bigg{(}\frac{1+\frac{5}{2}\sigma}{2(1-\frac{5}{4}\sigma)} \bigg{)}
\bigg{(}\frac{1}{\sqrt{\sigma}} + \frac{3}{2} \bigg{)}
\bigg{(}
1 - \sigma^2\bigg{(}\frac{1}{1-\frac{5}{4}\sigma} \bigg{)}
\bigg{(}1 + \frac{5\sigma}{4} \bigg{)} - \frac{\sigma^2}{2}
\bigg{)}^{-1} \\
&\hspace{0.5cm} +
\bigg{(}\frac{C^2\sigma^2}{1-\frac{5}{4}\sigma} \bigg{)}
\bigg{(}\frac{1}{2\sqrt{\sigma}} + \frac{3}{2} \bigg{)}
\bigg{(}
1 - \sigma^2\bigg{(}\frac{1}{1-\frac{7}{4}\sigma} \bigg{)}
\bigg{(}\frac{1}{2} + \frac{7\sigma}{4} \bigg{)} - \frac{\sigma^2}{2}
\bigg{)}^{-1} \\
&\hspace{0.5cm} +
\bigg{(}\frac{(C^2\sigma^2)(1+ \frac{5\sigma}{2})}{(1-\frac{5}{4}\sigma)} 
\bigg{)} +
2C^2\sigma^2
\bigg{]}\vert\vert u\vert\vert^2.
\end{align*}
\end{lem}

\begin{proof}
Using the commutation formula \eqref{connection_3}, we can write
\begin{align*}
\vert\vert \Delta_A\nabla_{(dh^{1/2})^{\#}}u\vert\vert^2 &=
\vert\vert \nabla_{(dh^{1/2})^{\#}}\Delta_Au + 
\nabla_{(dh^{1/2})^{\#}}(Rm+F)*u 
+ (Rm+F)*\nabla_{(dh^{1/2})^{\#}}u \vert\vert^2 \\
&\leq 2\vert\vert \nabla_{(dh^{1/2})^{\#}}\Delta_Au\vert\vert^2 + 
2\vert\vert \nabla_{(dh^{1/2})^{\#}}(Rm+F)*u \vert\vert^2 + 
2\vert\vert (Rm+F)*\nabla_{(dh^{1/2})^{\#}}u \vert\vert^2 \\
&\leq 2\sigma^2\vert\vert h^{5/4}\nabla\Delta_Au\vert\vert^2 + 
2C^2\sigma\vert\vert h^{5/4}u\vert\vert^2 + 
2C^2\sigma^2\vert\vert h^{5/4}\nabla u\vert\vert^2 \\
&\leq 
2\sigma^2\vert\vert h^{5/4}\nabla\Delta_Au\vert\vert^2 + 
2C^2\sigma\vert\vert u\vert\vert^2 + 
2C^2\sigma^2\vert\vert h^{5/4}\nabla u\vert\vert^2
\end{align*}
were to get the inequality on the third line we have used 
\eqref{star_notation_ex}, and the estimate
of $\vert dh^{1/2}\vert$. To get the last inequality, we have used the
fact that $h^{5/4} \leq 1$.

Applying proposition \ref{mag_covariant_est_1} to estimate the term
$\vert\vert h^{5/4}\nabla u\vert\vert^2$, and
proposition \ref{mag_covariant_est_2} to estimate the term
$\vert\vert h^{5/4}\nabla\Delta_Au\vert\vert^2$, proves the lemma.

\end{proof}

\begin{cor}\label{mag_main_est_2_a}
Given $u \in C_c^{\infty}(M)$ and
$\sigma > 0$ sufficiently small, we have 
\begin{equation*}
\vert\vert \Delta_A\nabla_{({dh^{1/2}}^{\#})}u\vert\vert^2 \leq
\sigma C_1(\sigma)\vert\vert h^{1/2}\Delta_A^2u\vert\vert^2 + 
C_2(\sigma)\vert\vert u\vert\vert^2
\end{equation*}
where $C_1(\sigma)$ and $C_2(\sigma)$ are constants depending on
$\sigma$, and such that
$\lim_{\sigma \rightarrow 0}C_1(\sigma) < \infty$ and
$\lim_{\sigma \rightarrow 0}C_2(\sigma) = 0$.
\end{cor}

\begin{proof}
By lemma \ref{mag_main_est_2} we have
\begin{align*}
\vert\vert \Delta_A\nabla_{({dh^{1/2}}^{\#})}u\vert\vert^2 &\leq
\bigg{[}
\frac{\sigma^2}{1-\frac{5}{4}\sigma} + (2\sigma^2)
\bigg{(}\frac{1+\frac{5}{4}\sigma}{2(1-\frac{5}{4}\sigma)} \bigg{)}
\bigg{(}\frac{\sqrt{\sigma}}{2} + \frac{\sigma^2}{2(1-\frac{5}{4}\sigma)} 
\bigg{)} \\
&\hspace{0.5cm}
\bigg{(}
1 - \sigma^2\bigg{(}\frac{1}{1-\frac{5}{4}\sigma} \bigg{)}
\bigg{(}1 + \frac{5\sigma}{4} \bigg{)} - \frac{\sigma^2}{2}
\bigg{)}^{-1} \\
&\hspace{0.5cm}
+
\bigg{(}\frac{C^2\sigma^2}{1-\frac{5}{4}\sigma} \bigg{)}
\bigg{(}\frac{\sqrt{\sigma}}{2} + \frac{\sigma^2}{2(1-\frac{7}{4}\sigma)} 
\bigg{)}
\bigg{(}
1 - \sigma^2\bigg{(}\frac{1}{1-\frac{7}{4}\sigma} \bigg{)}
\bigg{(}\frac{1}{2} + \frac{7\sigma}{4} \bigg{)} - \frac{\sigma^2}{2}
\bigg{)}^{-1}
\bigg{]}\vert\vert h^{1/2}\Delta_A^2u\vert\vert^2 \\
&\hspace{0.5cm} +
\bigg{[}
(2\sigma^2)
\bigg{(}\frac{1+\frac{5}{2}\sigma}{2(1-\frac{5}{4}\sigma)} \bigg{)}
\bigg{(}\frac{1}{\sqrt{\sigma}} + \frac{3}{2} \bigg{)}
\bigg{(}
1 - \sigma^2\bigg{(}\frac{1}{1-\frac{5}{4}\sigma} \bigg{)}
\bigg{(}1 + \frac{5\sigma}{4} \bigg{)} - \frac{\sigma^2}{2}
\bigg{)}^{-1} \\
&\hspace{0.5cm} +
\bigg{(}\frac{C^2\sigma^2}{1-\frac{5}{4}\sigma} \bigg{)}
\bigg{(}\frac{1}{2\sqrt{\sigma}} + \frac{3}{2} \bigg{)}
\bigg{(}
1 - \sigma^2\bigg{(}\frac{1}{1-\frac{7}{4}\sigma} \bigg{)}
\bigg{(}\frac{1}{2} + \frac{7\sigma}{4} \bigg{)} - \frac{\sigma^2}{2}
\bigg{)}^{-1} \\
&\hspace{0.5cm} +
\bigg{(}\frac{(C^2\sigma^2)(1+ \frac{5\sigma}{2})}{(1-\frac{5}{4}\sigma)} 
\bigg{)} +
2C^2\sigma^2
\bigg{]}\vert\vert u\vert\vert^2.
\end{align*}
We can write the coefficient of 
$\vert\vert h^{1/2}\Delta_A^2u\vert\vert^2$ as
\begin{align*}
&\frac{\sigma^2}{1-\frac{5}{4}\sigma} + (2\sigma^2)
\bigg{(}\frac{1+\frac{5}{4}\sigma}{2(1-\frac{5}{4}\sigma)} \bigg{)}
\bigg{(}\frac{\sqrt{\sigma}}{2} + \frac{\sigma^2}{2(1-\frac{5}{4}\sigma)} 
\bigg{)} 
\bigg{(}
1 - \sigma^2\bigg{(}\frac{1}{1-\frac{5}{4}\sigma} \bigg{)}
\bigg{(}1 + \frac{5\sigma}{4} \bigg{)} - \frac{\sigma^2}{2}
\bigg{)}^{-1} \\
&\hspace{0.5cm}
+
\bigg{(}\frac{C^2\sigma^2}{1-\frac{5}{4}\sigma} \bigg{)}
\bigg{(}\frac{\sqrt{\sigma}}{2} + \frac{\sigma^2}{2(1-\frac{7}{4}\sigma)} 
\bigg{)}
\bigg{(}
1 - \sigma^2\bigg{(}\frac{1}{1-\frac{7}{4}\sigma} \bigg{)}
\bigg{(}\frac{1}{2} + \frac{7\sigma}{4} \bigg{)} - \frac{\sigma^2}{2}
\bigg{)}^{-1} \\
&= \\
\sigma
&\bigg{[}
\frac{\sigma}{1-\frac{5}{4}\sigma} + (2\sigma)
\bigg{(}\frac{1+\frac{5}{4}\sigma}{2(1-\frac{5}{4}\sigma)} \bigg{)}
\bigg{(}\frac{\sqrt{\sigma}}{2} + \frac{\sigma^2}{2(1-\frac{5}{4}\sigma)} 
\bigg{)} 
\bigg{(}
1 - \sigma^2\bigg{(}\frac{1}{1-\frac{5}{4}\sigma} \bigg{)}
\bigg{(}1 + \frac{5\sigma}{4} \bigg{)} - \frac{\sigma^2}{2}
\bigg{)}^{-1} \\
&\hspace{0.5cm}
+
\bigg{(}\frac{C^2\sigma}{1-\frac{5}{4}\sigma} \bigg{)}
\bigg{(}\frac{\sqrt{\sigma}}{2} + \frac{\sigma^2}{2(1-\frac{7}{4}\sigma)} 
\bigg{)}
\bigg{(}
1 - \sigma^2\bigg{(}\frac{1}{1-\frac{7}{4}\sigma} \bigg{)}
\bigg{(}\frac{1}{2} + \frac{7\sigma}{4} \bigg{)} - \frac{\sigma^2}{2}
\bigg{)}^{-1}
\bigg{]}.
\end{align*}
Defining 
\begin{align*}
C_1(\sigma) &=
\frac{\sigma}{1-\frac{5}{4}\sigma} + (2\sigma)
\bigg{(}\frac{1+\frac{5}{4}\sigma}{2(1-\frac{5}{4}\sigma)} \bigg{)}
\bigg{(}\frac{\sqrt{\sigma}}{2} + \frac{\sigma^2}{2(1-\frac{5}{4}\sigma)} 
\bigg{)} 
\bigg{(}
1 - \sigma^2\bigg{(}\frac{1}{1-\frac{5}{4}\sigma} \bigg{)}
\bigg{(}1 + \frac{5\sigma}{4} \bigg{)} - \frac{\sigma^2}{2}
\bigg{)}^{-1} \\
&\hspace{0.5cm}
+
\bigg{(}\frac{C^2\sigma}{1-\frac{5}{4}\sigma} \bigg{)}
\bigg{(}\frac{\sqrt{\sigma}}{2} + \frac{\sigma^2}{2(1-\frac{7}{4}\sigma)} 
\bigg{)}
\bigg{(}
1 - \sigma^2\bigg{(}\frac{1}{1-\frac{7}{4}\sigma} \bigg{)}
\bigg{(}\frac{1}{2} + \frac{7\sigma}{4} \bigg{)} - \frac{\sigma^2}{2}
\bigg{)}^{-1}
\end{align*}
it is easy to see that $\lim_{\sigma\rightarrow 0}C_1(\sigma) = 0$.

We then define $C_2(\sigma)$ to be the coefficient of 
$\vert\vert u\vert\vert^2$
\begin{align*}
C_2(\sigma) &= 
(2\sigma^2)
\bigg{(}\frac{1+\frac{5}{2}\sigma}{2(1-\frac{5}{4}\sigma)} \bigg{)}
\bigg{(}\frac{1}{\sqrt{\sigma}} + \frac{3}{2} \bigg{)}
\bigg{(}
1 - \sigma^2\bigg{(}\frac{1}{1-\frac{5}{4}\sigma} \bigg{)}
\bigg{(}1 + \frac{5\sigma}{4} \bigg{)} - \frac{\sigma^2}{2}
\bigg{)}^{-1} \\
&\hspace{0.5cm} +
\bigg{(}\frac{C^2\sigma^2}{1-\frac{5}{4}\sigma} \bigg{)}
\bigg{(}\frac{1}{2\sqrt{\sigma}} + \frac{3}{2} \bigg{)}
\bigg{(}
1 - \sigma^2\bigg{(}\frac{1}{1-\frac{7}{4}\sigma} \bigg{)}
\bigg{(}\frac{1}{2} + \frac{7\sigma}{4} \bigg{)} - \frac{\sigma^2}{2}
\bigg{)}^{-1} \\
&\hspace{0.5cm} +
\bigg{(}\frac{(C^2\sigma^2)(1+ \frac{5\sigma}{2})}{(1-\frac{5}{4}\sigma)} 
\bigg{)} +
2C^2\sigma^2.
\end{align*}
It is easy to see that $\lim_{\sigma \rightarrow 0}C_2(\sigma) = 0$.

\end{proof}

\section{Proof of theorem \ref{main_theorem_2}}
\label{proof_main_theorem_2}

In this section we prove theorem \ref{main_theorem_2}. A key 
component of the proof is to obtain an estimate analogous to
proposition \ref{bilaplace_est_6}. 

In order to obtain such an estimate, we will proceed along the same lines
as we did for proposition \ref{bilaplace_est_6}. 

Let $u \in C_c^{\infty}(M)$,
we start by 
expanding out the following inner product.

\begin{align}
Re\langle \Delta_A^4u, hu\rangle &= Re\langle \Delta_A^2u,\Delta_A^2(hu)\rangle \nonumber \\ 
&= Re\langle \Delta_A^2u, h\Delta_A^2u + 2(\Delta_M h)(\Delta_Au) - 
2\nabla^A_{(dh)^\#}\Delta_Au + u\Delta_M^2h - 2\nabla_{(d\Delta_M h)^\#}^Au 
-2\Delta_A\nabla_{(dh)^\#}^Au\rangle \nonumber \\
&= 
Re\langle h^{1/2}\Delta_A^2u, h^{1/2}\Delta_A^2u\rangle + 
2Re\langle \Delta_A^2u, (\Delta_M h)(\Delta_Au)\rangle 
-2Re\langle \Delta_A^2u, \nabla_{(dh)^\#}^A\Delta_Au\rangle \nonumber \\
& \hspace{0.5cm}+ 
Re\langle \Delta_A^2u, u\Delta_M^2h\rangle - 
2Re\langle \Delta_A^2u, \nabla_{(dh)^\#}^Au\rangle -2Re\langle \Delta_A^2u, 
\Delta_A\nabla^A_{(dh)^\#}u\rangle \nonumber \\
&=
Re\langle \Delta_A^2(h^{1/2}u), \Delta_A^2(h^{1/2}u)\rangle \label{mag_eq_1}
\\
& \hspace{0.5cm} +
 4Re\langle h^{1/2}\Delta_A^2u, \nabla_{(dh^{1/2})^\#}\Delta_Au\rangle 
\nonumber \\
& \hspace{0.5cm}+
4Re\langle h^{1/2}\Delta_A^2u, \Delta_A\nabla_{dh^{1/2}}u\rangle  
+
4Re\langle h^{1/2}\Delta_A^2u, \nabla_{(d\Delta_M h^{1/2})^\#}u\rangle \nonumber \\
& \hspace{0.5cm}+
4Re\langle h^{1/2}\Delta_A^2u, (\Delta_M h^{1/2})(\Delta_Au)\rangle - 
2Re\langle h^{1/2}\Delta_A^2u, u\Delta_M^2h^{1/2}\rangle \nonumber \\
&\hspace{0.5cm}
-4\langle \nabla_{(dh^{1/2})^\#}\Delta_Au, \nabla_{(dh^{1/2})^\#}\Delta_Au\rangle 
-8Re\langle \nabla_{(dh^{1/2})^\#}\Delta_Au, \Delta_A\nabla_{(dh^{1/2})^\#}u\rangle \nonumber \\
&\hspace{0.5cm}
-8Re\langle \nabla_{(dh^{1/2})^\#}\Delta_Au, \nabla_{(d\Delta_M h^{1/2})^\#}u\rangle 
-8Re\langle \nabla_{(dh^{1/2})^\#}\Delta_Au, (\Delta_M h^{1/2})(\Delta_Au)\rangle \nonumber \\
&\hspace{0.5cm}
+ 4Re\langle \nabla_{(dh^{1/2})^\#}\Delta_Au, u\Delta_M^2h^{1/2}\rangle 
- 4\langle \Delta_A\nabla_{(dh^{1/2})^\#}u, \Delta_A\nabla_{(dh^{1/2})^\#}u\rangle \nonumber \\
&\hspace{0.5cm}
-8Re\langle \Delta_A\nabla_{(dh^{1/2})^\#}u, \nabla_{(d\Delta_M h^{1/2})^\#}u\rangle 
-8 Re\langle \Delta_A\nabla_{(dh^{1/2})^\#}u, (\Delta_M h^{1/2})(\Delta_Au)\rangle \nonumber \\
&\hspace{0.5cm}
+ 4Re\langle \Delta_A\nabla_{(dh^{1/2})^\#}u, u\Delta_M^2h^{1/2}\rangle 
-4\langle \nabla_{(d\Delta_M h^{1/2})^\#}u, \nabla_{(d\Delta_M h^{1/2})^\#}u\rangle \nonumber \\
&\hspace{0.5cm}
-8Re\langle \nabla_{(d\Delta_M h^{1/2})^\#}u, (\Delta_M h^{1/2})(\Delta_Au)\rangle 
+ 4Re\langle \nabla_{(d\Delta_M h^{1/2})^\#}u, u\Delta_M^2h^{1/2}\rangle 
\nonumber \\
&\hspace{0.5cm}
-4 \langle (\Delta_M h^{1/2})(\Delta_Au), (\Delta_M h^{1/2})(\Delta_Au)\rangle 
+ 4Re\langle (\Delta_M h^{1/2})(\Delta_Au), u\Delta_M^2h^{1/2}\rangle 
\nonumber \\
&\hspace{0.5cm}
-\langle u\Delta_M^2h^{1/2}, u\Delta_M^2h^{1/2}\rangle 
+2Re\langle \Delta_A^2u, (\Delta_M h)(\Delta_Au)\rangle \nonumber \\
&\hspace{0.5cm}
-2Re\langle \Delta_A^2u, \nabla_{(dh)^\#}\Delta_Au\rangle 
+ Re\langle \Delta_A^2u, u\Delta_M^2 h\rangle \nonumber \\
&\hspace{0.5cm}
-2Re\langle \Delta_A^2u, \nabla_{(d\Delta_M h)^\#}u\rangle 
-2Re\langle \Delta_A^2u, \Delta_A\nabla_{(dh)^\#}u\rangle \nonumber
\end{align}

As in the proof of proposition \ref{bilaplace_est_6}, we need to
estimate the terms on the right hand side.

Each of the last twenty five terms on the right hand side of
equation \eqref{mag_eq_1} are given as the real part of a complex number.
We now claim that the absolute value of these complex numbers
can all be bounded above by
\begin{equation*}
\sigma C_1(\sigma)\vert\vert h^{1/2}\Delta_A^2u\vert\vert^2 + 
C_2(\sigma)\vert\vert u\vert\vert^2
\end{equation*}
where $0 < \sigma < 1$, and
$C_1(\sigma)$, $C_2(\sigma) > 0$ are constants, depending on 
$\sigma$, such that
$\lim_{\sigma \rightarrow 0}C_1(\sigma) <\infty$ and 
$\lim_{\sigma \rightarrow 0}C_2(\sigma) = 0$.

We start by looking at the five terms:

\begin{align}
&4Re\langle h^{1/2}\Delta_A^2u, \nabla_{(dh^{1/2})^\#}\Delta_Au\rangle 
\label{ineq_1} \\
&
4Re\langle h^{1/2}\Delta_A^2u, \Delta_A\nabla_{dh^{1/2}}u\rangle  
\label{ineq_2} \\
&
4Re\langle h^{1/2}\Delta_A^2u, \nabla_{(d\Delta_M h^{1/2})^\#}u\rangle 
\label{ineq_3} \\
& 
4Re\langle h^{1/2}\Delta_A^2u, (\Delta_M h^{1/2})(\Delta_Au)\rangle 
\label{ineq_4} \\
&
2Re\langle h^{1/2}\Delta_A^2u, u\Delta_M^2h^{1/2}\rangle  
\label{ineq_5}
\end{align}

Lemma \ref{mag_main_est_1} and corollary \ref{mag_main_est_1_a}
show how to estimate the term \eqref{ineq_1}. A similar proof
shows how to bound each of the terms,
\eqref{ineq_2}, \eqref{ineq_3}, \eqref{ineq_4}, 
\eqref{ineq_5}, above by 
$\sigma C_1(\sigma)\vert\vert h^{1/2}\Delta_A^2u\vert\vert^2 + 
C_2(\sigma)\vert\vert u\vert\vert^2$, with 
$\lim_{\sigma \rightarrow 0}C_1(\sigma) <\infty$ and 
$\lim_{\sigma \rightarrow 0}C_2(\sigma) = 0$.
We note that in estimating \eqref{ineq_2} one needs to make 
use of corollary \ref{connection_3}.

The next step is to estimate the five terms:
\begin{align}
&-4\langle \nabla_{(dh^{1/2})^\#}\Delta_Au, \nabla_{(dh^{1/2})^\#}
\Delta_Au\rangle \label{ineq_6}\\
& 
-4
\langle \Delta_A\nabla_{(dh^{1/2})^\#}u, \Delta_A\nabla_{(dh^{1/2})^\#}u
\rangle \label{ineq_7}\\
&
-4\langle \nabla_{(d\Delta_M h^{1/2})^\#}u, \nabla_{(d\Delta_M h^{1/2})^\#}u
\rangle \label{ineq_8}\\
&
-4 \langle (\Delta_M h^{1/2})(\Delta_Au), (\Delta_M h^{1/2})(\Delta_Au)
\rangle \label{ineq_9}\\
&
-\langle u\Delta_M^2h^{1/2}, u\Delta_M^2h^{1/2}\rangle  
\label{ineq_10}
\end{align}

Lemma \ref{mag_main_est_2} and corollary \ref{mag_main_est_2_a} show
how to estimate \eqref{ineq_7}. A similar proof shows that
we can estimate \eqref{ineq_6}, \eqref{ineq_8}, 
\eqref{ineq_9}, \eqref{ineq_10} in the same way.

We then look at the ten terms:

\begin{align}
& -8Re\langle \nabla_{(dh^{1/2})^\#}\Delta_Au, \Delta_A
\nabla_{(dh^{1/2})^\#}u\rangle \label{ineq_11} \\
&
-8Re\langle \nabla_{(dh^{1/2})^\#}\Delta_Au, \nabla_{(d\Delta_M h^{1/2})^
\#}u\rangle \label{ineq_12}\\
&
-8Re\langle \nabla_{(dh^{1/2})^\#}\Delta_Au, (\Delta_M h^{1/2})(\Delta_Au)
\rangle \label{ineq_13}\\
&
4Re\langle \nabla_{(dh^{1/2})^\#}\Delta_Au, u\Delta_M^2h^{1/2}\rangle 
\label{ineq_14}\\
&
-8Re\langle \Delta_A\nabla_{(dh^{1/2})^\#}u, \nabla_{(d\Delta_M h^{1/2})^
\#}u\rangle \label{ineq_15}\\
&
-8 Re\langle \Delta_A\nabla_{(dh^{1/2})^\#}u, (\Delta_M h^{1/2})
(\Delta_Au)\rangle \label{ineq_16}\\
&
4Re\langle \Delta_A\nabla_{(dh^{1/2})^\#}u, u\Delta_M^2h^{1/2}\rangle 
\label{ineq_17}\\
&
-8Re\langle \nabla_{(d\Delta_M h^{1/2})^\#}u, (\Delta_M h^{1/2})(\Delta_Au)\rangle \label{ineq_18}\\
&
+ 4Re\langle \nabla_{(d\Delta_M h^{1/2})^\#}u, u\Delta_M^2h^{1/2}\rangle 
\label{ineq_19}\\
&
4Re\langle (\Delta_M h^{1/2})(\Delta_Au), u\Delta_M^2h^{1/2}\rangle. 
\label{ineq_20}
\end{align}

We can bound the absolute vale of each of these ten terms by applying
Cauchy-Schwarz and Young's inequality, followed by the estimates we
obtained for \eqref{ineq_6}, \eqref{ineq_7}, 
\eqref{ineq_8}, \eqref{ineq_9}, \eqref{ineq_10}.

The last five terms we have to estimate are:
\begin{align}
& 2Re\langle \Delta_A^2u, (\Delta_M h)(\Delta_Au)\rangle 
\label{ineq_21}\\
& -2Re\langle \Delta_A^2u, \nabla_{(dh)^{\#}}\Delta_Au\rangle 
\label{ineq_22}\\
&Re\langle \Delta_A^2u, u\Delta_M^2h\rangle 
\label{ineq_23}\\
& 2Re\langle \Delta_A^2u, \nabla_{(d\Delta_M h)^{\#}}u\rangle 
\label{ineq_24}\\
& 2Re\langle \Delta_A^2u, \Delta_A\nabla_{(dh)^{\#}}u\rangle
\label{ineq_25}
\end{align}

This is done in an analogous way to \eqref{ineq_1}, \eqref{ineq_2}, 
\eqref{ineq_3}, \eqref{ineq_4}, \eqref{ineq_5}, and using our 
assumptions on $h$. For example, looking at \eqref{ineq_21} we have
\begin{align*}
\vert\langle \Delta_A^2u, (\Delta_M h)(\Delta_Au)\rangle \vert &\leq
\vert \langle \vert\Delta_A^2u\vert, 
\vert\Delta_M h\vert\vert \Delta_A\vert\rangle\vert \\
&\leq
\vert \langle \vert\Delta_A^2u\vert, 
\sigma h^{3/4} \vert \Delta_A\vert\rangle\vert \\
&\leq \vert \langle h^{1/2}\vert\Delta_A^2u\vert, 
\sigma h^{1/4} \vert \Delta_A\vert\rangle\vert \\
&\leq
\frac{\sigma}{2}\vert\vert h^{1/2}\Delta_A^2u\vert\vert^2 + 
\frac{\sigma}{2}\vert\vert h^{1/4}\Delta_Au\vert\vert^2
\end{align*}
where to get the second inequality we have used the assumption that
$\vert \Delta_M h\vert \leq \sigma h^{3/4}$, and to get the last 
inequality we have applied Cauchy-Schwarz and Young's inequality.
We then apply lemma \ref{mag_laplace_est_2} to estimate the second
term, in the above last inequality, and thereby obtain the required
estimate.

We remark that each of the above estimates involved obtaining a 
$\sigma > 0$ sufficiently small. As there were only finitely many
estimates involved, what we see is that we can find a 
$0 < \sigma_0 < 1$ such that if $0 < \sigma \leq \sigma_0 < 1$
then the above estimates for the right hand side of 
\eqref{mag_eq_1} hold.

This allows us to obtain an analogue of proposition \ref{bilaplace_est_6}.
The proof of which is exactly similar to the proof of
proposition \ref{bilaplace_est_6}.

\begin{prop}\label{mag_bilaplace_est_main}
There exists $0 < \sigma_0 < 1$, such that given 
$u \in C_c^{\infty}(M)$ and $0 < \sigma \leq \sigma_0$,
we have
\begin{equation*}
\vert\vert h^{1/2}\Delta_A^2u\vert\vert^2 \leq 
\bigg{(}
\frac{1}{1- \sigma C_1(\sigma)}\bigg{)}
\vert\vert\Delta_A^2(h^{1/2}u)\vert\vert^2 + 
\bigg{(}
\frac{C_2(\sigma)}{1-\sigma C_1(\sigma)}\bigg{)} 
\vert\vert u\vert\vert^2
\end{equation*}
where $C_1(\sigma)$ and $C_2(\sigma)$ are constants, depending 
on $\sigma$, such that
$\lim_{\sigma \rightarrow 0} C_1(\sigma) < \infty$ and 
$\lim_{\sigma \rightarrow 0}C_2(\sigma) = 0$.
\end{prop}

We are now in a position to prove theorem \ref{main_theorem_2}. 
We will follow the strategy of Milatovic. We start with the following
abstract result of Sohr, see \cite{sohr}.

\begin{lem}\label{sohr}
Assume that $T_1$ and $T_2$ are non-negative self-adjoint 
operators on a Hilbert space $\mathcal{H}$ with inner product
$\langle \cdot, \cdot\rangle$ and norm $\vert\vert\cdot\vert\vert$.
Assume that $Dom(T_1)\cap Dom(T_2)$ is dense in $\mathcal{H}$.
Additionally, assume that there exists constants $\tau > 0$ and
$0 \leq \xi < 1$ such that
\begin{equation}
\langle T_2u, u\rangle \leq \tau\vert\vert u\vert\vert^2, 
\text{ for all } u \in Dom(T_2) \label{sohr_1},
\end{equation}
\begin{equation}
Re\langle T_1u, T_2^{-1}u\rangle + \xi\vert\vert u\vert\vert^2 \geq 0, 
\text{ for all } u \in Dom(T_1) \label{sohr_2}.
\end{equation}

Then, the operator $T_1 + T_2$ is self-adjoint on 
$Dom(T_1)\cap Dom(T_2)$.
\end{lem}

\begin{proof}[\textbf{Proof of theorem} \ref{main_theorem_2}]

Define $T_1u := (\Delta_A^4)u$ with domain $Dom(T_1) = \mathcal{D}_1$
and $T_2u := wu$ with $Dom(T_2) = \mathcal{D}_2$. Under the completeness 
assumption on $M$, we have that the operator 
$(\Delta_A^4)\vert_{C_c^{\infty}(M)}$ is essentially self-adjoint, 
see \cite{chernoff} or \cite{cordes}. Furthermore, we have that its closure
$\overline{(\Delta_A^4)\vert_{C_c^{\infty}(M)}} = T_1$. 
The operator $T_1$ is non-negative, and it is clear that 
the operator $T_2$ is self-adjoint. Moreover, since 
$w \geq 1$ we have 
$\langle T_2u, u\rangle \geq \vert\vert u\vert\vert^2$, for all
$u \in Dom(T_2)$. This establishes \eqref{sohr_1} with $\tau = 1$. 
By lemma \ref{sohr}, the self-adjointness of 
$(\Delta_A^4) + w$ on $\mathcal{D}_1\cap\mathcal{D}_2$ will follow if
we can establish \eqref{sohr_2}. As $T_1 = 
\overline{(\Delta_A^4)\vert_{C_c^{\infty}(M)}}$, it is enough to show
that there exists a constant $0 \leq \xi < 1$ such that
\begin{equation}
Re\langle \Delta_A^4u, hu\rangle + \xi\vert\vert u\vert\vert^2 \geq 0,
\text{ for all } u \in C_c^{\infty}(M). \label{sohr_3}
\end{equation}

Combining proposition \ref{mag_bilaplace_est_main}, 
with the fact that we can estimate the 
absolute value of the complex number, corresponding to each
of the last twenty five terms on the right hand side of equation
\eqref{mag_eq_1}, by 
$\sigma C_1(\sigma)\vert\vert h^{1/2}\Delta_A^2u\vert\vert^2 + 
C_2(\sigma)\vert\vert u\vert\vert^2$. 
We see that we can bound the
absolute value of said complex number by 
\begin{equation*}
\bigg{(}
\frac{\sigma C_1(\sigma)}{1-\sigma C_1(\sigma)}\bigg{)}
\vert\vert\Delta_A^2(h^{1/2}u)\vert\vert^2 + 
\bigg{(}
\frac{\sigma C_1(\sigma)C_2(\sigma)}{1-\sigma C_1(\sigma)} + C_2(\sigma)
\bigg{)}\vert\vert u\vert\vert^2.
\end{equation*}

Using this bound, and the fact that given a complex number $z \in \C$ with
$\vert z \vert \leq \lambda$ then $-\lambda \leq Re(z) \leq \lambda$.
We can then go back to equation \eqref{mag_eq_1} and obtain the estimate
\begin{equation*}
Re\langle \Delta_A^4u, hu\rangle \geq 
\bigg{(}
1 - \frac{\sigma C_1(\sigma)}{1-\sigma C_1(\sigma)}
\bigg{)}\vert\vert\Delta_A^2(h^{1/2}u)\vert\vert^2  -
\bigg{(}
\frac{\sigma C_1(\sigma)C_2(\sigma)}{1-\sigma C_1(\sigma)} + C_2(\sigma)
\bigg{)}\vert\vert u\vert\vert^2
\end{equation*}
where $C_1(\sigma)$ and $C_2(\sigma)$ are constants, depending 
on $\sigma$, such that
$\lim_{\sigma \rightarrow 0} C_1(\sigma) < \infty$ and 
$\lim_{\sigma \rightarrow 0}C_2(\sigma) = 0$.

Choosing $\sigma > 0$ small enough, we can make it so that
$1 - \frac{\sigma C_1(\sigma)}{1-\sigma C_1(\sigma)} > 0$, and so that
$0 \leq 
\frac{\sigma C_1(\sigma)C_2(\sigma)}{1-\sigma C_1(\sigma)} + C_2(\sigma)
< 1$.

This establishes \eqref{sohr_3}, and finishes the proof.

\end{proof}

\section{Application to the separation problem: Proof of corollary  \ref{separation_cor}}\label{App_separation}

Combining theorems \ref{main_theorem_1} and \ref{main_theorem_2}, 
we can give the proof of corollary \ref{separation_cor}.

\begin{proof}[\textbf{Proof of corollary \ref{separation_cor}}]
Let $T_{max}u := (\Delta_A)^4u + wu$ with domain $D_{max}$ as in 
\eqref{max_domain}. Replacing $V$ by $w$ in theorem
\ref{main_theorem_1}, and noting that since $w \geq 1$ we 
have that the conditions of theorem \ref{main_theorem_1}
are satisfied. This implies $T_{max}$ is self-adjoint. Appealing to theorem
\ref{main_theorem_2}, we have that 
$\Delta_A^4 + w$ is self-adjoint on $\mathcal{D}_1 \cap \mathcal{D}_2$.
Therefore, $D_{max} = \mathcal{D}_1 \cap \mathcal{D}_2$, which
implies $\Delta_A^4 + w$ is separated.
\end{proof}

\section{Concluding remarks: The case of higher order even perturbations}\label{conclusion}

In this section, we want to outline why the techniques of this paper, which
are modelled on the approach of Milatovic in \cite{milatovic}, cannot be made to
work for higher order perturbations of the form
$\Delta^{2n} + V$ for $n > 2$.

The setup will be as before. $(M, g)$ will be a fixed Riemannian manifold
admitting bounded geometry. $(E, h)$ will be a Hermitian vector bundle
over $M$, with Hermitian metric $h$. We will assume $\nabla$ is a 
metric connection on $E$. Furthermore, so as to make the discussion 
easier,
we will assume the triple
$(E, h, \nabla)$ admits bounded geometry.

Before we discuss the key issues with such a generalisation, we would like
to mention that the results of this paper do go through for perturbations
of the form $\Delta^6 + V$. It is when the power of the Laplacian exceeds
six that the techniques seem to break down.

A key ingredient in proving theorem \ref{main_theorem_1} is the 
estimate obtained in lemma \ref{main_lemma_1}. If we look back at the
proof of lemma \ref{main_lemma_1}, we see that it needed the estimate
obtained in proposition \ref{bilaplace_est_6}.
In fact, the estimates
carried out in sections \ref{sec_covariant_derivatives} and 
\ref{sec_laplacian_derivatives}, were done so primarily for the reason
of obtaining proposition \ref{bilaplace_est_6}.

If we are going
to prove a version of theorem \ref{main_theorem_1} for higher 
order perturbations of the form $\Delta^{2n} + V$, where $V$ is a 
potential satisfying the same assumptions as in theorem 
\ref{main_theorem_1}, we are going to need an analogous estimate
as the one stated in proposition \ref{bilaplace_est_6}. That is, 
given $u \in W^{2n,2}_{loc}(E)\cap L^2(E)$ and $\epsilon > 0$ sufficiently
small,
we want an estimate of
the form
\begin{equation*}
\vert\vert\chi_{\epsilon}^{nk}\Delta^n u\vert\vert^2 \leq
\bigg{(}\frac{1}{\big{(}1-\epsilon C_1(\epsilon)\big{)}}\bigg{)}
\vert\vert\Delta^n(\chi_{\epsilon}^{nk}u)\vert\vert^2 +
\frac{C_2(\epsilon, \epsilon_1)}{\big{(}1-\epsilon C_1(\epsilon)\big{)}}
\vert\vert u\vert\vert^2
\end{equation*}
where 
$C_1(\epsilon)$ and $C_2(\epsilon)$ are constants depending on 
$\epsilon$ such that
$\lim_{\epsilon\rightarrow 0}C_1(\epsilon) < \infty$ and
$\lim_{\epsilon\rightarrow 0}C_2(\epsilon) < \infty$.

In general, the techniques of this paper cannot be used to obtain such
an estimate, when $n > 3$. To make this discussion concrete and 
illustrative let us see why this is the case with the operator
$\Delta^8 + V$.

We want to obtain an estimate of the form
\begin{equation*}
\vert\vert\chi_{\epsilon}^{4k}\Delta^4 u\vert\vert^2 \leq
\bigg{(}\frac{1}{\big{(}1-\epsilon C_1(\epsilon)\big{)}}\bigg{)}
\vert\vert\Delta^4(\chi_{\epsilon}^{4k}u)\vert\vert^2 +
\frac{C_2(\epsilon, \epsilon_1)}{\big{(}1-\epsilon C_1(\epsilon)\big{)}}
\vert\vert u\vert\vert^2
\end{equation*}

Looking back at the proof of proposition \ref{bilaplace_est_6}, we see 
that we must look at the term
$\langle \Delta^4(\chi_{\epsilon}^{4k}u), \Delta^4(\chi_{\epsilon}^{4k}u)
\rangle$.

We can write
\begin{align*}
\Delta^4(\chi_{\epsilon}^{4k}u) &= \Delta^2\Delta^2(\chi_{\epsilon}^{4k}u)
\\
&=\Delta^2\bigg{(}
\chi_{\epsilon}^{4k}\Delta^2u - 
2\nabla_{(d\chi_{\epsilon}^{4k})^{\#}}\Delta u + 
2\Delta_M(\chi_{\epsilon}^{4k})\Delta u - 2\Delta
\nabla_{(d\chi_{\epsilon}^{4k})^{\#}}u - 
2\nabla_{(d\Delta_M\chi_{\epsilon}^{4k})^{\#}}u + 
(\Delta_M^2\chi_{\epsilon}^{4k})u
\bigg{)} \\
&=
\Delta^2(\chi_{\epsilon}^{4k}\Delta^2u) - 
2\Delta^2\nabla_{(d\chi_{\epsilon}^{4k})^{\#}}\Delta u + 
2\Delta^2(\Delta_M(\chi_{\epsilon}^{4k})\Delta u) - 2
\Delta^3\nabla_{(d\chi_{\epsilon}^{4k})^{\#}}u - 2\Delta^2
\nabla_{(d\Delta_M\chi_{\epsilon}^{4k})^{\#}}u \\
&\hspace{0.5cm}+ 
\Delta^2((\Delta^2_M\chi_{\epsilon}^{4k})u)
\end{align*}

We then substitute this into 
$\langle \Delta^4(\chi_{\epsilon}^{4k}u), \Delta^4(\chi_{\epsilon}^{4k}u)
\rangle$. This gives a number of inner products, and we want to bound
the absolute value of each such inner product above by
\begin{equation*}
\bigg{(}\frac{1}{\big{(}1-\epsilon C_1(\epsilon)\big{)}}\bigg{)}
\vert\vert\Delta^4(\chi_{\epsilon}^{4k}u)\vert\vert^2 +
\frac{C_2(\epsilon, \epsilon_1)}{\big{(}1-\epsilon C_1(\epsilon)\big{)}}
\vert\vert u\vert\vert^2.
\end{equation*}

The problem is that some of the inner product terms that come out
contain the term $\Delta^3\nabla_{(d\chi_{\epsilon}^{4k})^{\#}}u$. 
For example, if we do the above substitution we get the term
$\vert\vert \Delta^3\nabla_{(d\chi_{\epsilon}^{4k})^{\#}}u\vert\vert^2$.
If we follow the approach in the proof of proposition 
\ref{bilaplace_est_6}. We see that the way to estimate such a term is
to use the commutation formula, see section \ref{commutation_formulae}.

For this situation, we will then need to use the following formula
\begin{equation*}
\Delta^3\nabla_{(d\chi_{\epsilon}^{4k})^{\#}} u = 
\nabla_{(d\chi_{\epsilon}^{4k})^{\#}} \Delta^3 u - 
\sum_{j=0}^5 \bigg{(}
(\nabla_M^{(5-j)}Rm + \nabla^{(5-j)} F)*\nabla^ju\bigg{)}
_{(d\chi_{\epsilon}^{4k})^{\#}}.
\end{equation*}

In the above formula, we see that there is a term of the form
$\bigg{(}(Rm + F) * \nabla^5u\bigg{)}_{(d\chi_{\epsilon}^{4k})^{\#}}$.
Therefore, we will need to estimate the term
$\bigg{\vert}\bigg{\vert} 
\bigg{(}(Rm + F) * \nabla^5u\bigg{)}_{(d\chi_{\epsilon}^{4k})^{\#}}
\bigg{\vert}\bigg{\vert}^2$ above by 
\begin{equation*}
\bigg{(}\frac{1}{\big{(}1-\epsilon C_1(\epsilon)\big{)}}\bigg{)}
\vert\vert\Delta^4(\chi_{\epsilon}^{4k}u)\vert\vert^2 +
\frac{C_2(\epsilon, \epsilon_1)}{\big{(}1-\epsilon C_1(\epsilon)\big{)}}
\vert\vert u\vert\vert^2.
\end{equation*}

As 
$\bigg{(}(Rm + F) * \nabla^5u\bigg{)}_{(d\chi_{\epsilon}^{4k})^{\#}}$
contains a fifth order covariant derivative of $u$, we see that this
is not possible. The best we could hope for is to try and
obtain an estimate that bounds
$\bigg{\vert}\bigg{\vert} 
\bigg{(}(Rm + F) * \nabla^5u\bigg{)}_{(d\chi_{\epsilon}^{4k})^{\#}}
\bigg{\vert}\bigg{\vert}^2$
above by a term involving 
$\vert\vert\Delta^5(\chi_{\epsilon}^{4k}u)\vert\vert^2$. 
Unfortunately, such an estimate is not enough to obtain an analogue
of theorem \ref{main_theorem_1} for the operator 
$\Delta^8 + V$.

In general, for the operator $\Delta^{2n} + V$. We find that
applying the commutation formula, see section 
\ref{commutation_formulae}, to certain terms will give 
covariant derivatives of order $n < j < 2n$. Unfortunately, such terms
can only be bounded above by terms containing $\Delta^j$, and
as $j > n$ this means the techniques we used to prove theorem 
\ref{main_theorem_1} will not go through.

\end{document}